\documentclass[reqno, english,11pt]{smfart}

\usepackage{amsfonts, amsthm, amsmath, amssymb}
\usepackage[all]{xy} 

\usepackage{a4wide}

\RequirePackage{mathrsfs} \let\mathcal\mathscr

\numberwithin{equation}{section}

\newtheorem{theorem}{Theorem}[section]
\newtheorem{lemma}[theorem]{Lemma}

\theoremstyle{definition}
\newtheorem*{ack}{Acknowledgements}
\newtheorem{rem}[theorem]{Remark}

\renewcommand{\phi}{\varphi}

\newcommand{\PP}{\mathbb{P}}
\renewcommand{\AA}{\mathbb{A}}

\newcommand{\GG}{\mathbb{G}}
\newcommand{\ZZ}{\mathbb{Z}}
\newcommand{\ZZp}{\mathbb{Z}_{\mathrm{prim}}}
\newcommand{\NN}{\mathbb{N}}
\newcommand{\QQ}{\mathbb{Q}}
\newcommand{\RR}{\mathbb{R}}

\newcommand{\cA}{\mathcal{A}}
\newcommand{\cB}{\mathcal{B}}

\newcommand{\cD}{\mathcal{D}}
\newcommand{\cM}{\mathcal{M}}

\renewcommand{\leq}{\leqslant}

\renewcommand{\geq}{\geqslant}

\renewcommand{\bar}{\overline}

\newcommand{\ma}{\mathbf}

\newcommand{\q}{\mathbf{q}}

\renewcommand{\a}{\mathbf{a}}
\renewcommand{\k}{\mathbf{k}}
\newcommand*\bell{\ensuremath{\boldsymbol\ell}}

\newcommand{\ve}{\varepsilon}
\newcommand{\e}{\mathbf e}

\newcommand{\bnu}{\boldsymbol{\nu}}

\newcommand{\bve}{\boldsymbol{\varepsilon}}
\newcommand{\bde}{\boldsymbol{\delta}}

\DeclareMathOperator{\sign}{\sigma}
\DeclareMathOperator{\Spec}{Spec}

\DeclareMathOperator{\Br}{Br}
\DeclareMathOperator{\diag}{diag}

\DeclareMathOperator{\Mod}{mod} 
\renewcommand{\bmod}[1]{\,(\Mod{#1})}
\renewcommand{\rho}{\varrho}

\newcommand{\Nloc}{N_{\mathrm{loc}}}

\newcommand{\NBr}{N_{\mathrm{Br}}}

\newcommand{\Stot}{S_{\mathrm{tot}}}
\newcommand{\Tot}{T_{\mathrm{tot}}}
\newcommand{\Toti}{T_{\mathrm{tot}}^{(1)}}
\newcommand{\Totii}{T_{\mathrm{tot}}^{(2)}}

\newcommand{\Sloc}{S_{\mathrm{loc}}}
\newcommand{\Sglob}{S_{\mathrm{glob}}}

\newcommand{\4}{\bmod{4}}
\newcommand{\8}{\bmod{8}}
\renewcommand{\2}{\bmod{2}}
\renewcommand{\=}{\equiv}

\renewcommand{\i}{\mathrm{i}}
\newcommand{\eps}{\epsilon}

\newcommand{\lan}{\left[}
\newcommand{\ran}{\right]}
\newcommand{\ocD}{\overline{\mathcal{D}}}

\begin{document}

\title[Ch\^atelet surfaces failing the Hasse principle]{Density  of Ch\^atelet surfaces failing\\ the Hasse principle}

\author{R.\ de la Bret\`eche}
\address{ 
Institut de Math\'ematiques de Jussieu\\
Universit\'e  Denis Diderot\\
Case Postale 7012\\
F-75251 Paris CEDEX 13\\ France}
\email{breteche@math.jussieu.fr}

\author{T.D.\ Browning}
\address{School of Mathematics\\
University of Bristol\\ Bristol\\ BS8 1TW\\ United Kingdom}
\email{t.d.browning@bristol.ac.uk}

\date{\today}

\begin{abstract} 
Ch\^atelet surfaces provide a rich source of geometrically rational surfaces which do not always satisfy the Hasse principle. 
Restricting attention to a special class of Ch\^atelet surfaces, we investigate the frequency that such counter-examples  arise over the rational numbers.
\end{abstract}


\maketitle
\setcounter{tocdepth}{1}
\tableofcontents

\section{Introduction}\label{s:intro}


A family of geometrically integral 
algebraic varieties defined over a number field  $K$ is said to satisfy the  Hasse principle 
if any variety in the family has a point in $K$ as soon as it has points in every completion of $K$.
Quadrics are among the first examples of families satisfying this property. 
In dimension $2$, 
Ch\^atelet surfaces constitute a family of varieties for which the Hasse principle
is known to fail in general.  
A Ch\^{a}telet surface over $K$ is a proper smooth model
of an affine surface
\begin{equation}\label{eq:1}
Y^2-eZ^2=f(T), 
\end{equation}
where 
$e\in K^*$ is not equal to a square in $K^*$ and 
$f\in K[T]$ is a separable polynomial of  degree 
$3$ or $4$.  For these surfaces it follows from work of  Colliot-Th\'el\`ene, Sansuc and Swinnerton-Dyer \cite{CT} that all failures of the Hasse principle are accounted for by the Brauer--Manin obstruction, a cohomological 
obstruction based on the Brauer group of the surface.
In particular it is known that the Hasse principle holds whenever $f$ is irreducible over $K$, which is the generic situation, or when $f$ has a linear factor over $K$.
In the remaining case, when $f$ is a product of two irreducible quadratic polynomials over 
$K$, counter-examples to the Hasse principle can arise.

In what follows  we shall only consider Ch\^atelet surfaces defined over $\QQ$. 
A great deal of  recent work has been directed at the quantitative arithmetic of rational varieties, 
the aim being to count $\QQ$-rational points of bounded height on the variety, assuming that it contains a 
Zariski dense set of $\QQ$-rational points.  In this paper we seek instead to
vary the varieties in a family and 
measure how often  
counter-examples to the Hasse principle emerge. 
In the setting of Ch\^atelet surfaces, the main result in \cite{CT} implies that a random surface of the form \eqref{eq:1} will satisfy the Hasse principle, 
and furthermore, we can only expect to find counter-examples when $f$ is a product of 
two quadratic polynomials over $\QQ$,  which are both irreducible over $\QQ$.

For simplicity we will henceforth restrict our attention to  Ch\^atelet surfaces given by \eqref{eq:1}, for which $e=-1$ and 
$f$ factorises as a product of two diagonal quadratic polynomials over $\QQ$.
Thus, for any $(a,b,c,d)\in \AA^4$, let
$X_{a,b,c,d}$ denote the  Ch\^atelet surface 
defined by the equation
\begin{equation}\label{eq:surface}
Y^2+Z^2=(aT^2+b)(cT^2+d),
\end{equation}
with 
\begin{equation}\label{eq:stipulate}
abcd\neq 0,\quad ad-bc\neq 0.
\end{equation} 
The fact that the Brauer--Manin obstruction is the only obstruction to the Hasse principle for this family goes back to earlier work of Colliot-Th\'el\`ene, Coray and Sansuc \cite{ct-c-s}, wherein it is shown that $\Br(X_{a,b,c,d})/\Br(\QQ)\cong \ZZ/2\ZZ$ for any $(a,b,c,d)\in \QQ^4$ such that \eqref{eq:stipulate} holds. (This work actually covers general Ch\^atelet surfaces of the form \eqref{eq:1} in which 
$f=q_1q_2$ for distinct quadratic polynomials $q_1,q_2$ that are irreducible over the base field.)
Furthermore, in 
\cite[Prop.~C]{ct-c-s}, a  $1$-parameter family of counter-examples to 
the Hasse principle over $\QQ$ is exhibited.   This family, which we will meet again in \S \ref{s:ct}, 
is given by $X_{1,1-k,-1,k}$ for any 
positive integer $k\equiv 3 \bmod{4}$. It generalises the particular case $k=3$ 
first discovered by Iskovskikh \cite{iskov}.
Our chief object is to give a finer quantitative treatment of this  circle of ideas. 
As we vary over 
rational coefficients, producing surfaces
$X_{a,b,c,d}$ that are defined over $\QQ$, we will investigate the 
 proportion of surfaces that
 \begin{itemize}
\item[---]
  have points everywhere locally,
\item[---]  have $\QQ$-rational points,
\item[---]   fail the Hasse principle.
\end{itemize}
One of the issues that we shall need to address is how best to parameterise the Ch\^atelet surfaces that are of interest to us.

Poonen and Voloch \cite{PV} have discussed 
similar questions in the 
setting of projective hypersurfaces $V\subset \PP^{n}$ of degree $d$, assuming that
 $d,n\geq 2$ and $(d,n)\neq (2,2)$.
Let $\Omega$ denote the set of valuations of $\QQ$ and let $\QQ_v$ be the completion of $\QQ$ at $v$,
 following the convention that $\QQ_\infty=\RR$.
In \cite[Thm.~3.6]{PV} it is shown that 
the proportion of $V$ that are everywhere locally soluble converges to an Euler product 
$c=\prod_{v\in \Omega} c_v>0$, where 
each $c_v$ is the proportion of $V$ for which $V(\QQ_v)\neq \emptyset$.
This is the exact analogue of our Theorem \ref{c:cor1} for hypersurfaces. 
They conjecture, furthermore, that when $d\leq n$  (which implies that the hypersurface is generically Fano)
the proportion of $V$ that have $\QQ$-rational points should converge to this constant $c$.
This  is shown (see \cite[Prop.~3.4]{PV})
to follow from the conjecture of Colliot-Th\'el\`ene that the Brauer--Manin obstruction to the Hasse principle is the only one for smooth, proper, geometrically integral varieties over $\QQ$ which are geometrically rationally connected. 
In fact, for non-singular hypersurfaces $V\subset \PP^n$, with $n\geq 4$, the 
natural map $\Br(\QQ)\rightarrow \Br(V)$ is an isomorphism. Thus the 
Brauer--Manin obstruction to the Hasse principle is generically empty when $n\geq 4$.

Bhargava \cite{msri} has 
recently undertaken an 
extensive investigation of hyperelliptic curves, where it is well-known that the Hasse principle can fail. 
Any hyperelliptic curve over $\QQ$ of genus $g$ can be embedded in weighted projective space $\PP(1,1,g+1)$, via an  equation of the form
$$
T^2=F(Y,Z),
$$
where $F\in \ZZ[Y,Z]$ is a separable  binary form of degree $2g+2.$
In this setting one gets different behaviour to that predicted for Fano hypersurfaces. It is shown in  \cite[Thm.~22]{msri} 
that for each $g\geq 1$, a positive  proportion of 
hyperelliptic curves of genus $g$ over $\QQ$, when ordered by height, 
fail the Hasse principle.

A form  of degree $d$ in $n+1$ variables has $N=\binom{d+n}{d}$  
possible coefficients. Poonen and Voloch 
use affine space $\AA^N$ to parameterise  hypersurfaces $V\subset \PP^n$ of degree $d$, associating to each vector in  $\ZZ^N$
a hypersurface $V$   defined over $\QQ$.
One disadvantage of this approach is that two 
different vectors may produce the same $V$. 
In our work it will be convenient to 
identify certain obvious choices of coefficients $(a,b,c,d)$ which lead to the same  Ch\^atelet surface $X_{a,b,c,d}$
in \eqref{eq:surface}.

Let $(a,b,c,d)\in \AA^4$ and $ (a',b',c',d')\in \AA^4$, 
with \eqref{eq:stipulate} holding for both sets of coefficients. 
The former gives 
rise to the  Ch\^atelet surface
 $X_{a,b,c,d}$ in 
 \eqref{eq:surface} and, using different variables, the latter gives rise to the 
Ch\^atelet surface 
 $X_{a',b',c',d'}$ given by 
$$
 Y'^2+Z'^2=(a'T'^2+b')(c'T'^2+d').
$$
We will identify 
these surfaces, writing $X_{a',b',c',d'}=X_{a,b,c,d}$, if  there is 
matrix $M\in \mathrm{GL}_2(\QQ)$ and a scalar $m\in \QQ^*$ such that 
\begin{equation}\label{eq:transform}
\binom{Y'}{Z'}
=M
\binom{Y}{Z},
 \quad T'=m T.
\end{equation}
In such a situation it is clear that the surfaces 
 $X_{a',b',c',d'}$ and $X_{a,b,c,d}$ really are identical.

We begin by noting that  $X_{\lambda a, \lambda b, \lambda c, \lambda d}=X_{a,b,c,d}$ for any $\lambda \in \QQ^*$, as can be seen by taking 
 $M=\diag(\lambda,\lambda)$ and $m=1$ in \eqref{eq:transform}.
Thus we are led  to consider the 
open set 
$$
U=\{[a,b,c,d]\in \PP^3: \mbox{\eqref{eq:stipulate} holds}\}.
$$
Similarly, we have 
$X_{\lambda^2 a, \mu^2 b, \lambda^2 c, \mu^2 d}=X_{a,b,c,d}$
for any $\lambda,\mu\in \QQ^*$, 
as  seen by taking 
$M=\diag(\mu^2,\mu^2)$ and $m=\mu/\lambda$ in \eqref{eq:transform}.
Next, we let 
 $N(\zeta)=\zeta_1^2+\zeta_2^2$ denote the norm of any $\zeta=\zeta_1+\i\zeta_2\in \QQ(\i)$. 
For $\xi,\eta\in \QQ(\i)^*$, we have 
$X_{N(\xi) a, N(\xi) b, N(\eta) c, N(\eta) d}=X_{a,b,c,d}$,
where the transformation is given by 
$$
M=\left(
\begin{matrix}
\xi_1\eta_1-\xi_2\eta_2 & -(\xi_1\eta_2+\xi_2\eta_1)\\
\xi_1\eta_2+\xi_2\eta_1 &
 \xi_1\eta_1-\xi_2\eta_2
\end{matrix}\right)
$$
and $m=1$ in \eqref{eq:transform}.
In conclusion, we have produced an action of 
$\GG_{m,\QQ}^2\times R_{\QQ(\i)/\QQ}(\GG_{m,\QQ(\i)})^2$ on $U$, where 
 $\GG_{m,k}$ denotes the 
linear  algebraic  group $\Spec K[T,T^{-1}]$ associated to a field $K$.
Building on  this,
we note that $X_{(a,b,c,d)^{\rho_1}}=X_{(a,b,c,d)^{\rho_2}}=X_{a,b,c,d}$, where
\begin{equation}\label{eq:rho-i}
\rho_1: (a,b,c,d)\mapsto (c,d,a,b), \quad
\rho_2: (a,b,c,d)\mapsto (b,a,d,c).
\end{equation}
This  leads to a further action on  $U$ 
by the $0$-dimensional algebraic group $(\ZZ/2\ZZ)^2$.
We therefore obtain a group action 
$G\times U\rightarrow U$, where $G$ is the algebraic $\QQ$-group
$$
G=
\GG_{m,\QQ}^2\times R_{\QQ(\i)/\QQ}(\GG_{m,\QQ(\i)})^2 \rtimes (\ZZ/2\ZZ)^2.
$$
Given  $u=[a,b,c,d]\in U$, the orbit of $u$ under $ G(\QQ)$ produces  the same surface $X_{a,b,c,d}$. 

In our analysis we  will restrict attention to elements of $\cM=U(\QQ)/G(\QQ)$.
We will be interested in three basic subsets,  given by 
\begin{align*}
\cM_{\mathrm{loc}}
&=\left\{ [a,b,c,d]\in \cM: X_{a,b,c,d}(\QQ_v) \neq \emptyset ~\forall v\in \Omega \right\},\\
\cM_{\mathrm{glob}}
&=\left\{ [a,b,c,d]\in \cM: X_{a,b,c,d}(\QQ) \neq \emptyset \right\},\\
\cM_{\mathrm{Br}}
&=\cM_{\mathrm{loc}}\setminus \cM_{\mathrm{glob}}.
\end{align*}
The elements of $\cM_{\mathrm{Br}}$ 
are precisely the surfaces for which there is a non-empty Brauer--Manin obstruction to the Hasse principle.
In \S \ref{s:prelim} we will choose representative coordinates in  
$\ZZp^4/\{\pm 1\}$ for the parameter space $\mathcal{M}$. In this way, to each point $u\in \mathcal{M}$ we can associate a height $H(u)$,  which will turn out to be 
 $\max\{|a|,|b|,|c|,|d|\}$ on $\ZZp^4/\{\pm 1\}$.
With this in mind we wish to study the cardinalities
\begin{equation}\label{eq:card}
N_{\mathrm{loc/glob/Br}} (P)=
\#\{u\in \cM_{\mathrm{loc/glob/Br}}: H(u)\leq P\},
\end{equation}
as $P\rightarrow \infty$. Our first result
is the following.

\begin{theorem}\label{t:M-local}
We have 
$N_{\mathrm{loc}}(P)= \tau_{\mathrm{loc}} P^4+O (P^{4-1/8+\ve}),$ for any $\ve>0$, 
where 
if $\tau_{\mathrm{loc},v}$ is the density of points in $U(\QQ_v)/G(\QQ_v)$ for which the associated surface has $\QQ_v$-points then 
$$
\tau_{\mathrm{loc}}=\prod_{v\in \Omega} \tau_{\mathrm{loc},v}.
$$
\end{theorem}

Here, as throughout our work, the implied constant is allowed to depend at most on the choice of parameter $\ve>0$. 
Theorem \ref{t:M-local} 
 will be established in \S
\ref{s:asymptotic-Nloc}, where 
an explicit description of the factors $\tau_{\mathrm{loc},v}$ will appear.
Once combined with an estimation of 
the total number $N(P)$ of elements in $\cM$ with height at most $P$, we will also establish the following result.

\begin{theorem}\label{c:cor1}
We have 
$$
\lim_{P\rightarrow \infty} \frac{N_{\mathrm{loc}}(P)}{N(P)}=
\frac{33257}{39168}
\prod_{p\=3 \4} 
\left( 1-\frac{6-\frac{9}{p^2}+\frac{4}{p^4}}{p^4 b_p}\right),
$$
where
\begin{equation}\label{eq:b_p}
b_p=
\left(1+\frac{1}{p}\right)^2\left(1+\frac{2}{p}+\frac{3}{p^2}+\frac{4}{p^3}-\frac{4}{p^5}\right).
\end{equation}
In particular 83.3\% of the elements of $\cM$ are everywhere locally soluble.
\end{theorem}

 In \S \ref{s:S-global} we will characterise when 
a  surface $X=X_{a,b,c,d}$ in $\cM_{\mathrm{loc}}$ actually belongs to $\cM_{\mathrm{glob}}$.
The Brauer group $\Br(X)$ 
will play
a fundamental r\^ole in this analysis, for which purpose we recall some basic facts here.
For any field $K\supseteq \QQ$, 
each element $\mathcal{A}\in \Br(X)$ gives rise to an evaluation map $\mathrm{ev}_{\mathcal{A}}:X(K)\rightarrow \Br(K)$.
Class field theory gives the  exact sequence
$$
0 \longrightarrow  \Br (\QQ) \longrightarrow \bigoplus_v \Br(\QQ_v)
\xrightarrow{\sum_{v}
\mathrm{inv}_v} \QQ/\ZZ \longrightarrow 0.
$$
For any 
$\mathcal{A}\in \Br(X)$ we then have the commutative diagram
			\[
				\xymatrix{
					 X(\QQ) \ar[r] \ar[d]^{\mathrm{ev}_{\mathcal{A}}}   &
					 X(\AA_\QQ) \ar[d]^{\mathrm{ev}_{\mathcal{A}}}
					 \\
					 \Br(\QQ) \ar[r] &\bigoplus_v \Br(\QQ_v)  \ar[r] &\QQ/\ZZ
				}
			\]
			where $\AA_\QQ$ denotes the ad\`eles
and $X(\AA_\QQ)=\prod_v X(\QQ_v)$.
Let $\Theta_\mathcal{A}: X(\AA_\QQ) \rightarrow \QQ/\ZZ$ denote the composed map. 
Then it follows that $X(\QQ)\subset \ker \Theta_{\mathcal{A}}$ for all $\mathcal{A}\in \Br(X)$.
We write 
$$
X(\AA_\QQ)^{\Br(X)}=\bigcap_{\mathcal{A}\in \Br(X)} \ker \Theta_{\mathcal{A}}.
$$
In our setting, $\Br(X)/\Br(\QQ)\cong \ZZ/2\ZZ$ 
has order $2$ and so 
$
\mathrm{inv}_v\left( \mathrm{ev}_{\mathcal{A}}(M_v) \right)
$
takes the values $0$ or $1/2$ in $\QQ/\ZZ.$
An obstruction to the Hasse principle arises if and only if $X(\AA_\QQ)^{\Br(X)}=\emptyset$, that is to say,  if and only if for any $(M_v)\in X(\AA_\QQ)$ there exists $\mathcal{A}\in \Br(X)$ such that 
\begin{equation}\label{eq:shanghai}
\sum_v  \mathrm{inv}_v\left( \mathrm{ev}_{\mathcal{A}}(M_v) \right)=\frac{1}{2}.
\end{equation}
Moreover, a generator for 
$\Br(X)/\Br(\QQ)$ 
is given by 
the quaternion algebra 
$(-1, aT^2+b)$ on 
 \eqref{eq:surface}. 
Fundamental to our work is the observation that 
\eqref{eq:shanghai} is impossible 
if there exists a valuation $v$ for which 
$
\mathrm{inv}_v\left( \mathrm{ev}_{\mathcal{A}}(M_v) \right)
$
takes both values $0$ and $1/2$, as a function of $M_v$.
Thus, in order to obtain counter-examples to the Hasse principle, it will be necessary for 
$\mathrm{inv}_v\left( \mathrm{ev}_{\mathcal{A}}(M_v) \right)$ to be  constant for every valuation $v$. 

In \S \ref{s:lower-upper} we shall build on  
\S \ref{s:S-global} to
establish the following asymptotic formula 
for the density of surfaces which provide counter-examples to the Hasse principle.

\begin{theorem}\label{thm2}
Let $\ve>0$. There exists $\tau_{\mathrm{Br}}>0$ such that 
$$
N_{\mathrm{Br}} (P)=
\frac{\tau_{\mathrm{Br}}P^4}{(\log P)^{1/4}} +O \left(\frac{P^4}{(\log P)^{3/4-\ve}} \right).
$$
\end{theorem}

It follows from Theorems \ref{t:M-local} and  \ref{thm2} that $N_{\mathrm{glob}}(P)$ satisfies the asymptotic behaviour predicted by Poonen and Voloch \cite{PV} for Fano hypersurfaces. Moreover, 
once coupled with Theorem \ref{c:cor1}, Theorem \ref{thm2} shows that 
83.3\% of the elements of $\cM$ are  soluble over $\QQ$.
It should be stressed that our argument provides a completely explicit algorithm for determining whether or not a given Ch\^atelet surface $X_{a,b,c,d}$ of the form \eqref{eq:surface} 
gives a counter-example to the Hasse principle, 
 without explicitly needing to work with elements of  
$\Br(X_{a,b,c,d})/\Br(\QQ)$.

\begin{ack}
The authors are grateful to Jean-Louis  Colliot-Th\'el\`ene and \'Etienne Fouvry for suggesting this problem to us, to Pierre Le Boudec and the anonymous referee for comments on an earlier draft and to Tim Dokchitser for help with verifying numerically the proof of Lemma \ref{lem:tau-2}.
While working on this paper the first author was supported by an {\em IUF Junior} and  {\em ANR project (PEPR)}, while the second author was supported by {\em ERC grant} 306457.
\end{ack}

\section{Preliminaries}\label{s:prelim}

We will need to choose representative coordinates for the parameter  space $\cM$  that we met in the introduction.  
Two sets that will feature  heavily in our work are  
\begin{equation}\label{eq:AB}
\cA=\{n\in \NN: \mu(n)^2=1\}, \quad 
\cB=\{n\in \cA: p\mid n \Rightarrow p\equiv 3 \4 \}, 
\end{equation}
where $\mu$ denotes the M\"obius function. 
Let $\ZZp^4$ denote the set of relatively prime $4$-tuples of integers. 
We define
\begin{equation}\label{eq:Stot}
S_{\mathrm{tot}}=
\left\{(a,b,c,d)\in \ZZp^4/\{\pm 1\}:  
\begin{array}{l}
ad-bc\neq 0, ~abcd\neq 0\\
\gcd(a,c), \gcd(b,d)\in \cA\\
\gcd(a,b), \gcd(c,d)\in \cB
\end{array}
\right\}.
\end{equation}
It is clear, in view of the  invariance under the action of the maps \eqref{eq:rho-i},
that we have a $1:4$ bijection between 
$\cM$ and our set of representatives $\Stot$.

When it comes to determining whether or not elements of $\Stot$ 
produce surfaces $X_{a,b,c,d}$ with $2$-adic points, a
tedious number of subcases arise, depending on  the residue classes modulo $4$  of the coefficients. 
In order to reduce the number of cases that need to be considered we note that
for any $(a,b,c,d)\in \Stot$, we have $(a,b)\in \{(1,1),(1,0),(0,1)\} \2$. 
But  $S_{\mathrm{tot}}$  is left invariant under the action of the map
$
\rho_2$ defined in \eqref{eq:rho-i}.  Hence we may write
$$
\Stot=\{(a,b,c,d)\in \Stot: 2\nmid a\} \sqcup   
\{(a,b,c,d)\in \Stot: 2\nmid a, ~2\mid b\}.
$$
The  elements of $\Stot$  have non-zero coordinates.
Since we identify $(a,b,c,d)$ with $-(a,b,c,d)$ in $\Stot$, it will  suffice to take representative coordinates
in which $a> 0$.  The sets in which we are interested  therefore take the shape
\begin{equation}\label{eq:Stot'}
\Stot^{(\iota)}=
\left\{(a,b,c,d)\in \ZZp^4:
\begin{array}{l}
a>0, ~bcd(ad-bc)\neq 0\\  
\gcd(a,c), \gcd(b,d)\in \cA\\
\gcd(a,b), \gcd(c,d)\in \cB\\
(a,2^\iota b)\=(1,0) \2
\end{array}
\right\},
\end{equation}
for  $\iota\in \{0,1\}$.  
Among the elements of these sets we will be interested primarily in those coefficients which give rise to Ch\^atelet surfaces $X_{a,b,c,d}$ which have points everywhere locally, or  points globally. Let us therefore define the sets
\begin{equation}\label{eq:Sloc+S}
\begin{split}
\Sloc^{(\iota)}&=
\left\{(a,b,c,d)\in \Stot^{(\iota)}: 
X_{a,b,c,d}(\QQ_v) \neq \emptyset ~\forall v\in \Omega \right\},\\
\Sglob^{(\iota)}&=
\left\{(a,b,c,d)\in \Stot^{(\iota)}: 
X_{a,b,c,d}(\QQ) \neq \emptyset \right\},
\end{split}
\end{equation}
for $\iota\in \{0,1\}$.

We proceed to introduce some further notation.
Given $(a,b,c,d)\in \Stot^{(\iota)}$, we 
define the binary quadratic forms
\begin{equation}\label{eq:Q12}
Q_1(U,V)=aU^2+bV^2, \quad
Q_2(U,V)=cU^2+dV^2.
\end{equation}
Let  $m=\gcd(a,b)$ and $n=\gcd(c,d)$. Then $m,n\in \cB$, with $\gcd(m,n)=1$. We 
henceforth write
\begin{equation}\label{eq:Q12'}
Q_1'(U,V)=a'U^2+b'V^2, \quad
Q_2'(U,V)=c'U^2+d'V^2,
\end{equation}
where 
$
(a,b)=m(a',b')$ and 
$(c,d)=n(c',d').$ 
In particular we have 
$
\gcd(a',b')=\gcd(c',d')=1
$
and 
\begin{equation}\label{eq:hcf}
\begin{split}
d'Q_1'(U,V)-b'Q_2'(U,V)&=\Delta'U^2,\\
-c'Q_1'(U,V)+a'Q_2'(U,V)&=\Delta'V^2,
\end{split}
\end{equation}
where 
 $\Delta'=a'd'-b'c'$.

Returning to the family of surfaces $X_{a,b,c,d}$ arising as proper smooth models of \eqref{eq:surface}, in gauging solubility over the completed field $\QQ_v$, for any $v \in \Omega$, 
it will suffice to work 
with a non-empty Zariski open subset $W$ of $X_{a,b,c,d}$. Indeed, according to 
\cite[Lemme~3.1.2]{ct-c-s}, the set $W(\QQ_v)$ has dense image in 
$X_{a,b,c,d}(\QQ_v)$ for the topology defined by the natural topology of $\QQ_v$.
In this way we see that it suffices to 
examine the local solubility of the affine equation
\begin{equation}\label{eq:chat}
0\neq Y^2+Z^2=Q_1(U,V)Q_2(U,V),
\end{equation}
where $Q_1,Q_2$ are as in \eqref{eq:Q12}.
 Our first order of business is a precise characterisation of when an element can locally be written as the sum of two squares.

Let $v\in \Omega$ and let $\Sigma_v$ denote the non-zero elements in $\QQ_v$ which can be written as a sum of two squares in $\QQ_v$.  
Firstly we note that $\Sigma_\infty = \RR_{>0}$.
When $v$ is a prime $p\equiv 1 \bmod{4}$ then $\Sigma_p=\QQ_p^*$.
Next when $v$ is a prime $p\equiv 3 \bmod{4}$ then 
$\Sigma_p$ is the set of $t\in \QQ_p^*$ for which 
the $p$-adic  
valuation $v_p(t)$ of $t$ is even. Finally when $v=2$ we
have $\Sigma_2=\mathcal{D}$, where
\begin{equation}\label{eq:D}
\mathcal{D}=\{2^n(1+4m): m\in \ZZ_2, ~n\in \ZZ \}.
\end{equation}
Frequent use will be made of this characterisation in \S \ref{s:S-local} and \S \ref{s:S-global}.

We will reserve the letter $p$ for denoting a prime number.  Our work will involve various standard arithmetic functions, including $\omega(n)=\sum_{p\mid n}1$, the generalised divisor function $\tau_k(n)=\sum_{n=d_1\cdots d_k}1$
and the Euler totient function $\phi$. 
In addition to this we will also meet the function
\begin{equation}\label{eq:phi*}
\phi_\delta^*(n)=\prod_{p\mid n}\left(1-\frac{1}{p^\delta}\right),
\end{equation}
for any $\delta>0$, so that $\phi_1^*(n)=\phi(n)/n$.
Note that $1/(\log \log n)\ll \phi_1^*(n)\leq 1$. Finally, let
\begin{equation}\label{eq:psi}
\psi(n)=
\prod_{\substack{p\mid n}} \left(1+\frac{1}{p}\right).
\end{equation}
We will typically abbreviate $\phi_1^*(n)$ by $\phi^*(n)$
and $\tau_2$ by $\tau$.

We will also require an estimate for the average order of the divisor function, as it ranges over the values of a binary linear form. A crucial aspect of the following result is the uniformity in the coefficients of the linear form that it enjoys.

\begin{lemma}\label{lem:divisor}
Let $\a\in \ZZ^2$ with $a_1a_2\neq 0$, 
let $V_1,V_2\geq 2$ and let $\ve>0$.
We have 
$$
\sum_{\substack{v_1\leq V_1\\ v_2\leq V_2}}\tau(|a_1v_1+a_2v_2|)\ll
\tau(\gcd(a_1,a_2))
 \psi(a_1a_2) V_1V_2\log (V_1V_2)+
\frac{\max_i|a_i|^\ve\max_i\{V_i\}^{1+\ve}}{\min_i\{V_i\}},
$$
where the implied constant depends at most on $\ve$.
\end{lemma}

\begin{proof}
Suppose without loss of generality that $V_1\geq V_2$.
The bound follows from previous work of the authors~\cite[Thm.~1]{nair}
when 
$V_2\gg  \max_i|a_i|^{\ve/4}V_1^{\ve/4}$. Alternatively it follows from the trivial bound $\tau(n)=O(n^\ve)$.
\end{proof}

We close this section with a modified result of Le Boudec \cite[Lemma 2]{LB}, concerning 
pairs of integers constrained to satisfy congruences modulo $q$ and modulo $r$, with $\gcd(q,r)=1$.
We are interested in the cardinality
\begin{equation}\label{eq:LB}
N(U,V;\a)=\#\left\{
(u,v)\in \NN^2: 
\begin{array}{l}
u\leq U, ~v\leq V\\
a_1u+a_2v\= 0\bmod{q}\\
\gcd(uv,q)=1\\
(u,v)\=(u_0,v_0)\bmod{r}
\end{array}
\right\},
\end{equation}
for suitable  $\a, u_0,v_0,U,V$. 
The following result allows us to approximate 
$N(U,V;\a)$ by the expected main term.

\begin{lemma}\label{lem:LB}
Let $\a\in \ZZ^2$ with $a_1a_2\neq 0$ and $\gcd(a_1a_2 r,q)=1$, let $
(u_0,v_0)\in \ZZ^2$ and 
let $
U,V\geq 1$.  We have 
$$
\left|
N(U,V;\a) - \frac{\phi(q)UV}{q^2r^2}
\right|\ll  \tau_3(q)\left(\frac{U+V}{q}+ (\log qr)^3\right) 
+E(\a),
$$
where
$$
E(\a)=\sum_{d\mid q}d
\sum_{\substack{0<|m|,|n| \leq qr/2\\ ma_2-na_1\=0\bmod{d}}}
\frac{1 }{|mn|}.
$$
\end{lemma}

\begin{proof}
As we have mentioned, this  result is based on 
work of Le Boudec \cite[Lemma 2]{LB}, which corresponds to $r=1$ and general closed sets $\mathcal{I},\mathcal{J}\subset \RR$ in place of $[1,U]$ and $[1,V].$
We have decided to include a sketch of the proof for completeness, the key idea being 
to break the $(u,v)$ into congruence classes modulo $qr$ and then use the orthogonality of additive characters $e_{qr}(\cdot)$ to detect these congruences. 

Since  $q$ and $r$ are coprime,  there exist
$\bar{q},\bar{r}\in \ZZ$  such that $q\bar{q}+r\bar{r}=1$. 
We therefore have
\begin{align*}
N(U,V;\a)
&=\sum_{\substack{\xi,\eta\bmod{q}\\
a_1\xi+a_2\eta\=0 \bmod{q}\\
\gcd(\xi\eta,q)=1
}}  
\#\left\{
(u,v)\in \NN^2: 
\begin{array}{l}
u\leq U, ~v\leq V\\
u\=q\bar{q}u_0+r\bar{r}\xi \bmod{qr}\\
v\=q\bar{q}v_0+r\bar{r}\eta \bmod{qr}
\end{array}
\right\}\\
&=
\frac{1}{q^2r^2}
\sum_{m,n\bmod{qr}}
S(m,n)
\sum_{u\leq U}
e_{qr}(-mu)
\sum_{v\leq V}
e_{qr}(-nv),
\end{align*}
where
$$
S(m,n)=
e_{r}(m\bar{q}u_0+n\bar{q}v_0)
\sum_{\substack{\xi,\eta\bmod{q}\\
a_1\xi+a_2\eta\=0 \bmod{q}\\
\gcd(\xi\eta,q)=1
}}  
e_{q}(m\bar{r}\xi+n\bar{r}\eta).
$$
Here, recalling that $\gcd(a_1a_2,q)=1$, we have  
$$
S(m,n)=
e_{r}(m\bar{q}u_0+n\bar{q}v_0)
c_q(ma_2-na_1),
$$
where $c_q(k)$ is the Ramanujan sum, satisfying $|c_q(k)|\leq \gcd(k,q)$.
In particular
$S(0,n)$ and $S(m,0)$ are seen to be independent of $a_1, a_2$.

We denote by $N(U,V)$ the contribution from $m=0$ or $n=0$. Then bounding the sums over $u,v$ using the standard bound for linear exponential sums, we see that 
\begin{equation}\label{eq:conducting}
N(U,V;\a)-N(U,V)\ll  
\sum_{0<|m|,|n| \leq qr/2}
\frac{|S(m,n)| }{|mn|}
\ll E(\a),
\end{equation}
in the notation of the lemma. 
Summing over $a_1 \in (\ZZ/q\ZZ)^*$ we deduce that 
\begin{align*}
\sum_{a_1}N(U,V;\a)- \phi(q)N(U,V)
&\ll  
\sum_{d\mid q}d
\sum_{\substack{0<|m|,|n| \leq qr/2}}
\frac{1 }{|mn|}
\sum_{\substack{a_1\bmod{q}\\
\gcd(a_1,q)=1\\
ma_2-na_1\=0\bmod{d}}}1
\\
&\ll  
\sum_{d\mid q}d
\sum_{h\mid d}\frac{1}{h^2}
\sum_{\substack{0<|m'|,|n'| \leq qr/(2h)\\
\gcd(m',n',d/h)=1
}}
\frac{qh }{d|m'n'|}
\\
&\ll 
q \tau_3(q)(\log qr)^2.
\end{align*}
Moreover, it is clear that 
\begin{align*}
\sum_{a_1}N(U,V;\a)
&=
\#\left\{
(u,v)\in \NN^2:
\begin{array}{l}
u\leq U, ~ v\leq V\\
\gcd(uv,q)=1\\
(u,v)\=(u_0,v_0)\bmod{r}
\end{array}
\right\}\\
&=\frac{\phi^*(q)^2UV}{r^2} +O\left(\tau(q)(U+V)\right).
\end{align*}
Note that $\phi^*(q)/\phi(q)=1/q$.  Combining this with \eqref{eq:conducting} we swiftly arrive at the statement of the lemma.
\end{proof}

\section{Local solubility constraints}\label{s:S-local}

As we have seen, in order to determine whether or not 
$X_{a,b,c,d}$ is soluble over  $\QQ_v$, for any $v \in \Omega$, 
it will suffice to work with the equation \eqref{eq:chat}.
Furthermore, when considering solubility over $\QQ_p$, 
it suffices to establish the existence of solutions $(y,z,u,v)$ for which $u,v\in \ZZ_p$ are coprime.
In order to summarise the local solubility criterion for $X_{a,b,c,d}$ it will be convenient to introduce a new symbol. For any odd prime $p$ and any $n\in \ZZ$ with $v_p(n)=\nu$,  we set 
\begin{equation}\label{eq:symbol}
\lan \frac{n}{p} \ran=
\begin{cases}
\left(\frac{n/p^{\nu}}{p}\right), & \mbox{if $2\mid \nu$},\\
1, & \mbox{if $2\nmid \nu$},
\end{cases}
\end{equation}
where $(\frac{\cdot}{p})$ is the usual Legendre symbol. This symbol takes values in $\{\pm 1\}$.
Finally, 
let $\sign(t)\in \{-,+\}$ denote the sign of any real number $t$, which we extend to vectors in the obvious way.
We may now record the following result,
in which $m=\gcd(a,b)$ and $n=\gcd(c,d)$.

\begin{lemma}\label{lem:X-local}
Let $(a,b,c,d)\in \Stot^{(\iota)}$. Then the following hold: 
\begin{itemize}
\item[(i)]
$X_{a,b,c,d}(\RR)\neq \emptyset$ if and only if $\sign(a,b,c,d)\neq (+,+,-,-)$;
\item[(ii)]
$X_{a,b,c,d}(\QQ_p)\neq \emptyset$ if $v_p(mn)=0$;
\item[(iii)]
if $v_p(mn)=1$ for $p\equiv 3 \bmod{4}$ then 
$X_{a,b,c,d}(\QQ_p)\neq \emptyset$ if and only if
\begin{itemize}
\item[---] 
$v_p(a')=v_p(c')$ or 
$v_p(b')=v_p(d')$ only when these $p$-adic orders are zero,
\item[---]
either $\lan \frac{-a'b'}{p}\ran=1$ or $\lan \frac{-c'd'}{p}\ran=1$;

\end{itemize}

\item[(iv)]
$X_{a,b,c,d}(\QQ_2)\neq \emptyset$ if and only if
$\exists$ coprime $u,v\in \ZZ_2$ such that $Q_1(u,v)Q_2(u,v)\in \cD$.
\end{itemize}
\end{lemma}

One recalls here that $m,n\in \cB$, with $\gcd(m,n)=1$. Thus  $v_p(mn)\leq 1$ for any prime $p$ and 
$X_{a,b,c,d}(\QQ_p)\neq \emptyset$ whenever $p\=1\4$.

\begin{proof}[Proof of Lemma \ref{lem:X-local}]
That the conditions are necessary is self-evident.   Turning to sufficiency,  we begin by establishing (i), noting that $a>0$ for any 
$(a,b,c,d)\in \Stot^{(\iota)}$.
The case in which $ac$ or $bd$ is positive is trivial, so we study only the case $ac<0$ and $bd<0$, with (i) being satisfied. In this case $bc>0$ and $ad>0$, whence
$$
Q_1Q_2(\sqrt{ad+bc},\sqrt{-2ac})=-ac(ad-bc)^2>0.
$$
Condition (iv) is a direct consequence of the criterion for solubility of $X_{a,b,c,d}$
over $\QQ_2$. Solubility over $\QQ_p$ is automatic for $p\equiv 1 \bmod{4}$.

It therefore remains to consider  
solubility over $\QQ_p$ when $p\equiv 3 \4$. For ease of notation we will  henceforth write $Q_i'$ for $Q_i'(u,v)$. 
We claim that there exist coprime $u,v\in \ZZ_p$ such that $v_p(Q_1'Q_2')=0$.
This will clearly suffice to show that $X_{a,b,c,d}(\QQ_p)$ is non-empty when 
$v_p(mn)=0$, as required for  condition (ii).
Recall that $\gcd(a',b')=1$ and $\gcd(c',d')=1$.   If $p\nmid a'c'$ then the claim is satisfied by taking $(u,v)=(1,0)$. Similarly, if $p\nmid b'd'$ then we may take $(u,v)=(0,1)$.
The remaining cases to consider are when $p\mid \gcd(b',c')$ with $p\nmid a'd'$, and 
lastly when
 $p\mid \gcd(a',d')$ with $p\nmid b'c'$. For both of these the choice $(u,v)=(1,1)$ is satisfactory, which thereby concludes the proof of the claim.
 
For the remainder of the proof we suppose that  
 $v_{p}(mn)=1$, with $p\equiv 3 \4$.  
 We wish to show that under the hypotheses of the lemma we can always find coprime $u,v\in \ZZ_p$ such that precisely one of $v_p(Q_1')$ or $v_p(Q_2')$ is odd. This will then imply that 
$v_p(Q_1(u,v)Q_2(u,v))$ is even as required for condition (iii).

 Suppose first that $p\nmid a'b'c'd'$. 
 Then by 
 hypothesis one of 
either $-a'b'$ or $-c'd'$ is a non-zero square modulo $p$.  There are two cases to consider.
Suppose first that 
$(\frac{-a'b'}{p})=1$ but $(\frac{-c'd'}{p})=-1$.
We can find coprime $u,v\in \ZZ_p$ such that 
$v_p(Q_1')=1$.  On the other hand 
$v_p(Q_2')$ is even since $-c'd'$ is not a square modulo $p$.
Next suppose that  $(\frac{-a'b'}{p})=(\frac{-c'd'}{p})=1$, with 
$v_p(\Delta')=\kappa\geq 0$. We may choose coprime $u,v\in \ZZ_p$ so that 
$v_p(Q_1')=\kappa+1$. 
Then if  $v_p(Q_2')=j$ one sees that 
$p^{\min\{j,\kappa+1\}}$ is a divisor of $\Delta'\gcd(u,v)^2$ by \eqref{eq:hcf}.  
It therefore follows from \eqref{eq:hcf} that $v_p(Q_2')= \kappa$, as required.

We now suppose that $p\mid a'b'$ but $p\nmid c'd'$.
If $(\frac{-c'd'}{p})=1$  then it is clear that we can find $u,v\in \ZZ_p$, with $p\nmid uv$, such that
$v_p(Q_2')$ is odd. Moreover, we will have $p\nmid Q_1'$ since
$p$ divides precisely one of  $a'$ or $b'$.
If 
$(\frac{-c'd'}{p})=-1$ 
then $v_p(Q_2')=0$ for any coprime $u,v\in \ZZ_p.$
Suppose $p\mid a' $, with $a'=p^\alpha a''$ and $p\nmid a''b'$. If $\alpha$ is odd then we can ensure that  $v_p(Q_1')$ odd by taking $(u,v)=(1,0)$.
Alternatively, if $\alpha$ is even then $(\frac{-a''b'}{p})=1$ by hypothesis. Hence we can force 
$v_p(Q_1')$ to be odd by considering solutions of the form 
$(u,p^{\alpha/2}v')$.
The case in which $p\nmid a'b'$ but $p\mid c'd'$ follows by symmetry.

It remains to consider the possibility  $p\mid a'b'$ and $p\mid c'd'$.
We will consider the cases $p\mid (a',d')$ or  $p\mid (a',c')$, the remaining cases following by symmetry.
Suppose first that $a'=p^\alpha a''$ and $d'=p^\delta d''$ with $p\nmid a''b'c'd''$.
If $\alpha$ (resp.  $\delta$) is odd then we are done on taking $(u,v)=(1,0)$ (resp. $(u,v)=(0,1)$). 
Suppose that $\alpha$ and $\delta$ are both even. Then our hypothesis ensures that one of the Legendre symbols $(\frac{-a''b'}{p})$ or $(\frac{-c'd''}{p})$  is equal to $1$. Supposing without loss of generality that it is the former, we 
 easily make $v_p(Q_1')$ odd with $v_p(Q_2')=0$.
Turning to the case  
$p\mid (a',c')$, we write $a'=p^\alpha a''$ and $c'=p^\gamma c''$ with $p\nmid a''b'c''d'$ and 
$\alpha,\gamma$ not both even.
If $\alpha,\gamma $ are of opposite parities then the situation is easy and we can proceed by taking $(u,v)=(1,0)$. 
Finally, if $\alpha,\gamma$ are both odd then by hypothesis they must be unequal. 
Suppose without loss of generality that $\alpha=2\alpha'+1<\gamma$. Then in fact $\gamma>\alpha+1$.
Taking $(u,v)=(1,p^{\alpha'+1})$ we easily deduce that $v_p(Q_1')=2\alpha'+1$ and 
$v_p(Q_2')=2\alpha'+2$, as required. 
This concludes the proof of the lemma. 
\end{proof}

Lemma \ref{lem:X-local} gives us a means of characterising the  elements
of   $\Sloc^{(\iota)}$, in the notation of \eqref{eq:Sloc+S}.
Before proceeding to the proof of Theorem \ref{t:M-local}, we will need a finer description of case (iv) in Lemma \ref{lem:X-local}.
This will depend intimately on the possible residue classes of $(a,b)$ modulo~$4$, with $a$ odd. The various constraints are obtained by taking 
$(A,B,C,D)=(a,b,c,d)$ in the second part of Lemmas~\ref{lem:1-1}--\ref{lem:3-2} 
in \S \ref{s:2adic},  where the description of  $\Totii$ 
exactly provides necessary and sufficient  
conditions for the non-nullity of 
$X_{a,b,c,d}(\QQ_2)$.

\section{Global solubility constraints}\label{s:S-global}

Throughout this section we will assume familiarity with the notation 
introduced in \S \ref{s:prelim}.
Among the set of $(a,b,c,d)\in \Sloc^{(\iota)}$ we require  a means of sifting for those $(a,b,c,d)$ such that  $X_{a,b,c,d}$ actually has points in  $\QQ$.  Our principal tool comes from the theory of descent, as formulated by 
Colliot-Th\'el\`ene, Coray and Sansuc \cite[Cor.~5.2]{ct-c-s}. This ensures that given any 
$(a,b,c,d)\in \Sloc^{(\iota)}$, we have 
$(a,b,c,d)\in \Sglob^{(\iota)}$  if and only if there exists $\e=(e_1,e_2)\in \QQ^2$, with $e_1e_2=1$, such that 
the smooth variety 
$$
\begin{cases}
0\neq Y_1^2+Z_1^2=e_1 (aT^2+b),\\
0\neq Y_2^2+Z_2^2=e_2 (cT^2+d),
\end{cases}
$$  
has points in $\QQ_w$ for every $w\in \Omega$.
This in turn is equivalent to the existence of $\e\in \QQ^2$, with 
$e_1e_2=mn$, such that the corresponding pair of equations with $a,b,c,d$ replaced by $a',b',c',d'$ is everywhere locally soluble. 
Note that, by  
\cite[Lemme~3.1.2]{ct-c-s}, we see that in checking local solubility for the pair of equations it suffices to replace $a'T^2+b'$ and $c'T^2+d'$ by $Q_1'(U,V)$ and $ Q_2'(U,V)$, respectively, as given by \eqref{eq:Q12'}.

We proceed to simplify the set of allowable $\e$ somewhat. 
On carrying out a suitable change of variables it suffices, 
without loss of generality, to consider the existence of $\e\in \QQ^2$ for which 
$$
v_p(e_i)
\in\begin{cases}
\{0\}, & \mbox{if $p\not\equiv 3\4$,}\\
\{0,-1,1\}, & \mbox{if $p\= 3\4$,}
\end{cases}
$$
for $i=1,2$.
Recall that $m,n\in \cB$, with $\gcd(m,n)=1$, in the notation of \eqref{eq:AB}.
Suppose that $e_1e_2=mn$ for $\e\in \QQ^2$ satisfying the above $p$-adic constraints. 
If $p\mid mn$ then we must  have $\{v_p(e_1),v_p(e_2)\}=\{0,1\}$. 
Alternatively, if
$p\nmid mn$,  then  $\{v_p(e_1),v_p(e_2)\}=\{0,0\}$ or 
$\{-1,1\}$.
Given such $\e$ let us put 
$f_i$ to be the product of primes $p\nmid mn$ for which $v_p(e_i)=1$.
Then we may assume that $e_1=e_1'f_1/f_2$ and 
$e_2=e_2'f_2/f_1$, where $e_1',e_2'$ are square-free integers comprised of primes $p\mid mn$.
On multiplying the first equation by $f_2^2$ and the second by $f_1^2$, and making a further change of variables, we see that it suffices to check local solubility with $\e=(e_1,e_2)$ replaced by $\e''=(e_1'f,e_2'f)$, where $f=f_1f_2$.
In particular, we have the relation $e_1''e_2''=e_1'e_2'f^2=mnf^2$.
We summarise our remarks in the following result.

\begin{lemma}\label{lem:1}
Let $(a,b,c,d)\in \Sloc^{(\iota)}$. Then $X_{a,b,c,d}(\QQ)\neq \emptyset$ if and only if there exists 
$\e\in \cB^2$ and $f\in \cB$, with  $\gcd(f,mn)=1$ and
\begin{equation}\label{eq:square}
e_1e_2=mnf^2,
\end{equation}
such that  the smooth variety 
$$
\AA^6\supset
W_{\e}: \quad 
\begin{cases}
0\neq Y_1^2+Z_1^2=e_1 Q_1'(U,V),\\
0\neq Y_2^2+Z_2^2= e_2 Q_2'(U,V),
\end{cases}
$$  
has solutions everywhere locally.
\end{lemma}

The thrust of Lemma \ref{lem:1} is that our task of determining elements of  $\Sglob^{(\iota)}$ in \eqref{eq:Sloc+S} has become a purely local problem. 
In fact, the varieties $W_{\e}$ are known to satisfy the Hasse principle
and any $\QQ$-point on $W_{\e}$ gives rise to a $\QQ$-point on $X_{a,b,c,d}$.
We shall henceforth refer to $W_{\e}$ as torsor equations. In reality, as explained in \cite[\S VII]{CT}, any universal torsor over $X_{a,b,c,d}$ is birational to $\PP^1\times C\times \PP^1\times W_{\e}$, for a smooth conic $C$ defined over $\QQ$.

We proceed to study the sets $W_{\e}(\QQ_w)$ for each $w\in \Omega$.  
When determining solubility over $\QQ_p$ it will suffice to consider the existence of solutions $(y_i,z_i,u,v)$ with $u,v\in \ZZ_p$ coprime.   
Recall from \S \ref{s:intro} our discussion of the Brauer group $\Br(X_{a,b,c,d})/\Br(\QQ)$, which is generated by a single element, given by the quaternion algebra $(-1,aT^2+b)$. 
Given any local point $M_w=(y,z,u,v)$  in $X(\QQ_w)$, we have 
\begin{equation}\label{eq:campus}
\mathrm{inv}_w\left( \mathrm{ev}_{\mathcal{A}}(M_w) \right)
=\begin{cases}
	0, & \mbox{if $Y^2+Z^2=Q_1(u,v)T^2$ soluble over $\QQ_w$,}\\
	1/2, & \mbox{otherwise}.
	\end{cases}
\end{equation}
According to our discussion of \eqref{eq:shanghai}, in order for counter-examples to the Hasse principle to arise, 
we will need 
 $
\mathrm{inv}_w\left( \mathrm{ev}_{\mathcal{A}}(M_w) \right)
$
to be locally constant for each $w\neq 2$,
in which case 
we will show that there
is a unique  torsor $W_\e$ which has $\QQ_w$-points.

\subsection{Solubility over $\QQ_p$ for odd $p$}

It will be convenient to retain our shorthand notation $Q_i'$ for $Q_i'(u,v)$ in what follows.
The 
case in which $p$ is an odd prime not appearing in the factorisation of $e_1e_2$
 is straightforward, as the following result shows.

\begin{lemma}\label{lem:p1}
Let $\e\in \cB^2$ and let $p>2$, with $p\nmid e_1e_2$.
Then  $W_{\e}(\QQ_p)\neq \emptyset$.
\end{lemma}

\begin{proof}
If $p\equiv 1 \4$ then the conclusion is trivial. Suppose now that $p\equiv 3 \4$. 
The second paragraph in the proof of Lemma \ref{lem:X-local} shows that there exists $u,v\in \ZZ_p$ such that 
$v_p(Q_1'Q_2')=0$. But then 
$v_p(e_iQ_i')=0$ for $i=1,2$, which concludes the proof.
\end{proof}

Since $e_1,e_2$ are square-free and entirely composed of primes congruent to  $3$ modulo $4$, Lemma \ref{lem:p1} completely handles the solubility question over $\QQ_p$ for $p\equiv 1\4$. Turning to  $p\equiv 3\4$, it remains to consider the possibility that $p\mid \gcd(mn,e_1 e_2)$, together with the possibility that 
$p\nmid mn$ and $p\mid \e$.
We may assume that $$
\e=(e_1'p^{k_1},e_2'p^{k_2}),
$$ 
with $\k=(k_1,k_2)\in \{0,1\}^2$ and $p\nmid e_1'e_2'$.  Then the condition \eqref{eq:square} demands that 
$$
k_1+k_2\equiv v_p(mn) \bmod{2}.
$$
In particular we have $\k\in \{(0,0), (1,1)\}$
 when $v_p(mn)=0$ and 
$\k\in \{(0,1), (1,0)\}$
 when $v_p(mn)=1$. Recall that  $\Delta'=a'd'-b'c'$. 
We proceed to establish the  following result.

\begin{lemma}\label{lem:p2}
Let $(a,b,c,d)\in \Sloc^{(\iota)}$ and  let $p\equiv 3 \4$ with $k=v_p(mn)\in \{0,1\}$. 
Let $\e'\in \cB^2$ with $p\nmid e_1'e_2'$.
Then one of the  following must hold:
\begin{itemize}
\item[(i)]
if $[\frac{-a'b'}{p}]+[\frac{-c'd'}{p}]\leq 0$ then there exists a unique
$\k\in \{0,1\}^2$ with $k_1+k_2\= k\2$ such that  $W_{(e_1'p^{k_1},e_2'p^{k_2})}(\QQ_p)\neq \emptyset$;

\item[(ii)]
if $[\frac{-a'b'}{p}]=[\frac{-c'd'}{p}]=1$
and $k=0$, then 
$W_{\e'}(\QQ_p)\neq \emptyset$; furthermore, 
$W_{p\e'}(\QQ_p)\neq \emptyset$ if and only if $p\mid \Delta'$;

\item[(iii)]
if $[\frac{-a'b'}{p}]=[\frac{-c'd'}{p}]=1$
and $k=1$,
then 
$W_{(e_1'p,e_2')}(\QQ_p)\neq \emptyset$
unless $v_p(a')\geq 2$ is even and $v_p(c')=1$, or $v_p(b')\geq 2$ is even and $v_p(d')=1$; likewise, 
$W_{(e_1',e_2'p)}(\QQ_p)\neq \emptyset$
unless $v_p(c')\geq 2$ is even and $v_p(a')=1$, or $v_p(d')\geq 2$ is even and $v_p(b')=1$.
\end{itemize}
In particular, whether or not   $W_{\bf e}(\QQ_p)$ is empty depends only on 
$v_p(e_1)$ and $v_p(e_2)$.
\end{lemma}

\begin{rem}\label{rem:1}
Let $p$ be an odd prime.   Lemmas \ref{lem:p1} and \ref{lem:p2} together show that
whenever 
$X_{a,b,c,d}(\QQ_p)\neq \emptyset$, there exists  $\k\in \{ 0,1\}^2$, with $k_1+k_2\=v_p(mn)\2$,  such that 
$W_{(p^{k_1},p^{k_2})}(\QQ_p)\neq \emptyset$. 
\end{rem}

\begin{proof}[Proof of Lemma \ref{lem:p2}]
Our arguments will have a similar flavour to the proof of Lemma~\ref{lem:X-local}.
Beginning with case (i) we suppose without loss of generality that $[\frac{-c'd'}{p}]=-1$.
Now any solution $u,v$ must have 
 $2\mid v_p(Q_2')$, whence the choice $\k=(1,1)$ (resp.\ $\k=(0,1)$) is not admissible when $k=0$ (resp.\ $k=1$).  
Finally we note that the case 
$\k=(0,0)$ 
is admissible when $k=0$ by  Lemma \ref{lem:p1}, and the case
 $\k=(1,0)$
is admissible when $k=1$ since 
 Lemma~\ref{lem:X-local} implies that 
 $[\frac{-a'b'}{p}]=1$.

Let us consider
case (ii), wherein we have $ k=0$ and 
$[\frac{-a'b'}{p}]=[\frac{-c'd'}{p}]=1$.
Let us set 
$$
\alpha=v_p(a'), \quad
\beta=v_p(b'), \quad
\gamma=v_p(c'), \quad
\delta=v_p(d'), \quad 
\kappa=v_p(\Delta'),
$$
with $\min \{\alpha,\gamma\}, \min\{\beta,\delta\}\leq 1$ and $\min\{\alpha,\beta\}=\min\{\gamma,\delta\}=0$.
It is clear that $\k=(0,0)$ leads to a torsor with $p$-adic points, by Lemma~\ref{lem:p1}.
The choice $
\k=(1,1)$ is  permissible if and only if there exist
coprime $u,v$ such that 
 $v_p(Q_i')\equiv 1 \bmod{2}$ for $i=1,2$.
If $\kappa=0$ then \eqref{eq:hcf} implies that there are no such solutions.
Alternatively, suppose that $\kappa\geq 1$.
Suppose first that $\alpha=\beta=\gamma=\delta=0$. In particular we have $(\frac{-a'b'}{p})=1$.
If  $\kappa\geq 2$ then we may  choose $u,v$ such that $v_p(Q_1')=1$, in which case 
\eqref{eq:hcf} implies that
 $v_p(Q_2')=1$.  If $\kappa=1$ then we choose $u,v$ so that 
 $v_p(Q_1')=3$ and it will follow that $v_p(Q_2')=1$.
When $\max\{\alpha,\beta,\gamma, \delta\}\geq 1$ we argue according to parity.
Without loss of generality we may suppose that $\beta=\delta=0$ and $\alpha,\gamma, \kappa\geq 1$.
If  $(\alpha,\gamma)\=(1,1)\2$ then it suffices to take $(u,v)=(1,0)$.
If $\alpha,\gamma$ are of opposite parity, with say
 $(\alpha,\gamma)\=(0,1)\2$, 
we consider $(u,v)=(u',p^{\alpha/2}v')$ chosen so that 
$v_p(Q_1')\=1\2$. Then it follows that $v_p(Q_2')\=1\2$ since $\alpha>\gamma$.
Finally, the case  $(\alpha,\gamma)\=(0,0)\2$ is impossible.

In case (iii) we must have $\k=(1,0)$ or $(0,1)$. 
By symmetry it will suffice to consider the case 
$\k=(1,0)$.  We wish to determine when there exist coprime $u,v$ such that $v_p(Q_1')$ is odd and $v_p(Q_2')$ is even. 
Suppose first that $p\nmid a'b'$, so that $(\frac{-a'b'}{p})=1$. Let $v_p(\Delta')=\kappa$. 
If $\kappa$ is even we may choose $u,v$ such that $v_p(Q_1')=\kappa+1$, in which case it follows from \eqref{eq:hcf} that $v_p(Q_2')=\kappa$ is even.   Likewise, 
if $\kappa$ is odd then $p\nmid c'd'$ and so we also  have $(\frac{-c'd'}{p})=1$. 
Thus we may choose $u,v$ such that $v_p(Q_2')=\kappa+1$ is even, in which case  $v_p(Q_1')=\kappa$ is odd. 
We suppose now that $p\mid a'$, say,  with $a'=p^\alpha a''$ and $p\nmid a''b'$.
If $p\mid d'$, write $d'=p^\delta d''$ with $p\nmid c'd''$.   If $\alpha$ is odd then we may take $(u,v)=(1,0).$ If $\alpha=2\alpha'$ is even  then we choose $(u,v)=(u',p^{\alpha'}v')$ such that $v_p(Q_1')=\alpha+1$, which is satisfactory. 
If, on the other hand, we write  $c'=p^\gamma c''$ for $\gamma\geq 0$ and $p\nmid c''d'$, then the situation is more complicated. Suppose that 
$\alpha$ is odd and $\gamma$ is even. Then it suffices to take $(u,v)=(1,0)$. 
If $\alpha$  and $\gamma$ are both odd then $\alpha\neq \gamma$ by Lemma \ref{lem:X-local}.
Supposing that   $\alpha=2\alpha'+1<\gamma$ then the argument used in the last part of the proof of   Lemma \ref{lem:X-local} shows that $(u,v)=(1,p^{\alpha'+1})$ suffices. If $\alpha$ and $\gamma$ are both even then
$\gamma=0$ and 
 $(\frac{-a''b'}{p})=(\frac{-c'd'}{p})=1$, resulting in a case that is easy to handle. Finally if $\alpha$ is even and $\gamma$ is odd, a case that requires $\alpha\geq 2$ and $\gamma=1$,  then $v_p(\Delta')=1$ and it is impossible to find  suitable $u,v$.  The case in which $p\mid b'$ is symmetric.
This  completes the proof of the lemma.
\end{proof}

\subsection{Solubility over $\RR$}

We now turn to the set $W_{\e}(\RR)$, given $(a,b,c,d)\in \Sloc^{(\iota)}$.
Recall the notation for the sign  function $\sign(t)\in \{-,+\}$.
We have  $\sign(a)=+$ and $\sign(b,c,d)\neq(+,-,-)$, 
by part 
(i) of Lemma \ref{lem:X-local}. On 
recalling that $m$ and $n$ are positive integers, we see that 
the constraint \eqref{eq:square} demands that $e_1$ and $e_2$ share the same sign. 
Let $\Delta=ad-bc$. 
We wish to determine when there exist $u,v\in \RR$ such that 
$\sign(e_i)Q_i(u,v)>0$ for $i=1,2$, which will then ensure that $W_{\e}(\RR)$ is non-empty.
We write $W_{\pm}$ for $W_{\e}(\RR)$ with $\sigma(e_1)=\sigma(e_2)=\pm$.

Suppose first that $c>0$. Taking $u\neq 0$ and $v=0$ shows that 
$W_+\neq \emptyset$.  If, furthermore, $b>0$ or $d>0$ then 
$W_-=\emptyset$.  Finally, if 
 $b,d<0$ then $W_-\neq \emptyset$, as can be seen by taking $u=0$ and $v\neq 0$.
Next we consider the case $c<0$.  Here, for $W_+$, we seek the existence of 
$u,v$ with $v\neq 0$ such that 
$$
\frac{-d}{c}>\frac{u^2}{v^2}>\max\left\{\frac{-b}{a}, 0 \right\}.
$$
This occurs if and only if  $d>0$ and $\sign(\Delta)=+$.
For $W_-$ we require instead the existence of $u,v$ with $v\neq 0$ such that
$$
\frac{-b}{a}>\frac{u^2}{v^2}>\max\left\{\frac{-d}{c}, 0 \right\}.
$$
This occurs if and only if  $b<0$ and $\sign(\Delta)=-$.
Finally we note that the case in which $\sign(a)=\sign(b)=+$ and 
$\sign(c)=\sign(d)=-$ does not enter into consideration.
We summarise the situation in Table \ref{table1}, in which the  
final column lists the possible signs of $e_i$ which give rise to a non-empty set
$W_{\e}(\RR).$ 

\begin{table}[h]
\centering
\begin{tabular}{|c|c|c|c|c|c|}
\hline
 & $\sign(a)$ & $\sign(b)$ & $\sign(c)$ & $\sign(d)$ &  $\sign(e_i): 
W_{\e}(\RR)\neq \emptyset$ \\
\hline
(i) & $+$ & $+$ & $+$ & $+$ &   $+$ \\
(ii) & $+$ & $-$ & $+$ & $-$ &  $+$ and $-$ \\
(iii) & $+$ & $+$ & $+$ & $-$ &  $+$ \\
(iv) & $+$ & $-$ & $+$ & $+$ &  $+$ \\
(v) & $+$ & $+$ & $-$ & $+$ &    $+$ \\
(vi) & $+$ & $-$ & $-$ & $+$ &   $\sign(\Delta)$ \\
(vii) & $+$ & $-$ & $-$ & $-$ &    $-$ \\
\hline
\end{tabular}
\vspace{0.2cm}
\caption{$W_{\e}(\RR)$ for $(a,b,c,d)\in \Sloc$
}\label{table1}
\end{table}

\subsection{Brauer group considerations}
Our work so far shows that for certain choices of coefficients 
$(a,b,c,d)\in \Sloc^{(\iota)}$
there is more than one choice of torsor $W_\e$ with points everywhere locally.  We need to show, using our discussion of the Brauer group above, that such  $(a,b,c,d)$ actually belong to $\Sglob^{(\iota)}$ and so are easily dealt with.
In carrying out this plan let us write 
 $j_w=\mathrm{inv}_w\left( \mathrm{ev}_{\mathcal{A}}(M_w) \right)$, for any local point $M_w\in X_{a,b,c,d}(\QQ_w)$, whose value is given by \eqref{eq:campus}.
We may now record the  following pair of results.

\begin{lemma}\label{lem:easy}
Let $(a,b,c,d)\in \Sloc^{(\iota)}$ arise in case (ii) of Table \ref{table1}. Then
$(a,b,c,d)\in \Sglob^{(\iota)}$.
\end{lemma}

\begin{proof}
We argue by contradiction.  Thus we 
may  assume that $j_\infty$ is constant, else certainly 
$(a,b,c,d)\in \Sglob^{(\iota)}$.
But then if 
$j_\infty=0$ (resp.\ $j_\infty=1/2$) it follows from \eqref{eq:campus} that 
$\sign(e_1)=+$ (resp.\ $\sign(e_1)=-$).
Hence case (ii) in
Table \ref{table1} is impossible.
\end{proof}

\begin{lemma}\label{lem:easy-ish}
Let $(a,b,c,d)\in \Sloc^{(\iota)}$.
If there is a prime  $p\equiv 3 \4$,
for which there are precisely two choices of $\k\in \{0,1\}^2$ with $k_1+k_2\= v_p(mn)\2$ such that
 $W_{(p^{k_1},p^{k_2})}(\QQ_p)$ is non-empty, then $(a,b,c,d)\in \Sglob^{(\iota)}$. 
\end{lemma}

\begin{proof}
We argue by contradiction.  Thus we 
may  assume that $j_p$ is constant, 
where $p$ is as in the statement of the lemma. 
If $j_p=0$  (resp.\  $j_p=1/2$), then \eqref{eq:campus} implies that $k_1=0$ (resp.\ $k_1=1$) and $k_2$ is uniquely determined from the congruence $k_1+k_2\=v_p(mn)\2$.
This completes the proof of the lemma.
\end{proof}

For the cases that are not covered by Lemmas \ref{lem:easy} or \ref{lem:easy-ish}, there 
exists a unique ${\bf e}$ such that for any valuation $w\in \Omega$, with $w\neq 2$,  
 we have  $W_{\bf e}(\QQ_w)\neq \emptyset $. For this 
 $\e$ it therefore  remains to determine when $W_{\bf e}(\QQ_2)$ is non-empty.
This is so if and only if there exist coprime $u,v\in \ZZ_2$ such that $e_iQ_i'(u,v)\in \cD$ for $i=1,2$, where 
$\cD$ is given by  \eqref{eq:D}.  
This constraint  depends only on  $\e$ modulo $4$.

\subsection{Solubility over $\QQ_2$}\label{s:2adic}

Recall the definition 
\eqref{eq:D} of $\cD$. 
In what follows it will be convenient to introduce the complementary set
$$
\ocD=\{2^n(3+4m): m\in \ZZ_2, ~n\in \ZZ \}.
$$
In this section we will be specifically interested in the set 
$$
\Tot=\left\{
t=(A,B,C,D)\in \ZZ^4: 
\begin{array}{l}
2\nmid A, ~
\Upsilon=AD-BC\neq 0\\
\gcd(A,B)=\gcd(C,D)=1\\
\gcd(A,C), \gcd(B,D)\in \cA
\end{array}
\right\},
$$
where $\cA$ is given  by \eqref{eq:AB}.
Define
$$
R_1(U,V)=AU^2+BV^2, \quad
R_2(U,V)=CU^2+DV^2,
$$
We wish  to  classify exactly the $t\in \Tot$ for which 
there exist coprime $u,v\in \ZZ_2$  such that 
\begin{enumerate}\item[(1)]
$R_i(u,v)\in \cD$ for $i=1,2$; or
\item[(2)]
$R_1R_2(u,v)\in \cD$.
\end{enumerate}
We see that case (2) holds if and only if 
there exist coprime $u,v\in \ZZ_2$ such that 
case (1) holds or such that 
$R_i(u,v)\in \ocD$ for $i=1,2$.
Let us distinguish these two subsets of $\Tot$ by writing 
$\Tot^{(1)}$  and  $\Tot^{(2)}$, respectively. In particular 
$\Tot^{(1)}\subset \Tot^{(2)}$.

Returning briefly to the question of solubility of $W_\e$ over $\QQ_2$, 
suppose that 
$$
e_i\equiv \eps_i \bmod{4}, \quad (i=1,2),
$$
for $\eps_1,\eps_2\in \{\pm 1\}$. We will make the change of variables
$$
A=\eps_1a', \quad B=\eps_1b', \quad C=\eps_2c', \quad D=\eps_2d'.
$$
In particular $(A,B,C,D)\in \Tot$ and we note that 
$e_iQ_i'(u,v)\in \cD$ if and only if $R_i(u,v)\in \cD$,  for $i=1,2$.
From case (1)  above we will be able to determine precisely when $W_\e(\QQ_2)\neq \emptyset$.
Likewise  case (2) allows us to decide exactly when 
$X_{a,b,c,d}(\QQ_2)\neq \emptyset$, as required for part (iv) of Lemma \ref{lem:X-local}.

It is now time to characterise the sets $\Toti$ and $\Totii$. Our argument differs according to the residue  of $(A,B)$ modulo $4$.
Consequently, given a pair $(i,j)\in (\ZZ/4\ZZ)^*\times \ZZ/4\ZZ$, it will be convenient to 
put
\begin{align*}
\Tot^{(k)}(i,j)&=\{t\in  \Tot^{(k)}: (A,B)\equiv (i,j) \4\},
\end{align*}
for $k\in \{1,2\}$.
For each $(i,j)$, we shall list precise conditions which are both necessary and sufficient to ensure that $t$ 
belongs to $\Toti(i,j)$ or $\Totii(i,j)$.  In doing so we will make repeated use of the observation that 
\begin{equation}\label{eq:x-files}
\Totii(i,j)=-\Totii(-i,-j).
\end{equation}
Although we provide an explicit characterisation of all the sets  $\Toti(i,j)$ and $\Totii(i,j)$, it is only the latter that 
will actually be used to calculate the numerical value of a certain $2$-adic density in 
 Lemma \ref{lem:tau-2}.  Furthermore, we will lend support to the final numerical value obtained by a direct 
computer search. 
Consequently, the reader may choose to skip the proofs in this section at a first reading.
As emphasised in \S \ref{s:intro}, however, once taken in conjunction with our work so far the results in this section give an explicit algorithm  for testing whether or not any Ch\^atelet surface $X_{a,b,c,d}$ 
has $\QQ$-rational points. We will illustrate this aspect in \S \ref{s:ct}.

\medskip

Suppose one is given coprime integers $r,s$.  It will be useful to have a  clear classification of precisely when there exists coprime $u,v\in \ZZ_2$ such that 
$ru^2+sv^2\in \cD$, and what constraints are placed on any such solution, if any.
 Write $u=2^\mu u'$ and $v=2^\nu v'$, with 
 $\min\{\mu, \nu\}=0$ and $2\nmid u'v'$.  
 It will be convenient to set
 \begin{equation}\label{eq:notation}
u'=1+2u'', \quad v'=1+2v'', \quad  
 \tilde{u}=u''+{u''}^2, \quad 
  \tilde{v}=v''+{v''}^2.
 \end{equation}
 In particular $\tilde{u}, \tilde{v}$ are both even integers. 
Let us suppose that $r=2^\rho r'$ and $s=2^\sigma s'$, with $\min\{\rho, \sigma\}=0$ and $2\nmid r's'$.  In Tables \ref{T} and \ref{T'} 
we collect together conditions on the various parameters  under which we have 
$ru^2+sv^2\in \cD$. 

\begin{table}[h]
\centering
\begin{tabular}{|l|l|l|}
\hline
 condition
&
condition
& 
conditions on $u',v'$\\
on  $\nu$
&
on $r',s \4$
& 
\\
\hline
$2\nu\leq \rho-2$ & $s\equiv 1 \4$ & \\
$2\nu\geq \rho+2$ & $r'\equiv 1 \4$ & \\
$2\nu=\rho+1$ & $r'\equiv 3 \4$ & \\
$2\nu=\rho-1$ & $s\equiv 3 \4$ & \\
$2\nu=\rho$ & $r'+s\equiv 2 \8$ & \\
$2\nu=\rho$ & $r'+s\equiv 4 \8$ &  $4(r'\tilde{u}+s\tilde{v})\equiv 4-r'-s \bmod{16}$
\\
$2\nu=\rho$ & $r'+s\equiv 0 \8$ &  there exists $k\in \ZZ_{\geq 0}$ such that \\
& & 
$4(r'\tilde{u}+s\tilde{v})\equiv 2^{3+k}-r'-s \bmod{2^{5+k}}$\\
\hline
\end{tabular}
\vspace{0.2cm}
\caption{$ru^2+sv^2\in \cD$: the case 
$\rho\geq 0$ and $\sigma=\mu=0$
}\label{T}
\end{table}

We will establish Table \ref{T} shortly. 
It represents an exhaustive list, so that in every case not covered in the table, one has $ru^2+sv^2\in \ocD$.  
It is important to stress that in every case listed a choice of $u',v'$ exists. It is only in the final two cases that additional constraints are imposed on $u',v'$. In these cases one notes that the existence of suitable $k,\tilde{u},\tilde{v}$  is automatic for the congruence class of $r'+s \8$ considered. This then leads to the existence of suitable 
$u',v'$ via \eqref{eq:notation}, since for any $t\in 2\ZZ_2$ and any $\ell\in \NN$ one can always find $w\in \ZZ/2^\ell \ZZ$ such that $w+w^2\equiv t \bmod{2^\ell}$.

\begin{proof}[Proof of Table \ref{T}]
For this we write $Q(U,V)=rU^2+sV^2$.
Taking $\mu=0$ we note that  
$$
Q(u',2^\nu v')
=2^{\min\{\rho, 2\nu\}} \left(
2^{\rho-\min\{\rho, 2\nu\}} r'u'^2+ 
2^{2\nu-\min\{\rho, 2\nu\}} sv'^2\right).
$$
When  $\rho\neq 2\nu$ it is easy to characterise when 
$Q(u',2^\nu v')\in \cD$. Thus, if $\rho>2\nu$  then 
for $\rho=2\nu+1$ (resp.\ $\rho\geq 2\nu+2$) one  requires $s\equiv 3\4$ 
(resp.\ $s=1\4$).  The case $\rho<2\nu$ is handled similarly. 
Suppose next that $\rho=2\nu$, so that 
$$
Q(u',2^\nu v')
=2^{\rho} \left(
r'u'^2+ 
sv'^2\right).
$$
Since $r'$ and $s$ are odd we must have $r'+s\in \{0,2,4,6\} \8$.  The case $r'+s\equiv 6\8$ 
is impossible since then $Q(u',2^{\nu}v')\in \ocD$ for  any odd $u',v'$.
If $r'+s\equiv 2\8$, on the other hand,  then $Q(u',2^{\nu}v')\in \cD$ for  any odd $u',v'$.
For the remaining cases we 
note that 
$$
2^{-\rho}Q(u',2^{\nu}v')=r'+s+4(r'\tilde{u}+s\tilde{v}),
$$
in the notation of \eqref{eq:notation}.
When $r'+s\equiv 4\8$, this is clearly congruent to $4$ modulo $8$ since $\tilde{u},\tilde{v}$ are both even.
We conclude in this case that 
$Q(u',2^{\nu}v')\in \cD$ if and only if  $u,v$ are chosen so that 
$$
r'+s+4(r'\tilde{u}+s\tilde{v})\equiv 4 \bmod{16},
$$
as claimed in the table. Similarly, if 
$r'+s\equiv 0\8$, then $8\mid 2^{-\rho}Q(u',2^{\nu}v')$
and  we arrive at the constraint that there exists $k\in \ZZ_{\geq 0}$ such that 
$$
r'+s+4(r'\tilde{u}+s\tilde{v})\equiv 2^{3+k} \bmod{2^{5+k}}.
$$
This concludes the proof of Table \ref{T}.
\end{proof}

We are now able to deduce a number of further tables from Table \ref{T}. 
Table \ref{T'} follows by  symmetry.   
We would also like to characterise precisely when  
there exist  coprime $u,v\in \ZZ_2$ such that $ru^2+sv^2\in \ocD$.  Tables \ref{uT} and \ref{uT'} are obtained from 
 Tables \ref{T} and \ref{T'}, respectively, 
 by multiplying the right hand side of each congruence in the second column by $-1$.

\begin{table}[h]
\centering
\begin{tabular}{|l|l|l|}
\hline
 condition
&
condition
& 
conditions on  $u',v'$\\
on  $\mu$ 
&
on $r,s' \4$
& 
\\
\hline
$2\mu\leq \sigma-2$ & $r\equiv 1 \4$ & \\
$2\mu\geq \sigma+2$ & $s'\equiv 1 \4$ & \\
$2\mu=\sigma+1$ & $s'\equiv 3 \4$ & \\
$2\mu=\sigma-1$ & $r\equiv 3 \4$ & \\
$2\mu=\sigma$ & $r+s'\equiv 2 \8$ & \\
$2\mu=\sigma$ & $r+s'\equiv 4 \8$ &  $4(r\tilde{u}+s'\tilde{v})\equiv 4-r-s' \bmod{16}$
\\
$2\mu=\sigma$ & $r+s'\equiv 0 \8$ &  there exists $k\in \ZZ_{\geq 0}$ such that \\
& & 
$4(r\tilde{u}+s'\tilde{v})\equiv 2^{3+k}-r-s' \bmod{2^{5+k}}$\\
\hline
\end{tabular}
\vspace{0.2cm}
\caption{$ru^2+sv^2\in \cD$: the case 
$\sigma\geq 0$ and $\rho=\nu=0$
}\label{T'}
\end{table}

\begin{table}[h]
\centering
\begin{tabular}{|l|l|l|}
\hline
 condition
&
condition
& conditions on $u',v'$\\
on  $\nu$
&
on $r',s \4$
& 
\\
\hline
$2\nu\leq \rho-2$ & $s\equiv 3 \4$ & \\
$2\nu\geq \rho+2$ & $r'\equiv 3 \4$ & \\
$2\nu=\rho+1$ & $r'\equiv 1 \4$ & \\
$2\nu=\rho-1$ & $s\equiv 1 \4$ & \\
$2\nu=\rho$ & $r'+s\equiv 6 \8$ & \\
$2\nu=\rho$ & $r'+s\equiv 4 \8$ &  $4(r'\tilde{u}+s\tilde{v})\equiv 4-r'-s \bmod{16}$
\\
$2\nu=\rho$ & $r'+s\equiv 0 \8$ &  there exists $k\in \ZZ_{\geq 0}$ such that \\
& & 
$4(r'\tilde{u}+s\tilde{v})\equiv 2^{3+k}-r'-s \bmod{2^{5+k}}$\\
\hline
\end{tabular}
\vspace{0.2cm}
\caption{$ru^2+sv^2\in \ocD$: the case 
$\rho\geq 0$ and $\sigma=\mu=0$
}\label{uT}
\end{table}

\begin{table}[h]
\centering
\begin{tabular}{|l|l|l|}
\hline
 condition
&
condition
& conditions on $u',v'$\\
on  $\mu$ 
&
on $r,s' \4$
& 
\\
\hline
$2\mu\leq \sigma-2$ & $r\equiv 3 \4$ & \\
$2\mu\geq \sigma+2$ & $s'\equiv 3 \4$ & \\
$2\mu=\sigma+1$ & $s'\equiv 1 \4$ & \\
$2\mu=\sigma-1$ & $r\equiv 1 \4$ & \\
$2\mu=\sigma$ & $r+s'\equiv 6 \8$ & \\
$2\mu=\sigma$ & $r+s'\equiv 4 \8$ &  $4(r\tilde{u}+s'\tilde{v})\equiv 4-r-s' \bmod{16}$
\\
$2\mu=\sigma$ & $r+s'\equiv 0 \8$ &  there exists $k\in \ZZ_{\geq 0}$ such that \\
& & 
$4(r\tilde{u}+s'\tilde{v})\equiv 2^{3+k}-r-s' \bmod{2^{5+k}}$\\
\hline
\end{tabular}
\vspace{0.2cm}
\caption{$ru^2+sv^2\in \ocD$: the case 
$\sigma\geq 0$ and $\rho=\nu=0$
}\label{uT'}
\end{table}

\begin{rem}\label{rem:odd}
Armed with Tables \ref{T}--\ref{uT'} we are able to record a rather succinct condition under which there exist odd coprime  $u,v\in \ZZ_2$ such that $ru^2+sv^2$ belongs to $\cD$ or $\ocD$.   Taking $\mu=\nu=0$ we see that 
there exist odd coprime $u,v\in \ZZ_2$ such that $ru^2+sv^2\in \cD$ (resp.\ 
$ru^2+sv^2\in \ocD$) if and only if $r+s\in \{0,1,2,4,5\} \bmod{8}$
(resp.\ $r+s\in \{0,3,4,6,7\} \bmod{8}$). Note that when $4\mid r+s$,  additional 
constraints are placed on the admissible $u,v$.
\end{rem}

For the remainder of this section
we will adhere to the  notation 
$$
B=2^\beta B', \quad 
C=2^\gamma C', \quad 
D=2^\delta D',
$$
for integers $\beta, \gamma, \delta\geq 0$ such that $2\nmid B'C'D'$, with $\min\{\beta,\delta\}\leq 1$ and $\min\{\gamma,\delta\}=0$.

\begin{lemma}\label{lem:1-1}
We have $t\in \Toti(1,1)$ if and only if one of the following holds:
\begin{itemize}
\item $C\in \cD$ or $D\in \cD$;
\item 
$(C,D)\equiv (3,3)\4$ with 
$A+B\equiv C+D\equiv 2 \bmod{8}$;

\item	 $C\equiv 3\4$, 
$\delta\geq 1$ and $D'\equiv 3\4$, such that $C+D'\equiv 2 \8$ if 
 $2\mid \delta$;

\item	 $D\equiv 3\4$,  $\gamma\geq 1$ and $C'\equiv 3\4$, such that 
$C'+D\equiv 2 \8$ if $2\mid \gamma$.

 \end{itemize}
Moreover $t\in \Totii(1,1)$ if and only if one of the following holds: 
\begin{itemize}
\item $C\in \cD$ or $D\in \cD$;
\item 
$(C,D)\equiv (3,3)\4$ with 
$A+B\equiv C+D \bmod{8}$;

\item	 $C\equiv 3\4$
 and $D'\equiv 3\4$, with one of the following:
\begin{itemize}
\item
$2\nmid \delta$,
\item
$2\mid \delta$, $\delta\geq 2$ and  $C+D'\=2\8$,
\item
$2\mid \delta$, $\delta\geq 2$ and 
$A+B\=C+D'\=6\8$;
\end{itemize}

\item	 $D\equiv 3\4$ and $C'\equiv 3\4$, 
with one of the following:
\begin{itemize}
\item
$2\nmid \gamma$,
\item
$2\mid \gamma$, $\gamma\geq 2$ and  $C'+D\=2\8$,
\item
$2\mid \gamma$, $\gamma\geq 2$ and
$A+B\=C'+D\=6\8$.
\end{itemize}
\end{itemize}
\end{lemma}

\begin{proof}
When $(A,B)\equiv (1,1) \4$, it is easy to see that $R_1(u,v)\in \cD$ if and only if 
$\mu\geq 1$ or $\nu\geq 1$ or $\mu=\nu=0$ and $A+B\equiv 2\8$.
Clearly   $R_2(1,0)\in \cD$ if  $C\in \cD$ and 
 $R_2(0,1)\in \cD$ if $D\in \cD$.
Suppose now that $C,D \in \ocD$. If 
$(C,D)\equiv (3,3)\4$ then in order to have $t\in \Toti$ we must restrict to the case in which $u,v$ are odd.  But then it follows from 
Remark \ref{rem:odd} that 
 $t\in \Toti(1,1)$ if and only if $A+B\equiv C+D\equiv 2 \8$, 
 with $t\in \Totii(1,1)\setminus \Toti(1,1)$ if and only if $A+B\equiv C+D\equiv 6 \8$.
Next we suppose that  $\delta\geq 1$. 
If $A+B\=6\8$ then we see that $t\in \Totii(1,1)$ by taking $(u,v)=(1,1)$. Next, 
Table~\ref{T'} implies that $R_2(u,v)\in \cD$ if and only if one of the following holds:
\begin{itemize}
\item[---]
$\nu=0$, $2\mu=\delta+1$ and $D'\equiv 3 \4$, 
\item[---]
$\nu=0$, $2\mu=\delta-1$ and $C\equiv 3 \4$, 
\item[---]
$\nu=0$, $2\mu=\delta$ and $C+D'\equiv 2 \8$.
\end{itemize}
Hence 
$t\in \Toti(1,1)$ if and only if 
$\delta$ is odd or $\delta$ is even and $C+D'\equiv 2 \8$.
Alternatively, if $\delta\geq 2$ is even then $t\in \Totii(1,1)$ if $C+D'\=6\8$ and $A+B\=6\8$, by our existing argument.
Finally, the case in which 
$D\equiv 3\4$ and  $\gamma\geq 1$ is symmetric. 
The statement of the lemma now follows.
\end{proof}

\begin{lemma}\label{lem:3-3}
We have $t\in \Toti(3,3)$ if and only if:
\begin{itemize}
\item
$A+B\equiv 2\bmod{8}$ and $C+D\in \{0,1,2,4,5\}\bmod{8}$. 
\end{itemize}
Moreover
the characterisation of $\Totii(3,3)$ follows 
from \eqref{eq:x-files} and 
Lemma \ref{lem:1-1}.
\end{lemma}

\begin{proof}
In analysing $\Toti(3,3)$ we note
$A+B\in \{2,6\}\8$.  We must have $u,v$ both being odd in any solution and it therefore follows from Remark \ref{rem:odd}  that 
there exist odd $u,v$ such that $R_i(u,v)\in \cD$, for $i=1,2$, if and only if
 $A+B\equiv 2\8$ and
$C+D\in \{0,1,2,4,5\}\bmod{8}$.
\end{proof}

\begin{lemma}\label{lem:1-3-3-1}
Suppose
that $t=(A,B,C,D)\=(1,3,3,1)\4$.
Then 
$t\in \Toti$ if and only if one of the following holds:
\begin{itemize}
\item 
$A+B\= 0 \8$;
\item
$A+B\=4 \8$ and $C+D\= 0 \8$;
\item
$A+B\=4 \8$ and $A+B\=C+D \bmod{16}$.
 \end{itemize}
Moreover $t\not\in \Totii\setminus \Toti$.
\end{lemma}

\begin{proof}
Write
$$
A+B=2^{k_1+2}\ell_1, \quad C+D=2^{k_2+2}\ell_2, 
$$
for $2\nmid \ell_1\ell_2$ and $k_1,k_2\geq 0$.
We will show that 
$t\in \Toti$ if and only if $
(k_1,k_2)\neq (0,0)$ or  $\ell_1+\ell_2\not \= 0\4$.
A little thought shows that this is equivalent to the constraints recorded in the first part of the lemma.

In determining whether or not each $R_i(u,v)$ belongs to $\cD$ we must necessarily restrict to odd coprime values of $u,v\in \ZZ_2$.
It will be convenient to set $r_i=R_i/4$ for $i=1,2$ and $AD-BC=4\Upsilon'$. 
Recall the  notation introduced in \eqref{eq:notation}.
We have the identities
\begin{align*}
D\ell_1 2^{k_1}-B\ell_2 2^{k_2}&=\Upsilon',\\
-C\ell_1 2^{k_1}+A\ell_2 2^{k_2}&=\Upsilon',
\end{align*}
and 
\begin{align*}
r_1(u,v)&=2^{k_1}\ell_1 +A\tilde{u}+B\tilde{v},\\
r_2(u,v)&=2^{k_2}\ell_2 +C\tilde{u}+D\tilde{v},\\
-Cr_1(u,v)+Ar_2(u,v)&=\Upsilon' (1+4\tilde{v}).
\end{align*}
In what follows we will make frequent use of the fact that the map 
$\ZZ_2\rightarrow 2\ZZ_2$,
given by $w\mapsto w+w^2$,  
is surjective.

By symmetry we may restrict attention to the case $k_2\geq k_1$.
Assume that $k_1\geq 2$, with $k_2\geq k_1+1$.
In this case  we choose $\tilde{v}=0$ and $\tilde{u}=2^{k_1-1}\tilde{u}'$, with 
$\tilde{u}'\=3\4$. In particular $\tilde{u}$ is even and we have
$$
r_1/2^{k_1-1}=2\ell_1 +A\tilde{u}' \equiv 1 \4, \quad 
r_2/2^{k_1-1}=2^{k_2-k_1+1}\ell_2 +C\tilde{u}' \equiv 1 \4, 
$$
as required.  We may henceforth assume that $k_1\leq 1$ or $k_2=k_1\geq 2$.

Suppose that $k_2=k_1=k$ and $\ell_1\equiv \ell_2\4$. Either $\ell_1,\ell_2$ are congruent to $1$ modulo $4$  and we  take $\tilde{u}=0$ and $\tilde{v}=2^{k+2}$, or 
$\ell_1,\ell_2$ are congruent to $3$ modulo $4$ and we  take $\tilde{u}=2^{k+1}$
and $\tilde{v}=0$.
Next suppose  that $k_2=k_1=k\geq 1$ and $\ell_1+\ell_2\equiv 0\4$. 
In this case 
$k'=v_2(\Upsilon')\geq 
k+2$.
We choose $\tilde{v}=0$ and $\tilde{u}$ such that 
$$
A\tilde{u}\equiv -2^{k}\ell_1+2^{k'-1} \bmod{2^{k'+1}}. 
$$
Then it follows that 
$r_1\equiv r_2 \equiv 2^{k'-1} \bmod{2^{k'+1}}$, as required. 
Next, when $k_2=k_1=0$ and $\ell_1+\ell_2\=0\4$, 
it follows $r_1$ and $r_2$ are both odd, with 
$r_1+r_2\= 0 \4$.
In this case, therefore, we must have 
$t\not\in \Toti\cup \Totii$.

Suppose now that $k_1=0$ and $k_2\geq 1$.  
Choose $\tilde{u}$ such that 
$2^{k_2}\ell_2+C\tilde{u} \= 2^i \bmod{2^{2+i}}$
for $i\in \{1,2\}$, 
and $\tilde{v}=0$, 
so that $r_2\in \cD$. Then 
$$
r_1=\ell_1+A\tilde{u} \equiv \ell_1+A\overline{C}(2^i-2^{k_2}\ell_2) \bmod{4}.
$$
If $k_2\geq 2$ we choose $i$ such that $\ell_1\equiv (-1)^i \4$. Alternatively, if $k_2=1$, we choose $i$ 
such that $\ell_1\equiv (-1)^{i-1} \4$.
This then ensures that $r_1\in \cD$, as required.

Let us now consider the case $k_1=1$, with $k_2\geq 2$, for which we take $\tilde{v}=0$.
 For any $i\in \{1,2\}$ we can choose $\tilde{u}$ such that $2^{k_2}\ell_2+C\tilde{u}\equiv 2^{1+i} \bmod{2^{3+i}}$, which thereby implies that $r_2\in \cD$ and $2\mid \tilde{u}$.
Next, with these choices for $\tilde{u}, \tilde{v}$, we have
$$
r_1/2 \= \ell_1 +A\overline{C} (2^i-2^{k_2-1}\ell_2) \bmod{4}.
$$
Thus, in order to ensure that $r_1\in \cD$, it suffices to take $i=2$ (resp.\ $i=1$) if 
$\ell_1\equiv 1 \4$ and $k_2\geq 3$,  or 
$\ell_1\equiv 3 \4$ and $k_2=2$
(resp.\
$\ell_1\equiv 1 \4$ and $k_2=2$,  or 
$\ell_1\equiv 3 \4$ and $k_2\geq 3$).
\end{proof}

\begin{lemma}\label{lem:1-3}
We have 
$t\in \Toti(1,3)$ if and only if one of the following holds:
\begin{itemize}
\item $C\in \cD$;
\item $C\in \ocD$ and $C+D\in \{1,2,5\}\8$;
\item $D\=1\4$, $\gamma=1$ and $C'\=3\4$;

\item $D\=3\4$, 
$\gamma\geq 2$ and
$C'\=3\4$, with 
$C'+D\=2\8$ if $2\mid \gamma$;
\item $(C,D)\=(3,1)\4$, with  one of the following:
\begin{itemize}
\item
$A+B\= 0 \8$,
\item
$A+B\=4 \8$ and $C+D\= 0 \8$,
\item
$A+B\=4 \8$ and $A+B\=C+D \bmod{16}$.
 \end{itemize}
 \end{itemize}
Moreover $t\in \Totii(1,3)$ if and only if  one of the following holds:
\begin{itemize}
\item $C\in \cD$;
\item $C\in \ocD$, $2\mid C$;
\item $C\=3\4$ and $C+D\in \{1,2,3,5,6,7\}\8$;
\item $(C,D)\=(3,1)\4$, with one of the following:
\begin{itemize}
\item
$A+B\= 0 \8$,
\item
$A+B\=4 \8$ and $C+D\= 0 \8$,
\item
$A+B\=4 \8$ and $A+B\=C+D \bmod{16}$.
 \end{itemize}
 \end{itemize}
\end{lemma}

\begin{proof}
We first deal with  $\Toti(1,3)$, the first two cases corresponding to the choices $(u,v)=(1,0)$ or $(u,v)\equiv (1,1)\2$ such that $R_1(u,v)\in \cD$, respectively. 
Assume now that $C\in \ocD$ and $C+D\not\in \{1,2,5\}\8$.
Suppose that $\gamma=0$. If $D\=1\4$ then the result follows from Lemma \ref{lem:1-3-3-1}. If $D\=2\4$   this case is already covered since then $C+D\=1\4$. If
$D\=3\4$ then either $C+D\=2\8$, which is already covered, or $C+D\=6\8$ wherein $t\not\in \Toti(1,3)$.
The case $4\mid D$ is also seen to be impossible. 
Suppose that $\gamma=1$. If $D\=3\4$ then $C+D\in\{1,5\}\8$, a case that we have excluded. If $D\=1\4$ then an analysis of Table \ref{T} shows that $t\in \Toti(1,3)$ on taking $\mu=0$ and $\nu=1$. Suppose now that $\gamma\geq 2$. If $D\=1\4$ then again $C+D\in \{1,5\}\8$. If $D\=3\4$ then $\nu\geq 1$ and Table \ref{T} easily leads to the constraints in the lemma.

For the characterisation of $\Totii(1,3)$, 
it will suffice to show that 
 $t\in \Totii(1,3)\setminus \Toti(1,3)$ if and only if  one of the following holds:
\begin{itemize}
\item $C\=3\4$ and $C+D\=6\8$;
\item $D\=3\4$,  $\gamma\geq 2$ is even and $C'\=3\4$, with 
$C'+D\=6\8$;
\item $C\=3\4$, 
$\delta\geq 2$.
 \end{itemize}
The first condition comes from 
Remark \ref{rem:odd}.  
Assume now that $C\in \ocD$ and $C+D\in \{0,3\}\4$. 
It remains to consider the case 
$(C,D)$ being congruent to $(3,0)$ or $(0,3)$ modulo $4$,
since  $t\not\in \Totii(1,3)\setminus \Toti(1,3)$  when $(C,D)\=(3,1)\4$ by Lemma \ref{lem:1-3-3-1}.
 In the latter case it is necessary to consider the case of even $\gamma\geq 2$ and $C'+D\not\=2\8$. One deduces from Table \ref{uT} that one must  have $C'+D\=6\8$, as required.  Turning to the case $(C,D)\=(3,0)\4$, for which one automatically has $t\not\in \Toti(1,3)$, an analysis of Table \ref{uT'} shows that $t\in \Totii(1,3)$ precisely when 
$C+D'\=0\4$ if $D'\=1\4$ and $\delta$ is even. 
\end{proof}

\begin{lemma}\label{lem:1-0}
We have 
$t\in \Toti(1,0)$ if and only if one of the following holds:
\begin{itemize}
\item $C\in \cD$ or $C+D\in \{0,1,2,4,5\}\8$, with $4\nmid D$;

\item $\gamma= 1$  and $(C',D)\= (3,1)\4$;
\item $\gamma\geq 2$ and $(C',D)\= (3,3)\4$,  with 
 $C'+D\equiv 2 \8$ if $2\mid \gamma$.

 \end{itemize}
Moreover $t\in \Totii(1,0)$ if and only if one of the following holds:
\begin{itemize}
\item $C\in \cD$ and $4\nmid D$;
\item
$C\in \ocD$ and $C+D\in \{0,1,2,4,5\}\8$;
\item
$C\in \ocD$, $4\nmid D$ and $C+D\in \{3,6,7\}\8$,
with  one of the following:
\begin{itemize}
\item
$C\=3\4$ and  $C+D\=6\8$, with
$A+B'\=6\8$ if
$B'\=1\4$ and 
 $2\mid \beta$;
\item $C'\=3\4$, 
$\gamma=  1$  and $D\=1\4$;
\item $C'\=3\4$, 
$\gamma\geq  2$  and $D\=3\4$, with one of the following:
\begin{itemize}
\item
$2\nmid \gamma$;
\item
$2\mid \gamma$ and $C'+D\=2\8$;
\item
$2\mid \gamma$ and  $C'+D\=6\8$, with 
$A+B'\=6\8$ if
$B'\=1\4$ and 
 $2\mid \beta$.
\end{itemize}
\end{itemize}
\end{itemize}
\end{lemma}

\begin{proof}
Note that  $\delta\leq 1$ and $\beta\geq 2$. 
Dealing first with the criteria for $t\in \Toti(1,0)$, the first condition arises from the consideration of solutions with $\nu=0$ and $\mu\geq 0$,  using Remark~\ref{rem:odd}. This case covers the case $\gamma=0$.
Hence we may assume that $C\in \ocD$.

Suppose now that 
 $\gamma\geq 1$.
If $D\equiv 3 \4$ then we must have $\nu\geq 1$ and it follows from Table \ref{T} that $2\nu=\gamma\pm 1$ or $2\nu=\gamma$ with $C'+D\equiv 2\8$. 
This easily leads to the 
overall constraint 
 $C'+D\equiv 2 \8$ if $2\mid \gamma$ and 
$D\equiv 3 \4$.  Next, the case $D\equiv 1 \4$
and $B'\=1 \4$ is handled by taking $(u,v)=(0,1).$
Finally, if  $D\equiv 1 \4$
and $B'\=3 \4$ our argument differs according to the parity of $\beta$.
If $2\nmid \beta$ then 
Table \ref{T'} allows us to choose $2\mu=\beta+1$.
If $2\mid \beta$ then we choose
$2\mu=2\beta$ which yields the result, on eliminating the cases already handled by the constraint $C+D\in \{0,1,2,4,5\}\8$.

Having determined
the criteria for $\Toti(1,0)$, we now turn to the corresponding criteria for 
$t\in \Totii(1,0)$.  
The first pair of alternatives come from the first part of the lemma. 
It remains to consider the case 
$C\in \ocD$, $4\nmid D$ and $C+D\in \{3,6,7\}\8$.
If $C\=3\4$, so that $C+D\=6\8$,  then the desired constraints follow from a direct application of 
Table \ref{uT}.
If $C'\=3\4$, with $\gamma\geq 1$, then we apply the first part of the lemma when $\gamma=1$. 
The treatment of the case $\gamma\geq 2$ follows from 
Tables \ref{uT} and \ref{uT'}.
\end{proof}

\begin{lemma}\label{lem:3-0}
We have 
$t\in \Toti(3,0)$ if and only if one of the following holds:
\begin{itemize}
\item $D\in \cD$ and $4\nmid D$, with $A+B'\=2\8$ if $2\mid \beta$;

\item $\delta= 1$ and $D'\=3\4$, with  $\beta\in\{2,3\}$.
\end{itemize}
Moreover 
the characterisation of $\Totii(3,0)$ follows  from \eqref{eq:x-files} and  Lemma \ref{lem:1-0}.
\end{lemma}

\begin{proof}
As in the previous lemma we note here that $\delta\leq 1$ and $\beta\geq 2$. 
In analysing  $\Toti(3,0)$, 
it follows from  Remark \ref{rem:odd} that there are no solutions with $\mu=0$. We must therefore consider the existence of solutions with $\mu\geq 1$ and $\nu=0$.  Calling upon Table \ref{T'} we must have one of the following:
\begin{itemize}
\item[---] $2\mu =\beta-1$,
\item[---] $2\mu \geq \beta+2$ if $B'\=1\4$,
\item[---] $2\mu =\beta+1$ if $B'\=3\4$,
\item[---] $2\mu =\beta$ if $A+B'\in \{0,2,4\}\8$,
\end{itemize}
in order to have $R_1(u,v)\in \cD$. Moreover there are additional constraints on $u,v$ 
in the last case when  $A+B'\in \{0,4\}\8$.
Suppose that $D\in \cD$. Then either $B'\=1\4$,  and we take $2\mu\geq \beta+2$, or else 
$B'\=3\4$, and there exists a suitable choice of $\mu$ provided that $A+B'\=2\8$ when $\beta$ is even. The case $D\=3\4$ being impossible, we suppose next that $\delta=1$ and $D'\=3\4$. Then Table \ref{T'} implies that we must have one of the following:
\begin{itemize}
\item[---] $2\mu =2$,
\item[---] $2\mu =0$ if $C\=3\4$,
\end{itemize}
in order to have $R_2(u,v)\in \cD$.  The lemma now easily follows.
\end{proof}

\begin{lemma}\label{lem:1-2}
We have 
$t\in \Toti(1,2)$ if and only if one of the following holds:
\begin{itemize}
\item $C\in \cD$;
\item $C'\= 3\4$,  
with 
one of the following:
 \begin{itemize}
 \item 
$D\equiv 1\4$, $\gamma\geq 1$,
\item
$D\equiv 3\4$,  $\gamma\geq 1$ and  $C'+D\= 2 \8$ if $2\mid\gamma$,
\item
$B'\=3\4$, $\gamma=0$,
$\delta\in\{0,2,3\}$ and $D'\equiv 1\4$,
\item
$B'\=3\4$, $\gamma=0$,  
$\delta\in\{1,2,3\}$ and
$D'\equiv 3\4$, 
with $C+D'\=2\8$ if $\delta=2$,
\item
$B'\=1\4$, $\gamma=0$ and 
$D'\equiv 1\4$,
\item
$B'\=1\4$, $\gamma=0$, $\delta\geq 3$ and $D'\equiv 3\4$, 
with 
$C+D'\=2\8$ if $2\mid \delta$.
 \end{itemize}

 \end{itemize}
Moreover $t\in \Totii(1,2)$ if and only if one of the following holds:
\begin{itemize}
\item $C\in \cD$;
\item $C'\equiv 3\4$, $\gamma\geq 1$;
\item  $C\equiv 3\4$, $\gamma=0$, with one of the following:
\begin{itemize}
\item $B'\= D'\4$, 
\item $(B',D')\in \{(3,1),(1,3)\}\4$, $\delta\neq 1$.
\end{itemize}

\end{itemize}

\end{lemma}

\begin{proof}
In this setting  $\beta=1$.
Beginning with  $\Toti(1,2)$, the first constraint follows on making the choice $(u,v)=(1,0).$  Thus we may suppose that $c\in \ocD$.
An analysis of Table~\ref{T} easily allows one to handle the solutions in which $\nu\geq 1$ and $\mu=0$. In this case we have $Au^2+2B'v^2\equiv 1 \4$ and $2^{\gamma}C'u^2+Dv^2\in \cD$ if and only if one of the following holds:
\begin{itemize}
\item[---]
$2\nu=\gamma+1$,
\item[---]
 $2\nu\leq \gamma-2$ if $D\=1\4$,
 \item[---]
$2\nu=\gamma-1$ if $D\=3\4$, 
\item[---]
$2\nu=\gamma$ if $C'+D\in \{0,2,4\}\8$.
\end{itemize}
In particular it suffices to have $\gamma\geq 1$ when $D\=1\4$. When $D\=3\4$ then any odd $\gamma\geq 1$ is satisfactory and even $\gamma\geq 2$ is satisfactory  only when $C'+D\=2 \8$.

We now consider possible solutions with $\nu=0$, using Table \ref{T'}. The case $\mu=0$ being impossible, we pass to the case $\mu=1$, in which case we must have $B'\=3\4$ and $\gamma=0$, the case $\gamma\geq 1$ already having been handled.  
It follows 
that $2=\delta-1$ or $2\geq \delta+2$ or $2=\delta$ when $D'\=1\4$. Furthermore, if $D'\=3\4$ then the table implies that 
 $2=\delta-1$ or $2=\delta+1$ or $2=\delta$, with the 
 last only taking place if $C+D'\=2\8$.
 Let us now turn to solutions with $\mu\geq 2$ and $\gamma=0$, so that $B'\=1\4$.
We deduce from Table~\ref{T'}
that it suffices to take $\mu$ so that $2\mu\geq \delta+2$ when $D'\=1\4$. 
When  $D'\=3\4$ then the table implies that 
 $2\mu=\delta-1$ or $2\mu=\delta+1$ or $2\mu=\delta$, with the 
 last only taking place if $C+D'\=2\8$.
This easily leads to the remaining constraints in the statement of the lemma. 

We now  determine
$\Totii(1,2)$, for which it suffices to show that 
$ t\in \Totii(1,2)\setminus \Toti(1,2)$ if and only if $C\in \ocD$, with one of the following holding:
\begin{itemize}
\item
$D\= 3\4$, $\gamma\geq 2$ is even and $C'+D\=6\8$;
\item
$B'\=3\4$, $\gamma=0$, $D'\=1\4$ and $\delta\geq 4$;
\item
$B'\=3\4$, $\gamma=0$, $D'\=3\4$, with one of the following:
\begin{itemize}
\item $\delta\not\in\{1,2,3\}$, 
\item $\delta=2$ and $C+D'\=6\8$;
\end{itemize}
\item
$B'\=1\4$, $D'\=3\4$, $\gamma=\delta=0$;
\item
$B'\=1\4$, $D'\=3\4$, $\gamma=0$, $\delta\geq 2$ even with $C+D'\=6\8$.
\end{itemize}
To see this we must have $\nu=0$ in any solution, else $R_1(u,v)\in \cD$. Remark \ref{rem:odd} shows that we have solutions with $R_i(u,v)\in \ocD$ for $i=1,2$, with $\mu=\nu=0$ if and only if $C+D\in \{0,3,4,6,7\} \8$ since $A+B\in \{3,7\}\8$.
If $\gamma\geq 1$ is even and $D\=3\4$ then it follows that $t\in \Totii(1,2)\setminus \Toti(1,2)$ if $C'+D\=6\8$. Turning to the case $\gamma=0$, 
as soon as $\delta\geq 3$ it is possible to find $u,v$ such that $R_i(u,v)\in \ocD$ for $i=1,2$, since then $C+D\in \{3,7\}\8$. 
Suppose that $B'\=3\4.$ If $D'\=1\4$ and $\delta=1$ then we do not get points in $\Totii(1,2)\setminus \Toti(1,2)$. If $D'\=3\4$ and $\delta\in \{0,2\}$ then we do get points in $\Totii(1,2)\setminus \Toti(1,2)$ by taking $\mu\geq 2$. It remains to consider   the case $B'\=1\4$ and $D'\=3\4$, for which we must have $\mu\in \{0,1\}$, together with $2\mu\leq \delta-2$ or $2\mu\geq \delta+2$ or $2\mu=\delta$, the latter only holding when $C+D'\=6\8$. Taking $\mu=1$ we see that $\delta=0$ or $\delta\geq 4$ is permissible, with $\delta=2$ being permissible if $C+D'\=6\8$. 
\end{proof}

\begin{lemma}\label{lem:3-2}
We have 
$t\in \Toti(3,2)$ if and only if one of the following holds:
\begin{itemize}
\item $C+D\in \{0,1,2,4,5\}\8$;
\item $D\=1\4$ and $C+D\in \{3,6\}\8$;
\item $C+D\=3\8$, $\delta=1$ and 
$D'\=B'\4$;
\item $C\=3\4$,  $\delta\geq 2$, with 
one of the following:
 \begin{itemize}
 \item 
$B'\equiv 3\4$, $\delta=3$,
\item
$B'\equiv 3\4$, $\delta=2$
and $C+D'\in \{0,2,4\}\8$,
\item
$(B', D')\=(1,1)\4$, 
\item
$(B',D')\=(1,3)\4$, $\delta\geq 3$ and 
$C+D'\=2\8$ if $2\mid \delta$.
 \end{itemize}

 \end{itemize}
Moreover 
the characterisation of $\Totii(3,2)$ follows  from \eqref{eq:x-files} and Lemma \ref{lem:1-2}.
\end{lemma}

\begin{proof}
In this setting $\beta=1$ and we need only deal with the criteria for $\Toti(3,2)$.
We note that $\nu=0$ in any solution. Suppose first that $\mu=\nu=0$. 
The first constraint then follows immediately from Remark \ref{rem:odd}.  If $C+D\=6\8$ then it is clear that the case $(C,D)\=(3,3)\4$ doesn't produce a point in $\Toti(3,2)$. If 
$(C,D)\=(1,1)\4$ then we choose $(u,v)=(2u',1)$, with $u'$ such that  $R_1(u,v)\in \cD$. It therefore remains to consider the case $C+D\=3\4$, in which  we must have $D\not\=3\4$.
 When $D\=1\4$ it suffices to take $u=0$ (resp.\ $u=2$) if $B'\=1\4$ (resp.\ $B'\=3\4$).  
Suppose that  $D=2D'$ with $2\nmid D'$. If $D'\=B'\=1\4$ then one takes $(u,v)=(0,1)$, while
if $D'\=B'\=3\4$ then one takes $(u,v)=(2,1)$.  Table \ref{T'} implies that $t\not\in \Toti(3,1)$ if $D'\not\= B'\4$.

Now suppose that 
$C\=3\4$ and 
$\delta\geq 2$. Then $B'$ is either congruent to $1$ or $3$ modulo $4$ and we proceed to consider the latter case first.  
Since $A=B'\equiv 3\4$ it follows from Table \ref{T'} that 
$2\mu\in \{0,2\}$, whence in fact $\mu=1$ and $\delta\leq 3$. The case $\delta=3$ presents no problems. If $\delta=2$ then we require odd coprime $u',v\in \ZZ_2$ such that $C{u'}^2+D'v^2\in \cD$, which is always possible if $C+D'\not\=6\8$.
Now suppose that
 $B'\equiv 1\4$.  According to  Table \ref{T'} 
 we must have $\mu\geq 2$. The case $D'\=1\4$ is resolved on taking $(u,v)=(0,1)$. If $D'\=3\4$ then we must have 
$2\mu\in \{\delta-1,\delta+1\}$, or else $2\mu=\delta$ if $C+D'\=2\8$.
Thus if $\delta\geq 4$ is even we take $2\mu=\delta$ and if $\delta\geq 3$ is odd we take 
$2\mu=\delta+1$.
\end{proof}

\subsection{A worked  example}\label{s:ct}

In order to illustrate our investigation so far, it is instructive to verify 
Example 5.4 in \cite[\S 5]{ct-c-s}, 
concerning 
the family of Ch\^atelet
surfaces $X=X_{1,1-k,-1,k}$
for $k\in \ZZ$ such that  $k\not\in \{0,1\}$.  
In this setting $\Delta=1$ and $m=n=1$, in the notation of \S \ref{s:prelim}. 
We wish to show that there exists $w\in \Omega$ such that $X(\QQ_w)=\emptyset$ if and only if 
\begin{equation}\label{eq:eg1}
\mbox{$k<0$ or $k=4^n(8m+7)$ for $n\geq 2,m\geq 0$.}
\end{equation}
We will also show that $X(\QQ)= \emptyset$ if and only if 
\begin{equation}\label{eq:eg2}
\mbox{$k<0$ or $k\=3\4$ or $k=4^n(8m+7)$ for $n\geq 1,m\geq 0$.}
\end{equation}
This recovers \cite[Prop.~C]{ct-c-s}, since it shows that  for
each positive integer   $k\equiv 3 \bmod{4}$, the 
surface $X=X_{1,1-k,-1,k}$ fails  the Hasse principle.

According to Lemma \ref{lem:X-local} we have $X(\RR)=\emptyset$ if and only $k<0$. Since $X(\QQ_p)\neq \emptyset$ for any odd prime, it remains  to determine  when 
$X(\QQ_2)= \emptyset$, to complete the description in \eqref{eq:eg1}.
Turning to \eqref{eq:eg2}, for which we suppose that $k>0$ and 
$k\neq 4^n(8m+7)$ for $n\geq 2,m\geq 0$, 
it follows from Lemma \ref{lem:1} that $X(\QQ)\neq \emptyset$ if and only if 
there exists $f\in \cB$ and $\eps\in \{\pm 1\}$ such that 
$W_{(\eps f,\eps f)}$
has solutions in $\QQ_w$ for every $w\in \Omega$. 
In fact we will be forced to take $\eps=f=1$.
It follows from part (vi) of 
Table \ref{table1} that $W_{(\eps f,\eps f)}(\RR)\neq \emptyset$ if we take $\eps=1$.  
Lemmas~\ref{lem:p1} 
and \ref{lem:p2} show that 
$W_{(f,f)}(\QQ_p)\neq \emptyset$ for every odd prime $p$, provided that $f=1$, which we now assume. 
We thereby retrieve \cite[Lemme 5.4.1]{ct-c-s}, which states that $X(\QQ)\neq \emptyset$ if and only if $W(\QQ_2)\neq \emptyset$, with $W=W_{(1,1)}$.

Our final task is to  determine precisely when $X(\QQ_2)$ and $W(\QQ_2)$ are empty. 
Let us set 
$(a,b,c,d)=
(1,1-k,-1, k).$ 
If $k\in \{1, 2\} \4$, so that $b\in \{0,3\}\4$,  then it follows from Lemmas \ref{lem:1-3} and \ref{lem:1-0} that $(a,b,c,d)\in \Toti$, whence $X(\QQ_2)$ and $W(\QQ_2)$ are both non-empty. 
Suppose next that $k\=3\4$, so that $b\=2\4$ and we turn to Lemma \ref{lem:1-2}.
The first part of this result ensures that $W(\QQ_2)=\emptyset$, whereas the second part implies that $X(\QQ_2)\neq\emptyset$. Finally we must consider the case $k=2^{n}k'$, with $n\geq 2$ and $2\nmid k'$. In particular $(a,b)\=(1,1)\4$ and we apply Lemma \ref{lem:1-1}.
This readily implies that $(a,b,c,d)\in \Toti$ if and only if
$k'\not\=7\8$, whenever $n$ is even.  Likewise, 
 $(a,b,c,d)\in \Totii$ if and only if
$k'\not\=7\8$, whenever $n$ is even and $n\geq 4$.

Bringing everything together,  we conclude that $X(\QQ_2)=\emptyset$ if and only if $k=4^n(8m+7)$ for $n\geq 2,m\geq 0$, whereas 
 $W(\QQ_2)=\emptyset$ if and only if $k\=3\4$ or $k=4^n(8m+7)$ for $n\geq 1,m\geq 0$. 
This completes the necessity and sufficiency of the conditions \eqref{eq:eg1} and \eqref{eq:eg2}.

\section{Asymptotics: preliminaries}\label{s:asymptotic-prelim}

The purpose of this section is to set the scene for the final analysis of the counting functions $\Nloc(P)$ and $\NBr(P)$, defined in \eqref{eq:card}.  
According to our work in \S\ref{s:prelim}, it is clear that
\begin{align}\label{eq:tuesday1}
\Nloc(P)
&=\frac{1}{4} \sum_{\iota\in \{0,1\}}  \#\{(a,b,c,d)\in \Sloc^{(\iota)}: \max\{|a|,|b|,|c|,|d|\}\leq P\},\\
\NBr(P)&=\frac{1}{4} \sum_{\iota\in \{0,1\}}  \#\left\{(a,b,c,d)\in \Sloc^{(\iota)}: 
\begin{array}{l}
\max\{|a|,|b|,|c|,|d|\}\leq P\\
X_{a,b,c,d}(\QQ)= \emptyset
\end{array}
\right\},\label{eq:tuesday2}
\end{align}
where $\Sloc^{(\iota)}$ is given by  \eqref{eq:Sloc+S}. 
Furthermore, Theorem \ref{c:cor1} will require us to establish an asymptotic formula for 
\begin{equation}\label{eq:tuesday3}
N(P)=\frac{1}{4}  \#\{(a,b,c,d)\in \Stot: \max\{|a|,|b|,|c|,|d|\}\leq P\},
\end{equation}
where $\Stot$ is 
given by \eqref{eq:Stot}.  This is achieved in Lemma \ref{lem:M-total}, the proof of which will serve as a warm-up for the  treatments of $\Nloc(P)$ and $\NBr(P)$, in \S \ref{s:asymptotic-Nloc} and \S \ref{s:lower-upper}, respectively.

For given non-zero integers $a,d$ with modulus at most $P$, 
it follows from the trivial estimate for the divisor function that there
are $O (P^\ve)$ choices 
for integers $b,c$ such that $ad-bc=0$.
At the expense of adding an error term 
$O (P^{2+\ve})$, this shows that we can henceforth drop the constraint 
$ad-bc\neq 0$ from the definitions \eqref{eq:Stot}, \eqref{eq:Stot'} of $\Stot$ and $\Stot^{(\iota)}$, respectively.

Before proceeding to our assessment of \eqref{eq:tuesday1}--\eqref{eq:tuesday3}, we first set out a basic estimate that will be useful to us, in which we recall the definition \eqref{eq:phi*} of $\phi_1^*=\phi^*$.

\begin{lemma}\label{lem:basic}
Let  $a\in \ZZ$ and $q_1,q_2\in \NN$ be such that $\gcd(q_1,q_2)=1$. Let $\ve>0$ and $x,T\geq 1$, with
$x\gg  Tq_2q_1^\ve$. 
Then we  have 
$$
\#\left\{n\leq x: \begin{array}{l}
\gcd(n,q_1)=1\\ 
n\equiv a \bmod{q_2}
\end{array}
\right\} 
= \frac{\phi^*(q_1)x}{q_2}\left(1+O \left(\frac{1}{T^{1-\ve}}\right)\right).
$$
\end{lemma}

\begin{proof}
This follows on noting that  the left hand side is 
$\phi^*(q_1)x/q_2+O(2^{\omega(q_1)}).$
\end{proof}

\medskip

Recall the definitions \eqref{eq:AB} of $\cA$ and $\cB$. 
Given any positive integers $N_1,N_2$ we write 
$d\mid N_2^\infty$ if any prime divisor of $d$ is also a prime divisor of $N_2$, and 
we write 
$\gcd(N_1,N_2^\infty)$ for the largest such divisor which is also a divisor or $N_1$.
It is now time to make a number of changes of variables in order to simplify the various conditions that arise in the sets $\Stot$ and $\Sloc^{(\iota)}$.  In the former set we will always   choose representative coordinates in such a way that $a>0$.
Let $(a,b,c,d)\in \Stot$. 
We write
$$
m=\gcd(a,b), \quad n=\gcd(c,d), 
$$
so that $m,n\in \cB$ and  $\gcd(m,n)=1$. 
We  make the initial change of variables
\begin{equation}\label{eq:change1}
a=ma', \quad b=\ve_2 mb', \quad c=\ve_3 n c', \quad d=\ve_4n d',
\end{equation}
where $(\ve_2,\ve_3,\ve_4)=\sign(b,c,d)$ and $a',b',c',d'\in \NN$ satisfy
$\gcd(a',b')=\gcd(c',d')=1$.  
Note that this is a departure from the change of variables used in \S \ref{s:prelim}, where we allowed $a',b',c',d'$ to have arbitrary sign.
Next, we define
\begin{align*}
\ell_1&=\gcd(a',(mn)^\infty), \quad
\ell_2=\gcd(b',(mn)^\infty), \\
\ell_3&=\gcd(c',(mn)^\infty), \quad
\ell_4=\gcd(d',(mn)^\infty).
\end{align*}
We now make the further change of variables
\begin{equation}\label{eq:change2}
a'=\ell_1a'', \quad b'=\ell_2b'', \quad c'= \ell_3c'', \quad d'=\ell_4 d'',
\end{equation}
for $a'',b'',c'',d''\in \NN$ satisfying 
 $\gcd(a''b''c''d'',mn)=1$.
We have $\gcd(a',b')=\gcd(c',d')=1$ if and only if 
 $\gcd(a'',b'')=\gcd(c'',d'')=1$ and 
\begin{equation}\label{eq:monday-1}
\gcd(\ell_1,\ell_2)=\gcd(\ell_3,\ell_4)=1.
\end{equation}
Likewise, we see that $\gcd(a,c),\gcd(b,d)\in \cA$ if and only if
$\gcd(a'',c''),\gcd(b'',d'')\in \cA$
and 
\begin{equation}\label{eq:monday-1'}
\gcd(m\ell_1,n\ell_3), ~\gcd(m\ell_2,n\ell_4)\in \cA.
\end{equation}
For given $m,n$ we henceforth put
\begin{equation}\label{eq:def-Lmn}
L(m,n)=\{\bell\in \NN^4: \mbox{$\ell_i\mid (mn)^\infty$ and \eqref{eq:monday-1}, \eqref{eq:monday-1'} hold}\}.
\end{equation}

It will be convenient to set
$$
m''=\gcd(a'',c''), \quad n''=\gcd(b'',d''), 
$$
so that $m'',n''\in \cA$ and  $\gcd(m'',n'')=\gcd(m''n'',mn)=1$. 
 We are now led to make the  change of variables
\begin{equation}\label{eq:change3}
a''=m''a''', \quad b''=n''b''', \quad c''=m'' c''', \quad d''=n'' d''',
\end{equation}
for $a''',b''',c''',d'''\in \NN$ satisfying 
\begin{equation}\label{eq:change3-coprime}
 \gcd(b'''d''',mn m'')=\gcd(a'''c''',mn n'')=1
\end{equation} 
and
\begin{equation}\label{eq:change3-coprime'}
 \gcd(a'''d''',b'''c''')=1.
\end{equation}
We will execute the asymptotic evaluation of $N(P), \Nloc(P)$ and $\NBr(P)$ by 
estimating the number of allowable $a''',b''',c''',d'''$ associated to given 
$\ve_2,\ve_3,\ve_4,m,n,m'',n''$ and $\bell$. 
In fact, for the treatments of 
$ \Nloc(P)$ and $\NBr(P)$ we will need to refine these transformations further, in order to extract the $2$-adic valuations of the variables.

Beginning with $N(P)$, our main goal in this section is a proof of the following result.

\begin{lemma}\label{lem:M-total}
We have 
$N(P)=\tau  P^4+O (P^{4-1/5+\ve}),$ for any $\ve>0$, where 
$$
\tau=
\frac{17}{16}
\prod_{p\=1\4} a_p \prod_{p\=3 \4}\left(\frac{1-\frac{1}{p}}{1+\frac{1}{p}}\right)^2
b_p,
$$
with $b_p$ given by \eqref{eq:b_p} and 
\begin{equation}\label{eq:a_p}
a_p=
\left(1-\frac{1}{p}\right)^2
\left(1+\frac{2}{p}+\frac{1}{p^2}-\frac{2}{p^4}\right).
\end{equation}
\end{lemma}

\begin{proof}
Our starting point is the modified version of \eqref{eq:tuesday3}, in which we have eliminated the constraint that 
$ad-bc\neq 0$ in the  definition of $\Stot$.  On multiplying through by $2^3$, we may work on the domain where $a,b,c,d$ are all positive.  We now work through the changes of variables 
\eqref{eq:change1}, \eqref{eq:change2} and \eqref{eq:change3}, with $\ve_2=\ve_3=\ve_4=+$. 
This leads to the expression
\begin{equation}\label{eq:weds1}
N(P)
=
2
\sum_{\substack{m,n\in \cB\\ \gcd(m,n)=1}}
\sum_{\bell\in L(m,n)}
\sum_{\substack{m'',n''\in \cA\\ 
\gcd(m'',n'')=1\\
\gcd(m''n'',mn)=1
}}
N  + O (P^{2+\ve}),
\end{equation}
where 
$L(m,n)$ is given by 
\eqref{eq:def-Lmn}
and $N$
denotes the number of 
$(a''',b''',c''',d''')\in \NN^4$  such that 
\eqref{eq:change3-coprime}
and 
\eqref{eq:change3-coprime'} hold, with
$$
mm''\ell_1a''', ~mn''\ell_2b''', ~nm''\ell_3c''', ~nn''\ell_4 d'''\leq P.
$$
Let $K=K(mn,m'',n'')$ denote the set of 
$\k\in \NN^4$ for which 
$$
\gcd(k_1k_4,mnm''n'')=\gcd(k_2,mnn'')=\gcd(k_3,mnm'')=1.
$$
We use the M\"obius function to take care of  \eqref{eq:change3-coprime'}, obtaining 
\begin{equation}\label{eq:weds2}
N=\sum_{\substack{\k\in K}}
 \mu(k_1)\cdots \mu(k_4)N(\k),
\end{equation}
where  $N(\k)$ now
denotes the number of 
$(a''',  b''',  c''',  d''')\in \NN^4$  such that 
\eqref{eq:change3-coprime} holds, with 
\begin{align*}
mm''\ell_1[k_1,k_2]  a''', ~mn''\ell_2[k_1,k_3]  b''' &\leq P,\\
nm''\ell_3[k_2,k_4]  c''', ~nn''\ell_4 [k_3,k_4]  d'''&\leq P.
\end{align*}
The estimation of $N(\k)$ is straightforward, being based on counting 
integers in particular intervals which satisfy certain coprimality constraints.
A trivial upper bound is given by 
$$
N(\k)\ll \frac{P^4}{m^2n^2m''^2n''^2 \ell_1\cdots \ell_4 [k_1,k_2]\cdots [k_3,k_4]}.
$$
Note that 
\begin{equation}\label{eq:rain}
\sum_{\ell \mid (mn)^\infty} \frac{1}{\ell^{\ve}}=\frac{1}{\phi_\ve^*(mn)},
\end{equation}
for any $\ve>0$,
where $\phi_\ve^*$ is given by \eqref{eq:phi*}.
 The reciprocal of the latter is  a function of $mn$ with average order $O (1)$. 
We may use these facts to restrict attention to parameters satisfying
$$
\max\{m,n,m'',n'',\ell_i,[k_i,k_j]\}\leq T,
$$
for any $T\geq 1$, 
with overall error $O (T^{-1+\ve}P^4 )$. 

Applying Lemma \ref{lem:basic} with $q_2=1$ four times, 
we conclude that
$$
N(\k)=\frac{\phi^*(mnm'')^2\phi^*(mnn'')^2 P^4}{
m^2n^2m''^2n''^2 \ell_1\cdots \ell_4 [k_1,k_2]\cdots [k_3,k_4]}
\left(1+O \left(\frac{1}{T^{1-\ve}}\right)\right),
$$
provided that $P\gg  T^{5+3\ve}$. 
Substituting this estimate into \eqref{eq:weds1} and \eqref{eq:weds2}, and extending 
the summation over the outer parameters 
 to infinity, we are therefore led to the conclusion that
$$
N(P)=\tau P^4+O \left(P^{2+\ve}+T^{-1+\ve}P^4\right),
$$
provided that  $P\gg T^{5+3\ve}$, where
\begin{align*}
\tau=~&
2
\sum_{\substack{m,n\in \cB\\ \gcd(m,n)=1}}\frac{1}{m^2n^2}
\sum_{\bell\in L(m,n)} 
\frac{\phi^*(mn)^4}{\ell_1\ell_2\ell_3\ell_4}
\sum_{\substack{m'',n''\in \cA\\ 
\gcd(m'',n'')=1\\
\gcd(m''n'',mn)=1
}} \frac{\phi^*(m'')^2\phi^*(n'')^2}{m''^2n''^2}
\\
&\times
\sum_{\substack{\k\in K}}
\frac{\mu(k_1)\cdots \mu(k_4)}
{[k_1,k_2]\cdots [k_3,k_4]}.
\end{align*}
Taking $T=P^{1/5-\ve/2}$ shows that the error term in this estimate is satisfactory for Lemma \ref{lem:M-total} and it remains to check that our value of $\tau$ agrees with what is recorded there.

To begin with one easily calculates
\begin{align*}
\sum_{\substack{\k\in K}}
\frac{\mu(k_1)\cdots \mu(k_4)}
{[k_1,k_2]\cdots [k_3,k_4]}
&=
\prod_{p\nmid mnm''n''} \left(1-\frac{4}{p^2}+\frac{4}{p^3}-\frac{1}{p^4}\right)
\prod_{p\mid m''n''} \left(1-\frac{1}{p^2}\right)\\
&=\prod_{p\nmid mn} \left(1-\frac{1}{p}\right)^2\left(1+\frac{2}{p}-\frac{1}{p^2}\right)
\prod_{p\mid m''n''} \frac{(1+\frac{1}{p})}{(1-\frac{1}{p})(1+\frac{2}{p}-\frac{1}{p^2})}.
\end{align*}
Introducing the sum over over $m'',n''\in \cA$ such that
$\gcd(m'',n'')=\gcd(m''n'',mn)=1$, 
we find that
\begin{align*}
\sum_{\substack{m'',n''
}} \frac{\phi^*(m'')^2\phi^*(n'')^2}{m''^2n''^2}
&\prod_{p\mid m''n''}
\frac{(1+\frac{1}{p})}{(1-\frac{1}{p})(1+\frac{2}{p}-\frac{1}{p^2})}
=\prod_{p\nmid mn} \left( 
\frac{1+\frac{2}{p}+\frac{1}{p^2}-\frac{2}{p^4}}
{1+\frac{2}{p}-\frac{1}{p^2}}\right).
\end{align*}
Next we note that 
\begin{align*}
\sum_{\bell\in L(m,n)}
\frac{\phi^*(mn)^4}{\ell_1\ell_2\ell_3\ell_4}
&=\prod_{p\mid mn} 
\left(1-\frac{1}{p}\right)^2 
\left(
1+\frac{2}{p}+\frac{1}{p^2}-\frac{2}{p^3}
\right)
=\prod_{p\mid mn} c_p,
\end{align*}
say,
on recalling the definition \eqref{eq:def-Lmn} of $L(m,n)$. Let $a_p$ be given by \eqref{eq:a_p} for any prime $p$. 
Putting this all together we conclude that 
\begin{align*}
\tau&=2
\sum_{\substack{m,n\in \cB\\ \gcd(m,n)=1}}\frac{1}{m^2n^2}\prod_{p\nmid mn} a_p \prod_{p\mid mn}c_p\\
&= 2a_2
\prod_{p\=1\4} a_p \prod_{p\=3 \4} \left(a_p+\frac{2c_p}{p^2}
\right).
\end{align*}
Noting that $a_2=17/32$ 
and $a_p+\frac{2c_p}{p^2}=(1-\frac{1}{p})^2(1+\frac{1}{p})^{-2}b_p$, with $b_p$ given by \eqref{eq:b_p}, this therefore concludes the proof of the lemma.
\end{proof}

\section{Asymptotics:  $N_{\mathrm{loc}}(P)$}\label{s:asymptotic-Nloc}

Building on our work in the previous section we now turn to the resolution of Theorems~\ref{t:M-local} and \ref{c:cor1}. 
We begin with a few words about the local factors $\tau_{\mathrm{loc},v}$ that appear in the statement of the former. 
Each  $\tau_{\mathrm{loc},v}$ may be interpreted as the density of points 
in $U(\QQ_v)/G(\QQ_v)$ for which the associated Ch\^atelet surface has $\QQ_v$-points. 
When $v=p$ is prime, this density is equal to
\begin{equation}\label{eq:tau-p}
\tau_{\mathrm{loc},p}=\lim_{k\rightarrow \infty}
p^{-4k}\#\left\{
(a,b,c,d)\in A(p^k):
\begin{array}{l}
X_{a,b,c,d}(\QQ_p)\neq \emptyset
\\ 
\min\{v_p(a),v_p(c)\}\leq 1\\
\min\{v_p(b),v_p(d)\}\leq 1\\
\min\{v_p(a),v_p(b)\}\leq \delta_p\\
\min\{v_p(c),v_p(d)\}\leq \delta_p
\end{array}
\right\},
\end{equation}
where
$A(p^k)=(\ZZ/p^k\ZZ)^4\setminus p(\ZZ/p^k\ZZ)^4$ and 
$
\delta_p=1$ if  $p\=3\4$, with $\delta_p=0$ otherwise.
Note that the 
condition 
$ X_{a,b,c,d}(\QQ_p)\neq \emptyset$ is precisely equivalent to the union of conditions issuing from parts (ii)--(iv) of
Lemma \ref{lem:X-local}.  In particular, this condition is vacuous when $p\=1\4$.  Part (i) of 
Lemma \ref{lem:X-local} easily yields
\begin{equation}\label{eq:tau-inf}
\tau_{\mathrm{loc},\infty}
=\frac{7}{4},
\end{equation}
since there are $7$ possible choices for the signs of $a,b,c,d$, with $a>0$ and 
$\sign(b,c,d)\neq (+,-,-)$.

Our starting point is \eqref{eq:tuesday1}, followed by the changes of variables \eqref{eq:change1} and  \eqref{eq:change2}.
It will be  important to keep track of the $2$-adic valuations of the variables involved. 
In particular we have already arranged things so that $a''$ is odd in $\Sloc^{(\iota)}$. 
Let us write $\gcd(N_1,N_2)_\flat$ for the greatest odd common divisor
of two integers $N_1$ and $N_2$.
Instead  of passing directly to 
\eqref{eq:change3}, we set
$$
m''=\gcd(a'',c''), \quad n''=\gcd(b'',d'')_\flat, 
$$
so that $m'',n''\in \cA$ and  $\gcd(m'',n'')=\gcd(m''n'',2mn)=1$. 
Likewise, 
for $\iota\in \{0,1\}$
let us put
\begin{equation}\label{eq:L2}
L_2^{(\iota)}=\{(\beta,\gamma,\delta)\in \ZZ_{\geq 0}^3: 
\min\{\gamma,\delta\}=0, ~
\min\{\beta,\delta\}\leq 1, ~\beta\geq 1-\iota\}.
\end{equation}
We are now led to make the  change of variables
\begin{equation}\label{eq:change3'}
a''=m''a''', \quad b''=n''2^\beta b''', \quad c''=m'' 2^\gamma c''', \quad d''=n'' 2^\delta d''',
\end{equation}
for $(\beta,\gamma,\delta)\in L_2^{(\iota)}$ and 
$a''',b''',c''',d'''\in \NN$ satisfying 
\begin{equation}\label{eq:change3-coprime-mod}
 \gcd(b'''d''',2mn m'')=\gcd(a'''c''',2mn n'')=1
\end{equation} 
and
\begin{equation}\label{eq:change3-coprime'-mod'}
\gcd(a'''d''',b'''c''')_\flat=1.
\end{equation}

The key difference between our analysis of $N(P)$ and $\Nloc(P)$ is the use of 
Lemma 
\ref{lem:X-local}.
Part~(i) of this result implies that we shall only be interested in 
$\bve=(\ve_2,\ve_3,\ve_4)\neq (+,-,-)$.
Part~(ii) ensures that 
$X_{a,b,c,d}(\QQ_p)\neq \emptyset$ for any $p\nmid mn$.
Now suppose that $p\mid mn$. Then according to part (iii)
we will have 
$X_{a,b,c,d}(\QQ_p)\neq \emptyset$ if and only if 
we do not have 
\begin{equation}\label{eq:monday-2}
v_p(\ell_1)=v_p(\ell_3)\=1\2, \quad \mbox{or}  \quad
v_p(\ell_2)=v_p(\ell_4)\=1\2,
\end{equation}
but we do have $p\nmid a'''b'''c'''d'''$ and 
$$
\lan \frac{-\ve_2 \ell_1\ell_2m''n''2^{\beta}a'''b'''}{p}\ran+
\lan \frac{-\ve_3\ve_4 \ell_3\ell_4m''n''2^{\gamma+\delta}c'''d'''}{p}\ran\geq 0,
$$
in the notation of \eqref{eq:symbol}. This last constraint is equivalent to demanding that $p\nmid a'''b'''c'''d'''$ and 
$$
\left( \frac{-\ve_2 \ell_{1,p}\ell_{2,p}m''n''2^{\beta}a'''b'''}{p}\right)+
\left( \frac{-\ve_3\ve_4 \ell_{3,p}\ell_{4,p}m''n''2^{\gamma+\delta}c'''d'''}{p}\right)\geq 0,
$$
if $v_p(\ell_i)\in 2\ZZ$  and $\ell_{i,p}=p^{-v_p(\ell_i)}\ell_i$, 
for $1\leq i\leq 4$.
In particular,
for given $\beta,\gamma,\delta, \bve,\bell,m'',n''$,
 these constraints force the vector $(a''',b''',c''',d''')$ to lie in one of finitely many congruence classes
modulo $p$, for each $p\mid mn$. Let us denote the set of possible classes modulo $mn$ by 
$$
T_{mn}(\bell)\cong \bigoplus_{p\mid 
mn} T_p(\bell).
$$    
Here $T_{mn}(\bell)\subseteq (\ZZ/mn\ZZ)^4$ and 
 $T_{p}(\bell)\subseteq (\ZZ/p\ZZ)^4$ for each prime $p$.
The following result calculates their cardinality. 

\begin{lemma}\label{lem:card-T}
Let  $p\mid mn$. Then 
we have 
$$
\#T_{p}(\bell)=
\begin{cases}
\frac{3}{4}(p-1)^4, &\mbox{if $v_p(\ell_i)\in 2\ZZ$ for each $1\leq i\leq 4$,}\\
(p-1)^4, &\mbox{otherwise}.
\end{cases}
$$
\end{lemma}

\begin{proof}
Firstly, if 
 $v_p(\ell_i)\not\in 2\ZZ$ for some $1\leq i\leq 4$, then the only constraint modulo  $p$ is that the coordinates should all be coprime to $p$.  The claim is therefore obvious in this case. If  
 $v_p(\ell_i)\in 2\ZZ$ for each $1\leq i\leq 4$, then for given 
$\ve_2,  \ell_{1,p}, \ell_{2,p},m'',n''$ not divisible by $p$, it is clear that there are $\frac{1}{2}(p-1)^2$ choices for $a''',b'''$ modulo $p$ such that 
$$
\left( \frac{-\ve_2 \ell_{1,p}\ell_{2,p}m''n''2^{\beta}a'''b'''}{p}\right)=1.
$$
Likewise there are $\frac{1}{2}(p-1)^2$ choices of $c''',d'''$ modulo $p$ for which 
$$
\left( \frac{-\ve_3\ve_4 \ell_{3,p}\ell_{4,p}m''n''2^{\gamma+\delta}c'''d'''}{p}\right)=1,
$$
whence we arrive at the desired expression  on subtracting the contribution from $a''',b''',c''',d'''$ modulo $p$ for which  
both Legendre symbols are $1.$
\end{proof}

We need to restrict attention to $\bell\in L(m,n)$ for which 
\eqref{eq:monday-2} does not hold for any $p\mid mn$. We will capture the latter constraint by writing 
$\neg$\eqref{eq:monday-2}, where $\neg$ is the symbol  used for logical negation.
Summarising our investigation so far, 
our analogue of \eqref{eq:weds1} is 
\begin{equation}\label{eq:thurs1}
\begin{split}
\Nloc(P)
=~&
\frac{1}{4} \sum_{\iota\in \{0,1\}}
\sum_{\substack{ \bve\in\{\pm\}^3\setminus(+,-,-)}}
\sum_{\substack{m,n\in \cB\\ \gcd(m,n)=1}}
\sum_{\substack{\bell\in L(m,n)\\
\scriptsize{\mbox{$\neg$\eqref{eq:monday-2}}
}}}\\
&\times
\sum_{\substack{m'',n''\in \cA\\ 
\gcd(m'',n'')=1\\
\gcd(m''n'',2mn)=1
}}
\sum_{(\beta,\gamma,\delta)\in L_2^{(\iota)}}
\Nloc  + O (P^{2+\ve}),
\end{split}
\end{equation}
where 
$L(m,n)$ (resp.\ $L_2^{(\iota)}$) is given by 
\eqref{eq:def-Lmn} (resp.\  \eqref{eq:L2})
and $\Nloc$ 
denotes the total number of 
$(a''',b''',c''',d''')\in \NN^4$  such that 
\eqref{eq:change3-coprime-mod} and 
\eqref{eq:change3-coprime'-mod'} hold, with 
$$
mm''\ell_1a''', ~mn''\ell_22^\beta b''', ~nm''\ell_32^\gamma c''', ~nn''\ell_4 2^\delta 
d'''\leq P,
$$
$X_{a,b,c,d}(\QQ_2)\neq \emptyset$ and 
$
(a''',b''',c''',d''') \in T_{mn}(\bell) \bmod{mn}.
$ 

As before we use the M\"obius function to detect the condition \eqref{eq:change3-coprime'-mod'}.  Recall the definition of $K=K(mn,m'',n'')$ from \S
\ref{s:asymptotic-prelim} and define $K'=K'(mn,m'',n'')$ to be the subset of 
$\k\in K$ for which $2\nmid k_1k_2k_3k_4$.
It will be convenient to set
\begin{align*}
A&=
mm''\ell_1 [k_1,k_2] , \,\,\,\,\quad B=mn''\ell_22^\beta [k_1,k_3],\\ 
C&=nm''\ell_32^\gamma 
[k_2,k_4], \quad D=nn''\ell_4 2^\delta  [k_3,k_4].
\end{align*}
Then  
\begin{equation}\label{eq:thurs2}
\Nloc=\sum_{\substack{\k\in K'}}
 \mu(k_1)\cdots \mu(k_4)\Nloc(\k),
\end{equation}
where  $\Nloc(\k)$ is now the number of 
$(a''',  b''',  c''',  d''')\in \NN^4$  such that 
\eqref{eq:change3-coprime-mod} holds, with 
$$
Aa''', ~ Bb''', ~C c''', ~ Dd'''\leq P,
$$
$X_{a,b,c,d}(\QQ_2)\neq \emptyset$ and 
\begin{equation}\label{eq:pig1}
([k_1,k_2]a''',\ldots,[k_3,k_4]d''') \in T_{mn}(\bell) \bmod{mn}.
\end{equation}

Before  estimating $\Nloc(\k)$, it remains to interpret the 
constraint $X_{a,b,c,d}(\QQ_2)\neq \emptyset$, 
as described in 
part (iv) of Lemma \ref{lem:X-local}.
An inspection of Lemmas
\ref{lem:1-1}--\ref{lem:3-2} shows that for a given  choice of
$(\beta,\gamma,\delta)\in L_2^{(\iota)}$, the $2$-adic constraints placed on
$(a''', b''', c''', d''')$ constitute a union of particular congruence classes modulo $16$. Let us 
denote by $H_{\beta,\gamma,\delta}$ the set of all 
possible classes modulo $16$ that can arise. 
Then 
the constraint $X_{a,b,c,d}(\QQ_2)\neq \emptyset$ in $N(\k)$ is equivalent to demanding that 
\begin{equation}\label{eq:pig2}
( a''', b''', c''',d''' )
 \in H_{\beta,\gamma,\delta} \bmod{16}.
\end{equation}
Recall the definition \eqref{eq:L2} of $L_2^{(\iota)}$. 
Note that 
$H_{\beta,\gamma,\delta}$  depends on $\bve, \k, \bell, m,n,m'',n''$ in addition to 
$\beta,\gamma, \delta$, but the cardinality of this set does not depend on these additional parameters, being equal to 
the total number of $(A',B',C',D') \in(\ZZ/16\ZZ)^4$, with 
$2\nmid A'B'C'D'$, for which 
$(A',2^\beta B',2^\gamma C',2^\delta D')$ belongs to $\Totii$ modulo ${16}$, in the notation of \S \ref{s:2adic}.
A little thought therefore reveals that 
\begin{align*}
\sum_{\iota\in \{0,1\}}
\sum_{(\beta,\gamma,\delta)\in L_2^{(\iota)}}
\frac{\#H_{\beta,\gamma,\delta}}{2^{16+\beta+\gamma+\delta}}=\tau_{\mathrm{loc},2}, 
\end{align*}
where $\tau_{\mathrm{loc},2}$ is given by  \eqref{eq:tau-p}.
We will give a numerical value for $\tau_{\mathrm{loc},2}$ in Lemma \ref{lem:tau-2}.

The estimation of 
$\Nloc(\k)$ is straightforward, being based on Lemma \ref{lem:basic}.
A  trivial upper bound is  given by 
$
\Nloc(\k)\ll P^4/(ABCD).
$
As previously, this shows that we may restrict attention to parameters satisying
$$
\max\{m,n,m'',n'',\ell_i, [k_i,k_j],  2^{\beta},2^{\gamma},2^{\delta}\}\leq T,
$$
with error $O (T^{-1+\ve}P^4 )$.
Having reduced the size of the parameters suitably we turn to an asymptotic formula for 
$\Nloc(\k)$.
The coprimality conditions $\gcd(a'''b'''c'''d''',2mn)=1$
in \eqref{eq:change3-coprime-mod} will automatically be taken care of by 
\eqref{eq:pig1} and \eqref{eq:pig2}.
This time we will employ Lemma  \ref{lem:basic} with $q_1\in \{m'',n''\}$,  $q_2=16mn$ and $x\in \{P/A,\ldots,P/D\}$.
Under the assumption $P\gg  T^{8+\ve}$
we therefore obtain
$$
\Nloc(\k)=
\frac{
\phi^*(m'')^2\phi^*(n'')^2
 P^4}{ABCD}
\cdot \frac{\#T_{mn}(\bell)}{(mn)^4} \cdot \frac{\#H_{\beta,\gamma,\delta}}{2^{16}}\left(1+O \left(\frac{1}{T^{1-\ve}}\right)\right).
$$
We take $T=P^{1/8-\ve/2}$ and 
substitute this estimate into 
\eqref{eq:thurs1} and \eqref{eq:thurs2},  extending the summation over the outer parameters 
 to infinity. This  concludes the statement of Theorem \ref{t:M-local}, with 
\begin{align*}
\tau_{\mathrm{loc}}=~&
\frac{ \tau_{\mathrm{loc},2}}{4} 
\sum_{\substack{ \bve\in\{\pm\}^3\setminus(+,-,-)}}
\sum_{\substack{m,n\in \cB\\ \gcd(m,n)=1}}
\frac{1}{m^2n^2}
\sum_{\substack{\bell\in L(m,n)\\
\scriptsize{\mbox{$\neg$\eqref{eq:monday-2}}
}}}
\frac{\#T_{mn}(\bell)}{\ell_1\ell_2\ell_3\ell_4 (mn)^4} 
\\
&\times
\hspace{-0.4cm}
\sum_{\substack{m'',n''\in \cA\\ 
\gcd(m'',n'')=1\\
\gcd(m''n'',2mn)=1
}}
\hspace{-0.4cm}
\frac{\phi^*(m'')^2\phi^*(n'')^2}{m''^2n''^2}
\sum_{\substack{\k\in K'}}
\frac{\mu(k_1) \cdots \mu(k_4)}
{[k_1,k_2]\cdots [k_3,k_4]}.
\end{align*}
It remains to check that this constant satisfies the description given in the statement of the theorem.

The calculations at the close of \S \ref{s:asymptotic-prelim} immediately yield
\begin{align*}
\sum_{\substack{m'',n''
}} \frac{\phi^*(m'')^2\phi^*(n'')^2}{m''^2n''^2}
&\sum_{\substack{\k\in K'}}
\frac{\mu(k_1)\cdots \mu(k_4)}
{[k_1,k_2]\cdots [k_3,k_4]}
=\prod_{p\nmid 2mn} a_p,
\end{align*}
with $a_p$ given by \eqref{eq:a_p}.
For given $i,j\in \{0,1\}$, with $\min\{i,j\}=0$,  let
\begin{equation}\label{eq:lab-V}
E(i,j)=
\left\{
\bnu \in \ZZ_{\geq 0}^4:
\begin{array}{l}
\min\{\nu_1,\nu_2\}=
\min\{\nu_3,\nu_4\}=0\\
\min\{i+\nu_1,j+\nu_3\}\leq 1\\
\min\{i+\nu_2,j+\nu_4\}\leq 1\\
\nu_1=\nu_3 \Rightarrow 
\nu_1=\nu_3 =0\\
\nu_2=\nu_4 \Rightarrow 
\nu_2=\nu_4 =0
\end{array}
\right\}.
\end{equation}
The conditions in this set arise from the constraints 
 \eqref{eq:monday-1}, \eqref{eq:monday-1'}  in the definition of $L(m,n)$, together with the constraint that 
 \eqref{eq:monday-2} should not hold.
 Applying Lemma \ref{lem:card-T}, 
we   deduce that
\begin{align*}
\sum_{\substack{\bell\in L(m,n)\\
\scriptsize{\mbox{$\neg$\eqref{eq:monday-2}}
}}}
\frac{\#T_{mn}(\bell)}{\ell_1\ell_2\ell_3\ell_4 (mn)^4} 
&=\prod_{p\mid mn} c_p',
\end{align*}
where
\begin{align*}
c_p'
&=\sum_{\bnu\in E(v_p(m),v_p(n))}
\frac{
\#T_{p} (
p^{\nu_1},p^{\nu_2},p^{\nu_3},p^{\nu_4}  )}{
p^{4+\nu_1+\nu_2+\nu_3+\nu_4}} \\
&=
\left(1-\frac{1}{p}\right)^4
\sum_{\bnu\in E(v_p(m),v_p(n))}
\frac{
(3/4)^{f(\bnu)}}
{p^{\nu_1+\nu_2+\nu_3+\nu_4}},
\end{align*}
for any $p\mid mn$, 
where 
$$
f(\bnu)
=
\begin{cases}
1, & \mbox{if $\nu_i\in 2\ZZ$ for $1\leq i\leq 4$},\\
0, & \mbox{otherwise}.
\end{cases}
$$
A straightforward calculation reveals that
\begin{equation}\label{eq:omega'-p}
\begin{split}
c_p'
&=\left(1-\frac{1}{p}\right)^4 \left(
1+\frac{4+\frac{3}{p}+\frac{2}{p^2}+\frac{2}{p^3}}{p(1-\frac{1}{p^2})}
+\frac{2+\frac{4}{p}+\frac{3}{2p^2}}{p^2(1-\frac{1}{p^2})^2}
\right)\\
&=\left(1-\frac{1}{p}\right)^2 
\left(1+\frac{1}{p}\right)^{-2}
\left(
1+\frac{4}{p}+\frac{3}{p^2}+\frac{2}{p^3}+\frac{3/2}{p^4}
-\frac{2}{p^5}-\frac{2}{p^6}
\right).
\end{split}
\end{equation}

Returning to our expression for $\tau_{\mathrm{loc}}$ 
it remains to execute the summation over $\bve$ and $m,n$.
The latter leads to the identity
\begin{align*}
\sum_{\substack{m,n\in \cB\\ \gcd(m,n)=1}}\frac{1}{m^2n^2}\prod_{p\nmid 2mn} a_p \prod_{p\mid mn}c_p'
&= 
\prod_{p\=1\4} a_p \prod_{p\=3 \4} \left(a_p+\frac{2c_p'}{p^2}
\right).
\end{align*}
Carrying out the sum over $\bve$, and recalling \eqref{eq:tau-inf}, we are therefore led to the final expression
\begin{equation}\label{eq:final-c}
\tau_{\mathrm{loc}}=\tau_{\mathrm{loc},\infty} \tau_{\mathrm{loc},2}
\prod_{p\=1\4}a_p \prod_{p\=3 \4} \left(a_p+\frac{2c_p'}{p^2}\right),
\end{equation}
with $a_p,c_p'$  given by \eqref{eq:a_p} and \eqref{eq:omega'-p}, respectively.

Now it is easy to check that $a_p=\tau_{\mathrm{loc},p}$
in \eqref{eq:tau-p} when $p\=1\4$. 
In order to complete the proof of Theorem \ref{t:M-local} it therefore remains to show that 
$$
a_p+\frac{2c_p'}{p^2}=\tau_{\mathrm{loc},p},
$$
when $p\=3\4$.  To see this we may split the calculation of $\tau_{\mathrm{loc},p}$ into three densities, according to whether $p\nmid mn$ or $p\mid m$ or $p\mid n$.  The density corresponding to the first case is $a_p$. We claim that the density $d_p$, say,  corresponding to the second case is $c_p'/p^2$. This will suffice to establish the desired identity, since the third density follows by symmetry.  Recalling \eqref{eq:tau-p} it is clear that 
\begin{align*}
d_p&=
\frac{1}{p^2}
\lim_{k\rightarrow \infty}
p^{-4k}\#\left\{
(a',b',c,d)\in (\ZZ/p^k\ZZ)^4: 
\begin{array}{l}
X_{pa',pb',c,d}(\QQ_p)\neq \emptyset
\\ 
\min\{1+v_p(a'),v_p(c)\}\leq 1\\
\min\{1+v_p(b'),v_p(d)\}\leq 1\\
\min\{v_p(a),v_p(b)\}\leq \delta_p\\
\min\{v_p(c),v_p(d)\}=0
\end{array}
\right\}\\
&=\frac{1}{p^2} \sum_{\bnu\in E(1,0)} 
\frac{
\#T_{p} (
p^{\nu_1},p^{\nu_2},p^{\nu_3},p^{\nu_4}  )}{
p^{4+\nu_1+\nu_2+\nu_3+\nu_4}},
\end{align*}
in the notation of \eqref{eq:lab-V}.  It is now clear that $d_p=c_p'/p^2$, 
as claimed.

\bigskip

The remainder of this section is concerned with 
the proof of Theorem \ref{c:cor1}. There are two ingredients to this. The first is a numerical evaluation of the constant $\tau_{\mathrm{loc}}$ in \eqref{eq:final-c} and the second is the asymptotic estimate for  $N(P)$ in Lemma \ref{lem:M-total}.
Beginning with the former, we have the following calculation of the $2$-adic density $\tau_{\mathrm{loc},2}$ in \eqref{eq:tau-p}. 

\begin{lemma}\label{lem:tau-2}
We have $\tau_{\mathrm{loc},2}=4751/(9\times 2^{10})$.
\end{lemma}

\begin{proof}
Let $(i,j)\in (\ZZ/4\ZZ)^2$, with $2\nmid i$. 
To establish the lemma it will be convenient to 
set
\begin{equation*}
\tau_2(i,j)
= 
\sum_{(\beta,\gamma,\delta)\in L_2^{(1)}}
\frac{\#H_{\beta,\gamma,\delta}(i,j)}{2^{16+\beta+\gamma+\delta}},
\end{equation*}
where $L_2^{(1)}$ is given by \eqref{eq:L2} and 
$\#H_{\beta,\gamma,\delta}(i,j)$ is the number of $(A',B',C',D') \in(\ZZ/16\ZZ)^4$, with 
$2\nmid A'B'C'D'$, for which 
$(A',2^\beta B',2^\gamma C',2^\delta D')$ belongs to $\Totii$ modulo ${16}$
and 
$$
(A',2^\beta B')\=(i,j)\4.
$$
We note that 
$\tau_2(i,j)=\tau_2(-i,-j)$. Hence we have 
\begin{align*}
\tau_{\mathrm{loc},2}
&=\sum_{i\in\{1,3\}}\sum_{j\in\{0,2\}}\tau_2(i,j)+
\sum_{i\in\{1,3\}}\sum_{j\in\{0,1,2,3\}}\tau_2(i,j)\\
&=2\sum_{i\in\{1,3\}}\sum_{j\in\{0,2\}}\tau_2(i,j)+
\sum_{i,j\in\{1,3\}}\tau_2(i,j)\\
&=4\left(\tau_2(1,0)+\tau_2(1,2)\right)+
2\left(\tau_2(1,1)+\tau_2(1,3)\right).
\end{align*}
The calculation of the densities $\tau_2(i,j)$ is 
based on our analysis in \S \ref{s:2adic} 
The process is routine but  very
 tedious. We have decided merely to record the outcome of the investigation in Table~\ref{t:2-adic}.
 \begin{table}[h]
\centering
\begin{tabular}{|l|l|l|}
\hline
$(i,j) \4$
&
$\tau_2(i,j)$ & Proof
\\ 
\hline
$(1,0)$ & $89/(9\times 2^{8})$& Lemma \ref{lem:1-0}\\
$(1,1)$ & $17/(3\times 2^7)$ & Lemma \ref{lem:1-1}\\
$(1,2)$ & $23/2^9$ & Lemma \ref{lem:1-2}\\
$(1,3)$ & ${95}/{2^{11}}$ & Lemma \ref{lem:1-3}\\
\hline
\end{tabular}
\vspace{0.2cm}
\caption{$2$-adic densities $\tau_2(i,j)$
}\label{t:2-adic}
\end{table} 
Once combined with our formula for $\tau_{\mathrm{loc},2}$, this therefore concludes the proof of the lemma.
\end{proof}

The numerical value of $\tau_{\mathrm{loc},2}$ is $0.515516\cdots$. 
Let $N(2^k)$ denote  the cardinality on the right hand side of 
\eqref{eq:tau-p} when $p=2$.
Tim Dokchitser has kindly implemented 
a computer algorithm for calculating the ratios $2^{-4k}N(2^k)$ for small values of $k$. Taking 
$k\leq 7$  leads to the numerical value $
0.514905\cdots$, which agrees quite closely with Lemma \ref{lem:tau-2}.

The conclusion of Theorem 
 \ref{c:cor1} is now available. Combining Theorem \ref{t:M-local} and Lemma \ref{lem:M-total}, we conclude that 
$$
\lim_{P\rightarrow \infty} \frac{\Nloc(P)}{N(P)}=
 \frac{\tau_{\mathrm{loc}}}{\tau},
$$
with $\tau_{\mathrm{loc}}$ as in \eqref{eq:final-c}. 
It follows from \eqref{eq:a_p} and \eqref{eq:omega'-p} that
$$
a_p+\frac{2c_p'}{p^2}=
\frac{\left(1-\frac{1}{p}\right)^2}{\left(1+\frac{1}{p}\right)^2} \left(
1+\frac{4}{p}+\frac{8}{p^2}+\frac{12}{p^3}+
\frac{5}{p^4}+\frac{1}{p^6}-\frac{4}{p^7}-
\frac{4}{p^8}
\right).
$$
Applying 
\eqref{eq:tau-inf} and 
Lemma \ref{lem:tau-2}, we deduce that
\begin{align*}
 \frac{\tau_{\mathrm{loc}}}{\tau}
&= 
\frac{\tau_{\mathrm{loc},\infty} \tau_{\mathrm{loc},2}}{17/16} 
\prod_{p\=3 \4} 
\frac{(1+\frac{1}{p})^2 (a_p+\frac{2c_p'}{p^2})}{(1-\frac{1}{p})^2 b_p}
\\
&= 
\frac{
7\times 4751}{
2^{8} \times 3^{2} \times 17} 
\prod_{p\=3 \4} 
\left(\frac{1+\frac{4}{p}+\frac{8}{p^2}+\frac{12}{p^3}+
\frac{5}{p^4}+\frac{1}{p^6}-\frac{4}{p^7}-
\frac{4}{p^8} }{b_p
}\right)\\
&= 
\frac{33257}{39168}
\prod_{p\=3 \4} 
\left( 1-\frac{6-\frac{9}{p^2}+\frac{4}{p^4}}{p^4 b_p}\right),
\end{align*}
with $b_p$ given by \eqref{eq:b_p}.  The latter Euler product converges very rapidly and has numerical value
$0.98186722\cdots,$ whence $\tau_{\mathrm{loc}}/\tau=
0.8336897$, up to 7 decimal places.
This completes the proof of  Theorem  \ref{c:cor1}.

\section{Asymptotics:  $N_{\mathrm{Br}}(P)$}\label{s:lower-upper}

The goal of this section is to establish Theorem \ref{thm2}.
The argument will begin along similar lines to the treatments of $N(P)$ and $\Nloc(P)$, but the  analysis is ultimately more involved. Whereas our earlier work relied upon the  basic estimate in Lemma \ref{lem:basic} for the number of integers in an interval which are coprime to a given integer, the treatment of $\NBr(P)$ will require more sophisticated tools from analytic number theory.  

\subsection{Preliminary analytic tools}

Let $\a=(a_1,a_2), \q=(q_1,q_2)\in \NN^2$. We define an arithmetic function $g_{\a,\q}:\NN\rightarrow \RR$  
multiplicatively on prime powers via
\begin{equation}\label{eq:g}
g_{\a,\q}(p^\nu)=\begin{cases}
1, & \mbox{if $\nu=0$,}\\
\frac{1}{2}(\frac{a_1a_2q_1}{p})-\frac{1}{2}, 
 & \mbox{if $\nu=1$, $p\=3\4$ and $p\nmid a_1a_2q_1q_2$,}\\
0, & \mbox{otherwise.}
\end{cases}
\end{equation}
Let 
$$
\chi_{q_2}(z)=\begin{cases}
1, & \mbox{if $z\in \cB$ with $\gcd(z,  q_2)=1$,}\\
0, & \mbox{otherwise},
\end{cases}
$$
where $\cB$ is given  by \eqref{eq:AB}.
One finds that 
\begin{equation}\label{eq:gz}
g_{\a,\q}(z)=\frac{\mu(z)\chi_{q_2}(z)}{2^{\omega(z)}} \sum_{\ell\mid z} \mu(\ell) \left(\frac{a_1a_2q_1}{\ell}\right).
\end{equation}
We begin with an estimate for the average order of $|g_{\a,\q}|$.

\begin{lemma}\label{lem:g}
Let $\ve>0$. For  $x_1,x_2 , Z>2$, we have 
$$
\sum_{a_1\leq x_1}
\sum_{a_2\leq x_2}
\sum_{\substack{z\leq Z\\
\gcd(z,q_1)=1}} |g_{\a,\q}(z)| \ll  
\frac{x_1x_2Z}{(\log Z)^{3/4}} +\sqrt{x_1x_2}Z^{11/8+\ve},
$$
uniformly in $q_1,q_2$.
\end{lemma}

\begin{proof}
Applying the Burgess bound for character sums, it follows that 
$$
\sum_{a_1\leq x_1}\sum_{a_2\leq x_2}\left(\frac{a_1a_2}{\ell}\right)
\ll  \sqrt{x_1x_2} \ell^{3/8+\ve},
$$
for any integer $\ell>1$.
If $\chi_{q_2}(z)=1$ then \eqref{eq:gz} implies that 
$$
|g_{\a,\q}(z)|=\frac{1}{2^{\omega(z)}} \sum_{\ell\mid z} \mu(\ell) \left(\frac{a_1a_2q_1}{\ell}\right).
$$
Hence 
\begin{align*}
\sum_{a_1\leq x_1}
\sum_{a_2\leq x_2}
\sum_{\substack{z\leq Z\\
\gcd(z,q_1)=1}}
|g_{\a,\q}(z)| 
&=
\sum_{a_1\leq x_1}
\sum_{a_2\leq x_2}
\sum_{\substack{z\leq Z\\
\gcd(z,q_1q_2)=1\\
z\in \cB
}} 
\frac{1}{2^{\omega(z)}} 
+O (
\sqrt{x_1x_2}Z^{11/8+\ve})\\
&\ll 
x_1x_2
\sum_{\substack{z\leq Z\\
z\in \cB
}} 
\frac{1}{2^{\omega(z)}} +
\sqrt{x_1x_2}Z^{11/8+\ve}.
\end{align*}
The remaining sum over $z$ is easily seen to be $O(Z/(\log Z)^{3/4})$, by  an application of the Selberg--Delange method  (see \cite[\S II.5]{ten}, for example).
This concludes the proof.
\end{proof}

At a certain point in our argument it will be useful to approximate the function $g_{\a,\q}$ 
in \eqref{eq:gz} by the simpler arithmetic function
\begin{equation}\label{eq:gz1} 
\hat g_{q_2}(z)=\frac{\mu(z)\chi_{q_2}(z)}{2^{\omega(z)}}.
\end{equation}
We will need an asymptotic formula for the average order of $\hat g_{q_2}(z)h(z)/z$, where $h$ is an arbitrary multiplicative arithmetic function satisfying 
\begin{equation}\label{eq:hyp-h}
h(p)=1+O\left(\frac{1}{p}\right),
\end{equation}
for   primes $p$.  
This is achieved in the following result.

\begin{lemma}\label{lem:z-sum}
Let $\ve>0$ and let $h$ be a multiplicative arithmetic function satisfying \eqref{eq:hyp-h}.
For 
 $Z\geq 2$ and $q_2\in \NN$, we have 
$$
\sum_{\substack{z\leq Z\\ \gcd(z,q_2)=1}} 
\frac{\hat g_{q_2}(z)h(z)}{z}=
\frac{4c(h)\gamma(q_2;h)}{|\Gamma(-1/4)| (\log Z)^{1/4}}\left\{1+O \left(\frac{q_2^\ve}{\log Z}\right)\right\},
$$
where 
$$
\gamma(q_2;h)=
\prod_{\substack{p\=3\4\\ p\mid q_2}}
\left(1-\frac{h(p)}{2p}\right)^{-1}
$$
and $c(h)$ is the convergent Euler product
$$
c(h)=\prod_{p}\left(1-\frac{1}{p}\right)^{-1/4}
\prod_{p\=3\4}
\left(1-\frac{h(p)}{2p}\right).
$$
\end{lemma}

\begin{proof}
We consider  the associated Dirichlet series
\begin{align*}
D(s)=\sum_{\substack{z=1\\ \gcd(z,q_2)=1}}^\infty 
\frac{\hat g_{q_2}(z)h(z)}{z^s}
&=
\prod_{\substack{p\=3\4\\ p\nmid q_2}}
\left(1-\frac{h(p)}{2p^s}\right)\\
&=
\prod_{\substack{p\=3\4}}
\left(1-\frac{h(p)}{2p^s}\right)
\prod_{\substack{p\=3\4\\ p\mid q_2}}
\left(1-\frac{h(p)}{2p^s}\right)^{-1}\\
&=
D_1(s)
\prod_{\substack{p\=3\4\\ p\mid q_2}}
\left(1-\frac{h(p)}{2p^s}\right)^{-1},
\end{align*}
say, for $\Re(s)=\sigma >1$.
Since $h(p)=1+O(1/p)$, 
one sees that 
\begin{align*}
\zeta(s)D_1(s)^4
&=
\prod_{p\not\=3\4}\left(1-\frac{1}{p^s}\right)^{-1}
\prod_{p\=3\4}\left(1-\frac{1}{p^s}\right)^{-1}
\left(1-\frac{2}{p^s}+O\left(\frac{1}{p^{2\sigma}}+\frac{1}{p^{\sigma+1}}\right)\right)\\
&=L(s,\chi) G(s),
\end{align*}
where $L(s,\chi)$ is the Dirichlet $L$-function associated to the real character $\chi$ modulo $4$ and 
$G$ may be continued as a holomorphic function to the half-plane $\sigma>1/2$ and is bounded absolutely in this region. 
Hence it follows that 
$D(s)=\zeta(s)^{-1/4}G_{q_2}(s)$, where 
$$
G_{q_2}(s)=L(s,\chi)^{1/4}G(s)^{1/4}
\prod_{\substack{p\=3\4\\ p\mid {q_2}}}
\left(1-\frac{h(p)}{2p^s}\right)^{-1}.
$$
In particular
$
G_{q_2}(1)=c(h)\gamma(q_2;h),$ in the notation of  the lemma. 

We now invoke the Selberg--Delange method,  
as described in Tenenbaum \cite[\S II.5]{ten}.  This implies that 
$$
\sum_{\substack{z\leq Z\\ \gcd(z,q_2)=1}} 
\hat g_{q_2}(z)h(z)=
\frac{G_{q_2}(1)}{\Gamma(-1/4)}
\frac{Z}{(\log Z)^{5/4}}\left\{1+O \left(\frac{q_2^\ve}{\log Z}\right)\right\}.
$$
Noting that $\Gamma(-1/4)<0$, 
an application of partial summation completes the proof.
\end{proof}

\subsection{Asymptotic formula for $\NBr(P)$}

Our starting point in the proof of Theorem \ref{thm2} is \eqref{eq:tuesday2}, followed by the changes of variables \eqref{eq:change1},  \eqref{eq:change2} and \eqref{eq:change3'}.
In the present situation a further change of variables will be expedient.
Define
\begin{align*}
\ell_1''&=\gcd(a''',(m''n'')^\infty), \quad\ell_2''=\gcd(b''',(m''n'')^\infty), \\
\ell_3''&=\gcd(c''',(m''n'')^\infty), \quad
\ell_4''=\gcd(d''',(m''n'')^\infty).
\end{align*}
We now write 
$$
a'''=\ell_1'' \tilde a , \quad b'''=\ell_2''\tilde b, \quad c'''= \ell_3''\tilde c, \quad d'''=\ell_4'' \tilde d,
$$
for $\tilde a,\tilde b, \tilde c,\tilde d\in \NN$ satisfying  $\gcd(\tilde a\tilde b \tilde c\tilde d, m''n'')=1$.
The union of this constraint with \eqref{eq:change3-coprime-mod} 
and \eqref{eq:change3-coprime'-mod'} is equivalent to
\begin{equation}\label{eq:goat-cheese}
 \gcd(\tilde a\tilde b \tilde c\tilde  d,2mn m''n'')=1
\end{equation} 
and
\begin{equation}\label{eq:goat-cheese'}
\gcd(\tilde a \tilde d,\tilde b\tilde c)=1,
\end{equation}
with 
$\ell_1'',\ldots,\ell_4''\in \NN$ constrained to satisfy  
\begin{equation}\label{eq:sunday-1}
\gcd(\ell_2''\ell_4'',m'')=\gcd(\ell_1''\ell_3'',n'')=1,\quad \gcd(\ell_i'',\ell_j'')=1,
\end{equation}
for  $1\leq i<j\leq 4$. Note that since $\ell_i''\mid (m''n'')^\infty$ it automatically follows that $\ell_i''$ is coprime to $2mn$.
In what follows it will be convenient to redefine

\begin{equation}\label{eq:ABCD}
\begin{split}
A&=
mm''\ell_1\ell_1''  , \,\,\,\,\quad B=mn''\ell_2\ell_2''2^\beta ,\\ 
C&=nm''\ell_3\ell_3''2^\gamma 
, \quad D=nn''\ell_4\ell_4'' 2^\delta.
\end{split}
\end{equation}
Under our various transformations we may now write 
$
\Delta'=m''n'' \Delta'',
$
with 
\begin{equation}\label{eq:Delta''}
\Delta''=\ve_4\ell_1\ell_1''\ell_4\ell_4''2^\delta \tilde a \tilde d -
\ve_2\ve_3\ell_2\ell_2''\ell_3\ell_3''2^{\beta+\gamma} \tilde b \tilde c.
\end{equation}
In this new notation we must proceed to consider the constraints recorded in \S \ref{s:S-global} which are both necessary and sufficient to have $X_{a,b,c,d}(\QQ_v)\neq \emptyset$ for all $v\in \Omega$, but $X_{a,b,c,d}(\QQ)=\emptyset$.
Our key tool is Lemma \ref{lem:1} and the associated calculations.

For the infinite valuation we must have $\bve=(\ve_2,\ve_3,\ve_4)\not\in \{(+,-,-), (-,+,-)\}$ by part (i) of Lemma \ref{lem:X-local} and Lemma \ref{lem:easy}.
For the primes $p\=1\4$ there are no additional constraints arising. 
The situation is more complicated for the primes $p\=3\4$.
According to Lemma \ref{lem:easy-ish} we seek constraints under which there exists a unique choice of $\k\in \{0,1\}^2$ with $k_1+k_2\=v_p(mn)\2$ and 
$W_{(p^{k_1},p^{k_2})}(\QQ_p)\neq \emptyset$, where $W_\e$ is given in Lemma \ref{lem:1}.

Suppose that $p\mid mn$, so that $v_p(mn)=1$. Then Lemma \ref{lem:X-local}(iii) implies that 
\eqref{eq:monday-2} doesn't hold, that 
$p\nmid a''b''c''d''$ and that
$$
\left( \frac{-\ve_2 \ell_{1,p}\ell_{2,p}a''b''}{p}\right)+
\left( \frac{-\ve_3\ve_4 \ell_{3,p}\ell_{4,p}c''d''}{p}\right)\geq 0,
$$
if $v_p(\ell_i)\in 2\ZZ$  and $\ell_{i,p}=p^{-v_p(\ell_i)}\ell_i$, 
for $1\leq i\leq 4$.
The constraint in Lemma \ref{lem:p2}(iii) gives either
$v_p(\ell_i)\in 2\NN$ and $v_p(\ell_{i+2})=1$ for $i\in \{1,2\}$, or 
$v_p(\ell_i)=1$ and $v_p(\ell_{i+2})\in 2\NN$ for $i\in \{1,2\}$.
The constraint in Lemma \ref{lem:p2}(i) translates as
$$
\lan \frac{-\ve_2 \ell_{1}\ell_{2}a''b''}{p}\ran+
\lan\frac{-\ve_3\ve_4 \ell_{3}\ell_{4}c''d''}{p}\ran \leq 0.
$$
It follows from 
\eqref{eq:monday-1'}
that 
$v_p(\ell_1\ell_2)$ and $v_p(\ell_3\ell_4)$ cannot both be odd.
On recalling from \eqref{eq:monday-1} that $\ell_1,\ell_2$ are coprime, we see that $v_p(\ell_1)\in 2\ZZ$ if and only if $v_p(\ell_1\ell_2)\in 2\ZZ$. Similarly, 
$v_p(\ell_3)\in 2\ZZ$ if and only if $v_p(\ell_3\ell_4)\in 2\ZZ$, since $\ell_3,\ell_4$ are coprime.
Combining all these conditions therefore leads to the description that 
$p\nmid a''b''c''d''$, with one of the following
\begin{itemize}
\item
$v_p(\ell_i)\in 2\NN$ and $v_p(\ell_{i+2})=1$ for $i\in \{1,2\}$;
\item 
$v_p(\ell_i)=1$ and $v_p(\ell_{i+2})\in 2\NN$ for $i\in \{1,2\}$;
\item
$v_p(\ell_1\ell_2)\in 2\ZZ$ and $v_p(\ell_3\ell_4)\in 2\ZZ$, with
$$
\left( \frac{-\ve_2 \ell_{1,p}\ell_{2,p}a''b''}{p}\right)+
\left( \frac{-\ve_3\ve_4 \ell_{3,p}\ell_{4,p}c''d''}{p}\right)= 0;
$$

\item
$v_p(\ell_1\ell_2)\in 2\ZZ$ and $v_p(\ell_3\ell_4)\not\in 2\ZZ$, with
$$
\left( \frac{-\ve_2 \ell_{1,p}\ell_{2,p}a''b''}{p}\right)= -1;
$$

\item
$v_p(\ell_1\ell_2)\not\in 2\ZZ$ and $v_p(\ell_3\ell_4)\in 2\ZZ$, with 
$$
\left( \frac{-\ve_3\ve_4 \ell_{3,p}\ell_{4,p}c''d''}{p}\right)= -1.
$$

\end{itemize}
For given parameters $\bve,m,n,m'',n'',\bell,\bell''$, the 
constraints arising from primes $p\=3\4$ for which  $p\mid mn$ force the vector $(\tilde a,\tilde b, \tilde c, \tilde d)$ to lie in one of finitely many congruence classes modulo $p$. Let us denote the set of possible classes by 
$$
\tilde T_{mn}(\bell)\cong \bigoplus_{p\mid 
mn} \tilde T_p(\bell).
$$    
Here $\tilde T_{mn}(\bell)\subseteq (\ZZ/mn\ZZ)^4$ and 
 $\tilde T_{p}(\bell)\subseteq (\ZZ/p\ZZ)^4$ for each prime $p$.
Although 
$\tilde T_{mn}(\bell)$ also depends on $\bve,m'',n''$ and $\bell''$, its cardinality does not, as the following result shows.

\begin{lemma}\label{lem:card-T'}
Let $p\mid mn$. Then we have 
$$
\#\tilde T_{p}(\bell)=
\frac{1}{2^{1-\tau(v_p(\ell_1), \ldots, v_p(\ell_4))}}(p-1)^4,
$$
where for $\bnu\in \ZZ_{\geq 0}^4$ we define
$$
\tau(\bnu)=
\begin{cases}
1, & \mbox{if 
$\nu_i\in 2\NN$ and $\nu_{i+2}=1$ for $i\in \{1,2\}$,} \\
1, & \mbox{if 
$\nu_i=1$ and $\nu_{i+2}\in 2\NN$ for $i\in \{1,2\}$,}\\
0, & \mbox{otherwise.}
\end{cases}
$$
\end{lemma}
\begin{proof}
This follows since  the Legendre symbols 
$( \frac{-\ve_2 \ell_{1,p}\ell_{2,p}a''b''}{p})$ and 
$( \frac{-\ve_3\ve_4 \ell_{3,p}\ell_{4,p}c''d''}{p})$ only take two 
possible values, whatever the parity of  $v_p(\ell_1\ell_2)$ or 
 $v_p(\ell_3\ell_4)$.
\end{proof}

Next we suppose that $p\nmid mnm''n''$, still with $p\=3\4.$  Recall that $\Delta'=m''n''\Delta''$, with $\Delta''$ given by \eqref{eq:Delta''}.  
Then according to Lemma \ref{lem:p2}(ii) we must avoid the the case
$$
\left(\frac{-\ve_2a'b'}{p}\right)=\left(\frac{-\ve_3\ve_4 c'd'}{p}\right)=1,
$$ 
if $p\mid \Delta'$ and $p\nmid a'b'c'd'.$
Note  that 
 $p\mid \Delta'$ and $p\nmid a'b'c'd'$ if and only if 
 $p\mid \Delta''$ and $p\nmid m''n''$.
But for such primes   $(\frac{-\ve_2a'b'}{p})$ and
$(\frac{-\ve_2\ve_3c'd'}{p})$ are equal.  Hence, since $p\=3\4$, we see that the constraint is equivalent to 
$$
\left( \frac{\ve_2 a'b'}{p}\right)=1,
$$
if $p\mid \Delta''$.

Suppose now that $p\nmid mn$ and $p\mid m''n''$.  In particular
$p\mid \Delta'$ and 
Lemma \ref{lem:p2} implies that we must have 
 $[\frac{-\ve_2a'b'}{p}]+[\frac{-\ve_3\ve_4 c'd'}{p}]\leq 0$.
We cannot have both $v_p(a'b')$ and $v_p(c'd')$ being odd. Moreover, we must have 
$\min\{v_p(a'b'), v_p(c'd')\}\leq 1$.
We deduce that precisely one of 
$
v_p(\ell_1''\ell_2'')$ or $
v_p(\ell_3''\ell_4'')$ is odd.
Hence 
we have the pair of conditions
$$
 \left(
 v_p(\ell_1''\ell_2''), v_p(\ell_3''\ell_4'')
 \right)=(0,1)\2, \quad 
\left[\frac{-\ve_2a'b'}{p}\right]=-1,
$$
 or 
 $$
  \left(
 v_p(\ell_1''\ell_2''), v_p(\ell_3''\ell_4'')
 \right)=(1,0)\2, \quad 
\left[\frac{-\ve_3\ve_4 c'd'}{p}\right]=-1.
 $$
 In accordance with this we are now led to introduce the set
\begin{equation}\label{eq:def-Lmn'}
\tilde{L}({m'',n''})=\left\{\bell''\in \NN^4: 
\begin{array}{l}
\mbox{$\ell_i''\mid (m''n'')^\infty$ and \eqref{eq:sunday-1} holds}\\
\mbox{for each $p\mid \ell_1''\cdots \ell_4''$, $\exists!$ $i\in \{1,\ldots,4\}$}\\
\mbox{such that  $v_p(\ell_i'')\=1\2$}
\end{array}
\right\}.
\end{equation}
For given parameters $\bve,m,n,m'',n'',\bell,\bell''$, 
with $\bell''\in \tilde{L}(m'',n'')$, 
the 
constraints arising from primes $p\=3\4$ for which  $p\mid m''n''$ force the vector $(\tilde a,\tilde b, \tilde c, \tilde d)$ to lie in one of finitely many congruence classes modulo $p$. Let us denote the set of possible classes by 
$$
\tilde U_{m''n''}(\bell'')\cong \bigoplus_{p\mid 
m''n''} \tilde U_p(\bell'').
$$    
The analogue of Lemma \ref{lem:card-T'} is the following easy result.

\begin{lemma}\label{lem:card-T''}
Let $p\mid m''n''$. Then we have 
$
\#\tilde U_{p}(\bell'')=
\frac{1}{2}(p-1)^4.
$
\end{lemma}

For any $(a,b,c,d)$ constrained as above, our 
work so far has shown that there exists a unique $\e\in \cB^2$ satisfying 
\eqref{eq:square} for some $f\in \cB$ with $\gcd(f,mn)=1$, such that 
$W_\e(\QQ_v)\neq \emptyset$
for every valuation $v\neq 2$. 
Suppose that
$$
e_i\=\eps_i \4, \quad (i=1,2),
$$
for $\eps_1,\eps_2\in \{\pm 1\}$.
For  the prime $p=2$, it remains to distinguish precisely when $X_{a,b,c,d}(\QQ_2)$ is non-empty but  
$W_{\e}(\QQ_2)$ is empty. This is equivalent to demanding that 
$$
(\eps_1a',\eps_1 \ve_2b',\eps_2 \ve_3c',\eps_2 \ve_4d')\in \Totii\setminus \Toti,
$$
in the notation of \S \ref{s:2adic}.
An inspection of Lemmas~\ref{lem:1-1}--\ref{lem:3-2}
shows that for a given  choice of
$(\beta,\gamma,\delta)\in L_2^{(\iota)}$, the $2$-adic constraints placed on
$(\tilde a, \tilde b, \tilde c, \tilde d)$
take the shape of a union of particular congruence classes modulo $16$. Let us 
denote by $\tilde H_{\beta,\gamma,\delta}$ the set of all 
possible classes modulo $16$ that can arise. 
Then 
the pair of constraints $X_{a,b,c,d}(\QQ_2)\neq \emptyset$ 
and $W_{\e}(\QQ_2)=\emptyset$ are equivalent to demanding that 
\begin{equation}\label{eq:pig2'}
(\tilde a, \tilde b, \tilde c, \tilde d)
 \in \tilde H_{\beta,\gamma,\delta} \bmod{16},
\end{equation}
for given $(\beta,\gamma,\delta)\in L_2^{(\iota)}$.
In particular this constraint implies that 
$\tilde a \tilde b \tilde c \tilde d$ is odd.
Although 
$\tilde H_{\beta,\gamma,\delta}$ depends on numerous parameters, including the residue class of $\e$ modulo  $4$, its cardinality is independent of all of these. We will set\begin{equation}\label{eq:tbr-2}
\sigma_{2}=
\sum_{\iota\in \{0,1\}}
\sum_{(\beta,\gamma,\delta)\in L_2^{(\iota)}}
\frac{\#\tilde H_{\beta,\gamma,\delta}}{2^{16+\beta+\gamma+\delta}}.
\end{equation}
This constant is equal to the density 
of points $(a,b,c,d)$ for which there exist coprime integers $u,v\in \ZZ_2$ for which $Q_1Q_2(u,v)\in \cD$, yet for every choice of coprime $u,v\in \ZZ_2$ one never has both $Q_1(u,v)\in \cD$ and $Q_2(u,v)\in \cD$. 
Using our work in \S \ref{s:2adic}, as an analogue of Lemma \ref{lem:tau-2}, it is in principle possible to calculate a numerical value for 
$\sigma_{2}$. 
Such a calculation would be both lengthy and tedious, and we have 
 chosen not to pursue this here. For our purposes it will be  
 sufficient to note that $\sigma_{2}>0$, since $\Toti\neq \Totii$.

\medskip

We are now ready to return to the expression \eqref{eq:tuesday2} for $\NBr(P)$.
It will be convenient to put
$$
M=16mnm''n''.
$$
Let us define the function
$$
h(\Delta'',\tilde a,\tilde b,M)=\prod_{\substack{p\mid \Delta''\\
p\nmid M\\ p\=3 \4
}} \frac{1}{2}\left\{1+\left(\frac{\ve_2 AB \tilde a\tilde b}{p}\right)\right\},
$$
where $A,B$ are given by \eqref{eq:ABCD}.
One notes that $0\leq h(\Delta'',\tilde a,\tilde b,M)\leq 1$ and 
\begin{equation}\label{eq:invoke}
h(\Delta'',\tilde a,\tilde b,M)=\sum_{z\mid \Delta''} g_{\a,\q}(z),
\end{equation}
in the notation of \eqref{eq:g}, with $\a=(\tilde a, \tilde b)$ and $\q=(\ve_2AB,M)$.
In particular $g_{\a,\q}$ is only supported on positive integers coprime to $\tilde a \tilde b \tilde c \tilde d M$, 
which are built from primes congruent to $3$ modulo $4$.
Summarising our investigation so far, 
our analogue of \eqref{eq:thurs1} is 
\begin{equation}\label{eq:kim}
\begin{split}
\NBr(P)
=~&
\frac{1}{4} \sum_{\iota\in \{0,1\}}
\sum_{\substack{ \bve\in\{\pm\}^3\setminus\{ (+,-,-), (-,+,-)\}}}
\sum_{\substack{m,n\in \cB\\ \gcd(m,n)=1}}
\sum_{\substack{\bell\in L(m,n)\\
\scriptsize{\mbox{$\neg$\eqref{eq:monday-2}}
}}}\\
&\times
\sum_{\substack{m'',n''\in \cA\\ 
\gcd(m'',n'')=1\\
\gcd(m''n'',2mn)=1
}}
\sum_{(\beta,\gamma,\delta)\in L_2^{(\iota)}}
\sum_{\bell''\in \tilde L({m'',n''})}
 \tilde N  + O (P^{2+\ve}),
\end{split}
\end{equation}
where 
$L(m,n)$ (resp.\ $L_2^{(\iota)}$, $\tilde L({m'',n''})$) is given by 
\eqref{eq:def-Lmn} (resp.\  \eqref{eq:L2}, \eqref{eq:def-Lmn'})
and 
$$
\tilde N
=\sum_{
\tilde t=(\tilde a,\tilde b,\tilde c,\tilde d)}
h(\Delta'',\tilde a,\tilde b,M).
$$
The conditions of summation here are restricted to $
\tilde t\in \NN^4$ such that 
\eqref{eq:goat-cheese}, 
\eqref{eq:goat-cheese'} and \eqref{eq:pig2'} hold, with 
$
A\tilde a,~B\tilde b,~C\tilde c,~D\tilde d\leq P
$
and 
$\tilde t
\in  \tilde T_{mn}(\bell) \bmod{mn}\cap
\tilde U_{m''n''}(\bell'') \bmod{m''n''}.
$ 
The definitions 
of $\tilde T_{mn}(\bell)$ and 
$\tilde U_{m''n''}(\bell'')$
 ensure that 
the product $\tilde a\tilde b\tilde c\tilde d$ is coprime to $mnm''n''$. Likewise 
\eqref{eq:pig2'} implies that $\tilde a\tilde b\tilde c\tilde d$  is odd.
Hence 
\eqref{eq:goat-cheese} is redundant.

In our analysis of \eqref{eq:kim} it will frequently be useful to reduce the allowable ranges for the various parameters appearing in the outer summations. Two quantities that will feature heavily in this process are
\begin{equation}\label{eq:T1T2}
T_1=(\log P)^2, \quad T_2=(\log P)^{100}.
\end{equation}
Recall the definitions \eqref{eq:ABCD} of $A,B,C,D$. We proceed to establish the following result.

\begin{lemma}\label{lem:R1R2}
Let $\ve>0$ and $\bde\in \{0,1\}^3$.
For $1\leq R_1\leq R_2$, we have 
\begin{align*}
\sum_{m,n}\sum_{\bell} \sum_{m'',n''} \sum_{\beta,\gamma,\delta}\sum_{\bell''} 
\frac{1}{(AB)^{\delta_1}(CD)^{\delta_2}M^{\delta_3}}\ll \begin{cases}
R_2^{4+\ve}, & \mbox{if $\bde=\ma{0}$,}\\
R_1^{-1+\ve}, & \mbox{if $\bde=(1,1,0)$,}\\
R_2^{\ve}, & \mbox{if $\bde=(1,0,1)$,}
\end{cases}
\end{align*}
uniformly in $R_1,R_2$, where the sum is subject to 
$$
R_1\leq 
\max\{m,n,m'',n'',\ell_i,\ell_i'',2^\beta, 2^\gamma, 2^\delta\}\leq R_2.
$$
\end{lemma}

\begin{proof}
Let us denote the expression that is to be estimate by $\Sigma_{\bde}.$ Suppose first that $\bde=\ma{0}$.
Then there are $\log(R_2+1)^3$ choices for $\beta, \gamma,\delta$. Recall the definition \eqref{eq:phi*} of $\phi^*_\delta$, for any $\delta>0$.  Rankin's trick and \eqref{eq:rain} allow us to deduce that 
\begin{align*}
\Sigma_{\ma{0}}
&\ll
\log(R_2+1)^3
\sum_{m,n}\sum_{\bell} 
\frac{R_2^{\ve/4}}{(\ell_1\cdots \ell_4)^{\ve/16}}
\sum_{m'',n''} \sum_{\bell''} \frac{R_2^{\ve/4}}{(\ell_1''\cdots \ell_4'')^{\ve/16}}\\
&\ll
R_2^{\ve/2}\log(R_2+1)^3
\left(\sum_{k\leq R_2}  \frac{1}{\phi_{\ve/16}^*(k)^4}\right)^4\\
&\ll
R_2^{4+\ve},
\end{align*}
as required.

Let us next consider the case $\bde=(1,1,0)$. Then our argument in \S \ref{s:asymptotic-Nloc}, which was used  
to restrict the size of the parameters appearing in \eqref{eq:thurs1}, easily gives 
$\Sigma_{(1,1,0)}=O(R_1^{-1+\ve})$, again using Rankin's trick and \eqref{eq:rain} to handle the sum over $\bell$ and $\bell''$.
Finally we must consider the case $\bde=(1,0,1).$
In this case,
since $2^\beta,2^\gamma,2^\delta\leq R_2$ and the sum over $\beta$ is absolutely convergent, we obtain 
\begin{align*}
\Sigma_{(1,0,1)}
&\ll
\log (R_2+1)^2
\sum_{m,n,m'',n''}   \frac{1}{m^3n(m''n'')^2}
\sum_{\bell,\bell''}
\frac{1}{\ell_1\ell_2\ell_1''\ell_2'' }.
\end{align*}
Using Rankin's trick again, we see that 
\begin{align*}
\sum_{\bell,\bell''}
\frac{1}{\ell_1\ell_2\ell_1''\ell_2''}
&\leq
\frac{R_2^{\ve/2}}{\phi_{1+\ve/8}^*(mn)^2\phi_{1+\ve/8}^*(m''n'')^2}\ll R_2^{\ve/2},
\end{align*}
since $1/\phi_\delta^*(n)\ll_\delta 1$ for $\delta>1$.
Hence
$\Sigma_{(1,0,1)}
\ll
R_2^{\ve/2}\log (R_2+1)^3\ll  R_2^\ve,
$
as required.
\end{proof}

A trivial upper bound  is given by 
$
\tilde N\ll P^4/(ABCD).
$
Applying Lemma \ref{lem:R1R2} with $\bde=(1,1,0)$ and $R_1=T_1=(\log P)^2$, 
we may henceforth restrict attention to parameters in \eqref{eq:kim} 
for which 
\begin{equation}\label{eq:ogre}
\max\{m,n,m'',n'',\ell_i,\ell_i'',2^\beta, 2^\gamma,2^\delta\}\leq T_1,
\end{equation}
with satisfactory overall error $O (P^4(\log P)^{-2+\ve})$.

We wish to break the sum into congruence classes modulo $M$. 
We shall introduce the set 
$T_M$, say,  of 
$t_0=(a_0,b_0,c_0,d_0) \bmod{M}$ for which 
\begin{equation}\label{eq:cnta}
t_0\in  \tilde T_{mn}(\bell) \bmod{mn}\cap
\tilde U_{m''n''}(\bell'') \bmod{m''n''}\cap 
\tilde H_{\beta,\gamma,\delta} \bmod{16}.
\end{equation}
These two  conditions  imply that 
$\gcd( a_0 b_0c_0d_0,  M)=1.
$
We may therefore write
$$
\tilde N
=\sum_{t_0\in T_M}
\sum_{\substack{
\tilde a \leq P/A, ~\tilde b \leq P/B\\
(\tilde a, \tilde b) \= (a_0,b_0)\bmod{M}
}}
\sum_{\substack{
\tilde c \leq P/C,~ \tilde d \leq P/D\\
\gcd(\tilde a\tilde d,\tilde b\tilde c)=1\\
(\tilde c, \tilde d) \= (c_0,d_0)\bmod{M}
}}
h(\Delta'',\tilde a,\tilde b,M).
$$
Note from \eqref{eq:Delta''} that $\Delta''\ll P^2/M$. 
Invoking \eqref{eq:invoke} we obtain a sum over $z$ which a priori runs over all integers 
up to order $P^2/M$.
Define the $2$-dimensional lattice  
$$
\mathsf{\Lambda}_z=\{(X,Y)\in \ZZ^2:
\ve_4\ell_1\ell_1''\ell_4\ell_4''2^\delta \tilde a Y -
\ve_2\ve_3\ell_2\ell_2''\ell_3\ell_3''2^{\beta+\gamma} \tilde b X\equiv 0 \bmod{z}
\}.
$$
We then have 
\begin{equation}\label{eq:return}
\tilde N
=\sum_{t_0\in T_M}
\sum_{\substack{
\tilde a \leq P/A, ~\tilde b \leq P/B\\
(\tilde a, \tilde b) \= (a_0,b_0)\bmod{M}
}}
\sum_{\substack{z\ll P^2/M\\
\gcd(z,\tilde a \tilde b M)=1
}} g_{\a,\q}(z)\# \cA_z,
\end{equation}
where 
$$
\cA_z=\left\{
(\tilde c, \tilde d)\in \mathsf{\Lambda}_z:
\begin{array}{l}
0<\tilde c \leq P/C,~ 0<\tilde d \leq P/D\\
\gcd(\tilde a\tilde d,\tilde b\tilde c)=\gcd(\tilde c\tilde d,z)=1\\
(\tilde c, \tilde d) \= (c_0,d_0)\bmod{M}
\end{array}
\right\}.
$$
It will be crucial to show that there is a negligible contribution from large values of $z$
in this sum.
This will be achieved using Lemma \ref{lem:g}.

According to Heath-Brown \cite[Lemma 2]{square}
we have 
$$
\sum_{c_0,d_0}
\# \cA_z\ll  \frac{P^2}{CDz}+1.
$$
Recall the definition \eqref{eq:T1T2} of $T_2$.
The overall contribution to $\tilde N$ from $z>P^2/(MT_2)$ is therefore seen to be
\begin{align*}
&\ll
\sum_{a_0,b_0}
\sum_{\substack{
\tilde a \leq P/A, ~\tilde b \leq P/B\\
(\tilde a, \tilde b) \= (a_0,b_0)\bmod{M}
}}
\sum_{\substack{P^2/(MT_2)< z\ll P^2/M}} |g_{\a,\q}(z)|  
\left(\frac{P^2}{CDz}+1\right)\\
&\leq
\sum_{\substack{
\tilde a \leq P/A, ~\tilde b \leq P/B\\
}}
\sum_{\substack{P^2/(MT_2)< z\ll P^2/M}} |g_{\a,\q}(z)|  
\left(\frac{P^2}{CDz}+1\right).
\end{align*}
Note here that $P^2/(MT_2)\gg P^2/(T_1^{4}T_2)\gg  P^{2}/(\log P)^{108}$.
We break the summation over $z$ into $O(\log T_2)$ dyadic intervals $Z/2\leq z\leq Z$. Lemma
 \ref{lem:g} therefore yields  the contribution
\begin{align*}
&\ll  
(\log T_2)
\max_{
P^2/(MT_2)\ll Z \ll P^2/M}\left(\frac{P^2}{CD }+Z\right)
\left(\frac{P^2}{AB(\log Z)^{3/4}}
+\sqrt{\frac{P^2}{AB}}Z^{3/8+\ve}\right)\\
&\ll  
\frac{P^4\log\log P}{AB(\log P)^{3/4}}\left(\frac{1}{CD}+\frac{1}{M}\right)
+P^{4-1/4+\ve}.
\end{align*}
It remains to sum this  over the remaining parameters subject to \eqref{eq:ogre}.
Taking $R_1=1$ and $R_2=T_1$ in Lemma \ref{lem:R1R2}, one arrives at the  overall contribution 
$O(P^4(\log P)^{-3/4+\ve})$, which is satisfactory for Theorem \ref{thm2}.

We may now focus our attention on the contribution from $z\leq P^2/(MT_2)$ in \eqref{eq:return}, which we denote by $\tilde N_1$.  We will need to take care of the coprimality conditions 
$\gcd(\tilde a \tilde d,\tilde b \tilde c)=1$ in $\cA_z$, retaining the condition $\gcd(\tilde c \tilde d, z)=1$.
Recall that modulo $M$ the first condition is already implied by the definition of $T_M$, 
as is the fact that $\gcd(\tilde a\tilde b,M)=1$. In this way we deduce that
\begin{align*}
\tilde N_1
=\sum_{t_0\in T_M}
&
\sum_{\substack{
\tilde a \leq P/A, ~\tilde b \leq P/B\\
\gcd(\tilde a,\tilde b)=1\\
(\tilde a, \tilde b) \= (a_0,b_0)\bmod{M}
}}
\sum_{\substack{z\leq P^2/(MT_2)\\
\gcd(z,\tilde a \tilde b M)=1
}} 
g_{\a,\q}(z)\\
\times &
\sum_{\substack{k_1\leq P\\\gcd(k_1,Mz)=1}} \sum_{\substack{k_2\mid \tilde a\\
k_3\mid \tilde b}} \mu(k_1)\mu(k_2)\mu(k_3)
\# \cA_{z,\k},
\end{align*}
where a change of variables yields
$$
\cA_{z,\k}=\left\{
(X,Y)\in \NN^2:
\begin{array}{l}
([k_1,k_2]X, [k_1,k_3]Y)\in \mathsf{\Lambda}_{z}\\
X \leq P/([k_1,k_2]C),~ Y \leq P/([k_1,k_3]D)\\
(X, Y) \= (X_0,Y_0)\bmod{M}\\
\gcd(XY,z)=1
\end{array}
\right\}.
$$
Here,  if $\overline{[k_1,k_j]}\in \ZZ$ denotes the multiplicative inverse of 
$[k_1,k_j]$ modulo $M$ for $j=2,3$, then 
$X_0=\overline{[k_1,k_2]}c_0$
and $Y_0=\overline{[k_1,k_3]}d_0$.
Finally, we note that $k_2k_3$ is automatically coprime to $Mz$, since $\tilde a\tilde b$ satisfies this property.

We will also need to reduce the ranges of summation for $k_1,k_2,k_3$ in this expression.  
Recall
from \eqref{eq:g} that $|g_{\a,\q}(z)|\leq 1$ for any $z$.
To achieve our goal we invert the summation over $z$ and $(X,Y)$, finding that the overall contribution to $\tilde N_1$ from $k_1>\sqrt{T_2}$ is 
\begin{align*}
\ll \sum_{t_0\in T_M}
\sum_{\substack{
\tilde a \leq P/A, ~\tilde b \leq P/B\\
\gcd(\tilde a,\tilde b)=1\\
(\tilde a, \tilde b) \= (a_0,b_0)\bmod{M}
}}
\sum_{\substack{\sqrt{T_2}<k_1\leq P\\\gcd(k_1,Mz)=1}} \sum_{\substack{k_2\mid \tilde a\\
k_3\mid \tilde b}} 
\,
\sum_{X,Y}
\tau(L(X,Y)),
\end{align*}
where 
$$
L(X,Y)=\ve_4\ell_1\ell_1''\ell_4\ell_4''2^\delta \tilde a [k_1,k_3]Y -
\ve_2\ve_3\ell_2\ell_2''\ell_3\ell_3''2^{\beta+\gamma} \tilde b [k_1,k_2]X.
$$
Recall   \eqref{eq:ogre}.
Summing over $t_0\in T_M$, we may bound this contribution using Lemma \ref{lem:divisor} by
\begin{align*}
&\ll  T_1^\ve
\sum_{\substack{
\tilde a \leq P/A\\\tilde b \leq P/B}}
\sum_{\substack{\sqrt{T_2}<k_1\leq P\\ k_2\mid \tilde a, ~
k_3\mid \tilde b}} 
\left(
\frac{\psi(k_1)\psi(\tilde a \tilde b)^2  P^2\log P}{[k_1,k_2][k_1,k_3]CD}
+\frac{P^{1+\ve}}{k_1}\right)\\
&\ll  \frac{T_1^\ve P^4(\log P)^3}{ABCD \sqrt{T_2}}
+P^{3+\ve}.
\end{align*}
Applying Lemma \ref{lem:R1R2} with $R_1=1$ and $\bde=(1,1,0)$, this 
therefore shows that  $k_1>\sqrt{T_2}$ contribute $O(P^4(\log P)^{3+\ve}/\sqrt{T_2})$ to $\NBr(P)$. 
This is 
satisfactory for Theorem~\ref{thm2}.
Likewise, the same argument shows that parameters with $\max\{k_2,k_3\}>T_1$ make a satisfactory overall contribution to $\NBr(P)$.

Our work so far has therefore shown that we can approximate $\tilde N_1$ by 
\begin{align*}
\tilde N_2
=\sum_{t_0\in T_M}
&
\sum_{\substack{
\tilde a \leq P/A, ~\tilde b \leq P/B\\
\gcd(\tilde a,\tilde b)=1\\
(\tilde a, \tilde b) \= (a_0,b_0)\bmod{M}
}}
\sum_{\substack{z\leq P^2/(MT_2)\\
\gcd(z,\tilde a \tilde b M)=1
}} 
g_{\a,\q}(z)\\
\times &
\sum_{\substack{k_1\leq \sqrt{T_2}\\\gcd(k_1,Mz)=1}} \sum_{\substack{
k_2,k_3\leq T_1\\
k_2\mid \tilde a, ~
k_3\mid \tilde b
}} \mu(k_1)\mu(k_2)\mu(k_3)
\# \cA_{z,\k},
\end{align*}
with acceptable error.
To handle 
$\#\cA_{z,\k}$  we call upon  Lemma \ref{lem:LB}, noting that 
$\#\cA_{z,\k}=N(U,V;\a)$ in \eqref{eq:LB}, with 
$(q,r)=(z,M)$ and 
\begin{align*}
a_1
&=
-
\ve_2\ve_3\ell_2\ell_2''\ell_3\ell_3''2^{\beta+\gamma} \tilde b [k_1,k_2],\\
a_2
&=
\ve_4\ell_1\ell_1''\ell_4\ell_4''2^\delta \tilde a [k_1,k_3],
\end{align*}
and furthermore, 
$$
U=\frac{P}{[k_1,k_2]C}, \quad 
V=\frac{P}{[k_1,k_3]D}.
$$
In particular it is clear that $\gcd(a_1a_2r,q)=1$ and so all the conditions are met for an application of Lemma \ref{lem:LB}. This gives
\begin{equation}\label{eq:LOD}
\left|
\#\cA_{z,\k}-  \frac{\phi(z)P^2}{M^2z^2 [k_1,k_2][k_1,k_3]CD}
\right|\ll  \frac{\tau_3(z)P}{k_1z}+ \tau_3(z)(\log zM)^3
+E,
\end{equation}
where 
$$
E=\sum_{d\mid z}d
\sum_{\substack{0<|m|,|n| \leq Mz/2\\ ma_2-na_1\=0\bmod{d}}}
\frac{1 }{|mn|}.
$$
We need to sum these three error terms over the remaining parameters to check that they ultimately make a negligible contribution. 

The first term on the right of \eqref{eq:LOD} 
is easy to deal with. Lemma \ref{lem:R1R2} with 
$\bde=\ma{0}$ and 
$R_2=T_1$ easily shows that it contributes $O(P^{3+\ve})$ overall.
Recall
that $|g_{\a,\q}(z)|\leq 1$ and 
$\tau_3(z)$ has average order $(\log z)^2$. 
The contribution to $\tilde N_2$ from the second term on the right of \eqref{eq:LOD} 
is  therefore seen to be
\begin{align*}
&\ll 
(\log P)^5
\sum_{t_0\in T_M}
\sum_{\substack{
\tilde a \leq P/A, ~\tilde b \leq P/B\\
\gcd(\tilde a,\tilde b)=1\\
(\tilde a, \tilde b) \= (a_0,b_0)\bmod{M}
}}
\sum_{\substack{k_1\leq \sqrt{T_2}\\\gcd(k_1,M)=1}} \sum_{\substack{k_2\mid \tilde a\\
k_3\mid \tilde b
}}  
\frac{P^2}{MT_2}\\
&\ll  \frac{MP^2(\log P)^5}{\sqrt{T_2}}
\sum_{\substack{
\tilde a \leq P/A, ~\tilde b \leq P/B
}} \tau(\tilde a)\tau(\tilde b)\\
&\ll  \frac{MP^4(\log P)^7}{AB\sqrt{T_2}}.
\end{align*}
We note here that $M/AB\leq T_1$. Hence once summed over the
remaining parameters using Lemma \ref{lem:R1R2} with $\bde=\ma{0}$ and 
$R_2=T_1$, we obtain the satisfactory overall contribution
$O(T_1^{5+\ve}P^4(\log P)^7/\sqrt{T_2})$.

It remains to handle the contribution to $\tilde N_2$ from $E$. 
Let us write $a_1=a_1'\tilde b$ and $a_2=a_2'\tilde a$. Then since $d\mid z$ it follows that 
$\gcd(a_1'a_2',d)=1$.
Taking $|g_{\a,\q}(z)|\leq 1$ and carrying out the sums over $k_1,k_2,k_3$ and $t_0$, this contribution 
is seen to be
\begin{align*}
&\ll 
M^2 T_1^2\sqrt{T_2}
\sum_{\substack{z\leq P^2/(MT_2)\\ \gcd(z,M)=1}}
\sum_{d\mid z}d
\sum_{\substack{0<|m|,|n| \leq Mz/2}}
\frac{1 }{|mn|}
\sum_{\substack{
\tilde a \leq P/A, ~\tilde b \leq P/B\\
\gcd(\tilde a,\tilde b)=1\\
\gcd(\tilde a \tilde b, z)=1\\
ma_2'\tilde a-na_1'\tilde b\=0\bmod{d}}}
1.
\end{align*}
We will need to sort the sum according to the greatest common divisor $h=\gcd(m,n,d)$, writing 
$(m',n',d')=(m,n,d)/h$.
The final condition on $\tilde a,\tilde b$ forces the vectors $(\tilde a,\tilde b)$
in which we are interested to lie on a rank $2$ integer sublattice of determinant $d'$.
Heath-Brown \cite[Lemma 2]{square} therefore shows that the number of 
$(\tilde a,\tilde b)$ is
\begin{align*}
&\ll
\frac{ \gcd(m,n,d)P^2 }{dAB}+1.
\end{align*}
Substituting this into the above 
the second term here is seen to contribute 
\begin{align*}
&\ll M^2 T_1^2\sqrt{T_2} (\log P)^2
\sum_{\substack{z\leq P^2/(MT_2)}}
\sum_{d\mid z}d\\
&\ll \frac{P^4T_1^2(\log P)^2}{T_2^{3/2}},
\end{align*}
to $\tilde N_2$, 
which once combined with Lemma \ref{lem:R1R2} therefore leads to a satisfactory overall contribution. Finally, the first term 
contributes
\begin{align*}
&\ll 
\frac{M^2 T_1^2\sqrt{T_2}P^2  }{AB}
\sum_{\substack{z\leq P^2/(MT_2)}}
\sum_{d\mid z}
\sum_{\substack{0<|m|,|n| \leq Mz/2}}
\frac{\gcd(m,n,d) }{|mn|}\\
&\ll 
\frac{P^4  T_1^3(\log P)^3 }{\sqrt{T_2}}
\end{align*}
to $\tilde N_2$, since $M/AB\leq T_1$.
This too is found to be satisfactory once summed over all the remaining parameters using Lemma \ref{lem:R1R2}.

Having handled the contribution from the error terms in \eqref{eq:LOD}, we are
now free to approximate $\#\cA_{z,\k}$ by the expected main term. Having done so, furthermore, it is convenient to extend the summations over $k_1,k_2,k_3$ to infinity, which we may do with acceptable error using Lemma \ref{lem:R1R2}.
It henceforth suffices to consider the quantity
\begin{align*}
\tilde N_3
=\frac{P^2}{CDM^2}\sum_{t_0\in T_M}
&
\sum_{\substack{
\tilde a \leq P/A, ~\tilde b \leq P/B\\
\gcd(\tilde a,\tilde b)=1\\
(\tilde a, \tilde b) \= (a_0,b_0)\bmod{M}
}}
\sum_{\substack{z\leq P^2/(MT_2)\\
\gcd(z,\tilde a \tilde b M)=1
}} 
\frac{g_{\a,\q}(z)\phi^*(z)}{z}\\
\times &
\sum_{\substack{k_1=1\\\gcd(k_1,Mz)=1}}^\infty \sum_{\substack{k_2\mid \tilde a\\
k_3\mid \tilde b
}} \frac{\mu(k_1)\mu(k_2)\mu(k_3)}
{[k_1,k_2][k_1,k_3]},
\end{align*}
in place of  $\tilde N_2$.
We may carry out the summations over $k_1,k_2,k_3$, finding
that
\begin{equation}\label{eq:shop}
\sum_{\substack{k_1=1\\\gcd(k_1,Mz)=1}}^\infty \sum_{\substack{k_2\mid \tilde a\\
k_3\mid \tilde b
}} \frac{\mu(k_1)\mu(k_2)\mu(k_3)}
{[k_1,k_2][k_1,k_3]}=\frac{6}{\pi^2
\phi_2^*(Mz)\psi(\tilde a \tilde b)},
\end{equation}
in the notation of  \eqref{eq:phi*} and \eqref{eq:psi}.
Hence
\begin{align*}
\tilde N_3
=\frac{6P^2}{\pi^2 
CDM^2 \phi_2^*(M)
}\sum_{t_0\in T_M}
&
\sum_{\substack{
\tilde a \leq P/A, ~\tilde b \leq P/B\\
\gcd(\tilde a,\tilde b)=1\\
(\tilde a, \tilde b) \= (a_0,b_0)\bmod{M}
}}
\frac{1}{\psi(\tilde a \tilde b)}
\sum_{\substack{z\leq P^2/(MT_2)\\
\gcd(z,\tilde a \tilde b M)=1
}} 
\frac{g_{\a,\q}(z)}{z\psi(z)},
\end{align*}
since $\phi^*/\phi_2^*=1/\psi$.

Rather than carrying out the sum over $z$ directly, which intimately 
depends on $\tilde a, \tilde b$, we shall first show that $g_{\a,\q}$ can be approximated by the function $\hat g_{M}$ defined in \eqref{eq:gz1}.
To this end we   exame the sum
$$
\Sigma=
\sum_{\substack{
\tilde a \leq P_1, ~\tilde b \leq P_2\\
\gcd(\tilde a, \tilde b)=1\\
(\tilde a, \tilde b)\equiv (a_0,b_0) \bmod{M}
}}
\frac{1}{\psi(\tilde a \tilde b)}
\sum_{\substack{z\leq Z\\
\gcd(z,\tilde a\tilde bM)=1
}} 
\frac{g_{\a,\q}(z)-\hat g_M(z)}{z\psi(z)},
$$
for given parameters $P_1,P_2\leq P$ and $Z\leq P^2$. Our aim is to show that 
\begin{equation}\label{eq:treen}
\Sigma\ll P^{2-4/11+\ve},
\end{equation}
which leads to a satisfactory overall contribution 
once summed over the remaining parameters, via Lemma \ref{lem:R1R2}
with $\bde=\ma{0}$ and $R_1=T_1$.
To establish the desired bound for $\Sigma$ we deduce from \eqref{eq:gz} and \eqref{eq:gz1} that 
$$
g_{\a,\q}(z)-\hat g_M (z) = \frac{\mu(z)\chi_M(z)}{2^{\omega(z)}} \sum_{\substack{\ell\mid z \\ \ell>1}} \left(\frac{\tilde a\tilde	b q_1}{\ell}\right),
$$
with $q_1=\ve_2AB$.
This allows us to write
$$
\Sigma=
\sum_{\substack{z\leq Z}}
 \frac{\mu(z)\chi_M(z)}{2^{\omega(z)}z\psi(z)} 
 \sum_{\substack{\ell\mid z \\ \ell>1}} 
 \sum_{\substack{
\tilde a \leq P_1, ~\tilde b \leq P_2\\
\gcd(\tilde a \tilde b, z)
=\gcd(\tilde a, \tilde b)=1\\
(\tilde a, \tilde b)\equiv (a_0,b_0) \bmod{M}
}}
\frac{1}{\psi(\tilde a \tilde b)}
\left(\frac{\tilde a\tilde	b q_1}{\ell}\right).
$$
For  $L\leq P^2$ we let $\Sigma_1$ (resp.\ $\Sigma_2$) denote the contribution to $\Sigma$ from $\ell\leq L$ (resp.\ $\ell>L$). 

We begin by estimating $\Sigma_1$.	 For this it will be convenient to write $\psi^{-1}=1*f$,
where $f$ is given multiplicatively at prime powers by 
\begin{equation}\label{eq:1*f}
f(p^\nu)=
\begin{cases}
1, & \mbox{if $\nu=0$,}\\
\frac{-1}{p+1}, & \mbox{if $\nu=1$,}\\
0, & \mbox{otherwise}.
\end{cases}
\end{equation}
For given $\ell,q\in \NN$ with $\ell>1$, this allows us to deduce that
\begin{align*}
\sum_{\substack{n\leq x\\  \gcd(n,q)=1}} \frac{1}{\psi(n)} \left(\frac{n}{\ell}	\right)
&=
\sum_{\substack{e\leq x\\  \gcd(e,q)=1}} f(e)\left(\frac{e}{\ell}	\right)
\sum_{\substack{d\leq x/e\\  \gcd(d,q)=1}} \left(\frac{d}{\ell}	\right)\\
&=
\sum_{\substack{e\leq x\\  \gcd(e,q)=1}} f(e)\left(\frac{e}{\ell}	\right)
\sum_{k\mid q} \mu(k)\left(\frac{k}{\ell}	\right)
\sum_{\substack{d\leq x/{ek}}} \left(\frac{d}{\ell}	\right)\\
&\ll
\tau(q)\sqrt{x} \ell^{3/16+\ve} \sum_{\substack{e\leq x}} \frac{|f(e)|}{\sqrt{e}}\\
&\ll q^\ve \sqrt{x} \ell^{3/16+\ve},
\end{align*}
by the Burgess bound for character sums.
Equipped  with this it is straightforward to conclude that 
$\Sigma_1\ll P^\ve \sqrt{P_1}P_2 L^{3/16}\ll P^{3/2+\ve}L^{3/16}.$

Turning to the contribution from $\ell>L$ and writing $z=\ell z'$, we have
$$
\Sigma_2=
\sum_{\substack{z'\leq Z/L}}
 \frac{\mu(z')\chi_M(z')}{2^{\omega(z')}z'\psi(z')} 
 \sum_{\substack{L< \ell \leq Z/z'\\ \gcd(\ell,z')=1}} 
 \frac{\mu(\ell)\chi_M(\ell)}{2^{\omega(\ell)}\ell\psi(\ell)} 
 \sum_{\substack{
\tilde a, \tilde b}}
\frac{1}{\psi(\tilde a \tilde b)}
\left(\frac{\tilde a\tilde	b q_1}{\ell}\right).
$$
Note that $\ell$ is necessarily odd in this summation. Rearranging terms we are led to an inner sum of the form
$$
 \sum_{\substack{L< \ell \leq Z/z'\\ 2\nmid \ell}} 
\frac{1}{\ell}
 \sum_{
\tilde a \leq P_1}
\alpha_\ell \beta_{\tilde a}
\left(\frac{\tilde a}{\ell}\right),
$$
for suitable real numbers $\alpha_\ell, \beta_{\tilde a}$ with modulus at most $1$.
Breaking the $\ell$ sum into dyadic intervals and applying Heath-Brown's  large sieve for real characters 
\cite[Cor.~4]{real}, we obtain
\begin{align*}
\Sigma_2&\ll
P_2P^{\ve/2}
\sum_{\substack{z'\leq Z/L}}
 \frac{1}{z'}
\max_{L<L'\leq P^2}\left\{ P_1^{1/2}+L'^{-1/2}P_1\right\}\\
&\ll P^{3/2+\ve} +L^{-1/2}P^{2+\ve}.
\end{align*}
Taking $L=P^{8/11}$ and combining this with our estimate for $\Sigma_1$, we therefore arrive at the desired bound for $\Sigma$ in \eqref{eq:treen}.

Having shown that we may safely approximate $g_{\a,\q}$ by $\hat g_M$ in $\tilde N_3$, we proceed
to swap the sums over $(\tilde a, \tilde b)$ and $z$. In this way we are led to consider the sum 
$$
S=
\sum_{\substack{
\tilde a \leq P/A, ~\tilde b \leq P/B\\
\gcd(\tilde a,\tilde b)=1\\
\gcd(\tilde a \tilde b,z)=1\\
(\tilde a, \tilde b) \= (a_0,b_0)\bmod{M}
}}
\frac{1}{\psi(\tilde a \tilde b)}.
$$
The asymptotic evaluation of $S$ is the object of the following result.

\begin{lemma}\label{lem:calculate-S}
Let 
$$
S_1=
\frac{c_1\phi^*(z)^2 P^2 }{ABM^2}
\prod_{p\mid Mz}\frac{(1+\frac{1}{p})}{(1-\frac{1}{p})}
\left(1+\frac{2}{p}-\frac{1}{p^2}\right)^{-1},
$$
where
$$
c_1=
\prod_{p}\frac{(1-\frac{1}{p})}{(1+\frac{1}{p})}
\left(1+\frac{2}{p}-\frac{1}{p^2}\right).
$$
Then we have 
$$
\sum_{m,n}\sum_{\bell} \sum_{m'',n''} \sum_{\beta,\gamma,\delta}\sum_{\bell''} 
\frac{P^2}{CDM^2}
\sum_{t_0}\sum_{z} \frac{1}{z\psi(z)}
|S-S_1| \ll  \frac{P^4(\log P)^{1+\ve}}{T_1},
$$
for any $\ve>0$, 
where the outer sums are over parameters satisfying \eqref{eq:ogre}.
\end{lemma}

\begin{proof}
Since $\tilde{a}$ and $\tilde b$ are coprime, we have $\psi(\tilde a\tilde b)= 
\psi(\tilde a)\psi(\tilde b)$. Writing 
$\psi^{-1}=1*f$, where $f$ is given by \eqref{eq:1*f}, we obtain 
\begin{align*}
S&=\sum_{\substack{e_1,e_2=1 \\ \gcd(e_1e_2,Mz)=1\\
\gcd(e_1,e_2)=1
}}^\infty f(e_1)f(e_2)
\sum_{\substack{
\tilde a \leq P/A, ~\tilde b \leq P/B\\
\gcd(\tilde a,\tilde b)=1\\
\gcd(\tilde a \tilde b,z)=1\\
(\tilde a, \tilde b) \= (a_0,b_0)\bmod{M}\\
e_1\mid \tilde a, ~e_2 \mid \tilde b
}}
1.
\end{align*}
We recall that  
$\gcd( a_0  b_0 , M)=1$.
Making the obvious change of variables, and 
and using the M\"obius function to detect coprimality conditions, 
the inner sum becomes
\begin{equation}\label{eq:cough}
\sum_{\substack{
x \leq P/(e_1A), ~y \leq P/(e_2B)\\
\gcd(e_1x,e_2y)=1\\
\gcd(xy,z)=1\\
(x, y) \= (\bar{e_1}a_0,\bar{e_2}b_0)\bmod{M}
}}
1=
\sum_{\substack{k=1\\ \gcd(k,Mz)=1}}^\infty\sum_{\substack{k_1\mid e_1\\ k_2\mid e_2}} 
\mu(k)\mu(k_1)\mu(k_2) C_{k,k_1,k_2},
\end{equation}
where
$C_{k,k_1,k_2}$ denotes the number of positive integers 
$x\leq P/(e_1A)$ and $y \leq P/(e_2B)$ such that 
$$
\gcd(xy,z)=1, \quad  [k,k_2]\mid x, \quad [k,k_1]\mid y,
$$
and 
$(x, y) \= (\bar{e_1}a_0,\bar{e_2}b_0)\bmod{M}$.
Note here that  $k_1$ and $k_2$ are automatically both coprime to $Mz$ since $e_1$ and $e_2$ are.
Making a further change of variables, 
$C_{k,k_1,k_2}$ is equal to the number of positive integers 
$x'\leq P/([k,k_2]e_1A)$ and $y' \leq P/([k,k_1]e_2B)$ such that 
$\gcd(x'y',z)=1$ and 
$(x', y') \= (x_0,y_0)\bmod{M}$, for integers $x_0=\overline{[k,k_2]e_1}a_0$ and $y_0=
\overline{[k,k_1]e_2}b_0$.

Taking a trivial upper bound for $C_{k,k_1,k_2}$, we find that
\begin{align*}
\sum_{a_0,b_0 \bmod{M}}C_{k,k_1,k_2}
&\ll \frac{P^2}{[k,k_1][k,k_2]e_1e_2AB}.
\end{align*}
Note that 
$$
\sum_{e\geq E} \frac{|f(e)|}{e}\ll \frac{1}{E},
$$
for any $E\geq 1$. 
In the usual way, combining  these estimates with Lemma \ref{lem:R1R2} shows that the overall contribution from $\max\{k,e_1,e_2\}>T_1$ is 
$O (P^4T_1^{-1+\ve})$, for any $\ve>0$.
Hence we may assume that $\max\{k,e_1,e_2\}\leq T_1$ in \eqref{eq:cough}.

We estimate $C_{k,k_1,k_2}$ using Lemma \ref{lem:basic} twice, with $q_1=z$ and $q_2=M$.
This is valid provided that $P/([k,k_2]e_1A)\gg  T_1Mz^\ve$ and 
$P/([k,k_1]e_2B)\gg  T_1Mz^\ve$. Recall from \eqref{eq:ABCD} and \eqref{eq:ogre} that 
$A,B\leq T_1^5$ and $M\ll T_1^4$.
It therefore suffices to have $P\gg  T_1^{13}z^\ve$, which obviously holds. 
Hence we may conclude that 
$$
C_{k,k_1,k_2}=\frac{\phi^*(z)^2 P^2}
{[k,k_1][k,k_2]e_1e_2ABM^2} \left(1+O \left(\frac{1}{T_1^{1-\ve}}\right)\right).
$$
The overall contribution from the error term here, once 
substituted into 
 \eqref{eq:cough} and summed over the remaining parameters, is easily found to be 
 $O(P^4(\log P)T_1^{-1+\ve})$.
Similarly, the shape of the main term allows us to extend 
the summations over $k$ and $e_1,e_2$ to infinity with acceptable error.
We are therefore led to approximate $S$ by
\begin{align*}
S_1=\frac{\phi^*(z)^2 P^2 }{ABM^2}
\sum_{\substack{e_1,e_2=1\\ \gcd(e_1e_2,Mz)=1\\
\gcd(e_1,e_2)=1
}}^\infty \frac{f(e_1)f(e_2)}{e_1e_2}
\sum_{\substack{k=1\\ \gcd(k,Mz)=1}}^\infty\sum_{\substack{k_1\mid e_1\\ k_2\mid e_2}} 
\frac{\mu(k)\mu(k_1)\mu(k_2) }
{[k,k_1][k,k_2]}.
\end{align*}
The inner sum over $k,k_1,k_2$ here has already been calculated in \eqref{eq:shop}.
Finally, by calculating Euler factors, 
one sees that the sums over $e_1,e_2,k,k_1,k_2$ evaluate to
$$
\prod_{p\nmid Mz}\left(1-\frac{1}{p^2}\right)
\left(1-\frac{2}{(p+1)^2}\right)=
\prod_{p\nmid Mz}\frac{(1-\frac{1}{p})}{(1+\frac{1}{p})}
\left(1+\frac{2}{p}-\frac{1}{p^2}\right).
$$
It now follows that $S_1$ agrees with expression in the statement of the lemma.
\end{proof}

Armed with Lemma \ref{lem:calculate-S}, we now return to our expression for $\tilde N_3$. 
Let 
$$
h(z)=\frac{\phi^*(z)^2}{\psi(z)}
\prod_{p\mid z}\frac{(1+\frac{1}{p})}{(1-\frac{1}{p})}
\left(1+\frac{2}{p}-\frac{1}{p^2}\right)^{-1}=
\prod_{p\mid z}\left(1-\frac{1}{p}\right)
\left(1+\frac{2}{p}-\frac{1}{p^2}\right)^{-1}.
$$
Furthermore, we observe that 
\begin{align*}
\frac{1}{\phi_2^*(M)}
\prod_{p\mid M}\frac{(1+\frac{1}{p})}{(1-\frac{1}{p})}
\left(1+\frac{2}{p}-\frac{1}{p^2}\right)^{-1}
&=
\prod_{p\mid M}
\left(1-\frac{1}{p}\right)^{-2}
\left(1+\frac{2}{p}-\frac{1}{p^2}\right)^{-1}\\
&=\gamma_1(M),
\end{align*}
say.
Then, in place of $\tilde N_3$, our work so far has shown that we may work with 
\begin{align*}
\tilde N_4
=\frac{6c_1P^4}{\pi^2 
ABCD}
\times \frac{\gamma_1(M)\#T_M}{M^4}
\sum_{\substack{z\leq P^2/(MT_2)\\ \gcd(z,M)=1}} 
\frac{\hat g_M(z)h(z)}{z},
\end{align*}
where $c_1$ is given in the statement of Lemma \ref{lem:calculate-S}.

We have now arrived at the final stages of the argument, where it is necessary to evaluate the sum over $z$ asymptotically. This is achieved by taking 
$Z$ of order $P^2/MT_2$
in Lemma \ref{lem:z-sum}.
In particular it is clear that $\log Z\gg  \log P$. 
Let us write $c_2=c_2(h)$ and $\gamma_2(M)=\gamma(M;h)$.
Since \eqref{eq:hyp-h} clearly holds, we 
deduce that 
\begin{align*}
\tilde N_4
=\frac{24c_1c_2P^4}{2^{1/4}\pi^2 |\Gamma(-1/4)|
ABCD(\log P)^{1/4}}
\times \frac{\gamma_1(M)\gamma_2(M)\#T_M}{M^4}
\left\{1+O \left(\frac{1}{(\log P)^{1-\ve}}\right)\right\}.
\end{align*}
One notes that $\gamma_1(M)\gamma_2(M)\ll \log T_1$ and $\#T_M\ll M^4$. Hence the overall contribution from the error term here is satisfactory for Theorem \ref{thm2}, once summed over the remaining parameters satisfying \eqref{eq:ogre} using Lemma \ref{lem:R1R2}.
Finally, using this result a final time we can 
eliminate  \eqref{eq:ogre} and 
so extend the summation over the outer parameters to infinity, with acceptable error. Returning to \eqref{eq:kim}, we have therefore shown that the statement of Theorem~\ref{thm2} holds, with 
\begin{align*}
\tau_{\mathrm{Br}}
=~&
\frac{6c_1c_2}{2^{1/4}\pi^2|\Gamma(-1/4)|}
\sum_{\iota\in \{0,1\}}
\sum_{\substack{ \bve\in\{\pm\}^3\setminus\{ (+,-,-), (-,+,-)\}}}
\sum_{\substack{m,n\in \cB\\ \gcd(m,n)=1}}
\sum_{\substack{\bell\in L(m,n)\\
\scriptsize{\mbox{$\neg$\eqref{eq:monday-2}} 
}}}\\
&\times
\sum_{\substack{m'',n''\in \cA\\ 
\gcd(m'',n'')=1\\
\gcd(m''n'',2mn)=1
}}
\sum_{(\beta,\gamma,\delta)\in L_2^{(\iota)}}
\sum_{\bell''\in \tilde L({m'',n''})}
\frac{\gamma_1(M)\gamma_2(M)\#T_M}{
ABCDM^4}.
\end{align*}

\subsection{Evaluation of $\tau_{\mathrm{Br}}$}
We proceed to study this constant, 
with the goal of showing that it is positive. 
The set $T_M$ was defined to be the 
$t_0 \bmod{M}$ for which \eqref{eq:cnta} holds.
Since $16$, $mn$ and  $m''n''$ are pairwise coprime, we clearly have 
$$
\#T_M=
\#\tilde T_{mn}(\bell) \times 
\#\tilde U_{m''n''}(\bell'')
\times
\# \tilde H_{\beta,\gamma,\delta}.
$$
Furthermore, Lemma  \ref{lem:card-T''} implies that 
$$
\#\tilde U_{m''n''}(\bell'') =\frac{\phi(m''n'')^4}{2^{\omega(m''n'')}},
$$
whereas $\#\tilde T_{mn}(\bell) $ is given by 
Lemma  \ref{lem:card-T'}.

Recall 
the expression \eqref{eq:tbr-2} for 
the non-zero constant $\sigma_{2}$.
Returning to our formula for the constant, in view of 
the definition \eqref{eq:ABCD} of $A,B,C,D$, 
we sum over the six possible choices for $\bve$ to get
\begin{align*}
\tau_{\mathrm{Br}}
=~&
\frac{36c_1c_2 \sigma_{2}}{
2^{1/4}\pi^2|\Gamma(-1/4)|}
\sum_{\substack{m,n\in \cB\\ \gcd(m,n)=1}}
\sum_{\substack{\bell\in L(m,n)\\
\scriptsize{\mbox{$\neg$\eqref{eq:monday-2}}
}}}
\sum_{\substack{m'',n''\in \cA\\ 
\gcd(m'',n'')=1\\
\gcd(m''n'',2mn)=1
}}
\\
&\times
\sum_{\bell''\in \tilde L({m'',n''})}
\frac{\gamma_1(M)\gamma_2(M) \phi^*(m''n'')^4 }{2^{\omega(m''n'')} 
m^2n^2{m''}^2{n''}^2\ell_1\cdots\ell_4
\ell_1''\cdots \ell_4'' 
}\cdot \frac{\#\tilde T_{mn}(\bell)}{(mn)^4}.
\end{align*}

Next we undertake the summation over $\bell''$, where $\tilde L({m'',n''})$ is given  by \eqref{eq:def-Lmn'}.  A simple calculation reveals that 
$$
\sum_{\bell''\in \tilde L({m'',n''})}
\frac{1}{
\ell_1''\cdots \ell_4'' 
}=\frac{2^{\omega(m''n'')}}{m''n'' \phi_2^*(m''n'')}.
$$
Likewise, with reference to the definition \eqref{eq:def-Lmn} of $L(m,n)$ and 
Lemma  \ref{lem:card-T'}, 
it is straightforward  to check that the sum over $\bell$ evaluates to 
\begin{align*}
\sum_{\substack{\bell\in L(m,n)\\
\scriptsize{\mbox{$\neg$\eqref{eq:monday-2}}
}}}
\frac{1}{
\ell_1\cdots \ell_4
}\cdot
\frac{\#\tilde T_{mn}(\bell)}{(mn)^4}
&=\phi^*(mn)^4
\prod_{p\mid mn} \sum_{\bnu\in E(v_p(m),v_p(n))} \frac{2^{\tau(\bnu)-1}}{p^{\nu_1+\nu_2+\nu_3+\nu_4}}\\
&=\phi^*(mn)^4
\gamma_3(mn),
\end{align*}
where $E(i,j)$ is given by \eqref{eq:lab-V} and 
$$
\gamma_3(mn)=
\prod_{p\mid mn}\left(1-\frac{1}{p}\right)^{-1}\left(1+\frac{1}{p}+
\frac{1}{p^2-1}\left\{
1+\frac{3}{p}-\frac{1}{p^2}-\frac{1}{p^3}
\right\}
\right).
$$
Noting that $\gamma_1(16)=\frac{16}{7}$ and $\gamma_2(16)=1$, 
we may therefore conclude that
\begin{align*}
\tau_{\mathrm{Br}}
=~&
\frac{576c_1c_2 \sigma_{2}}{
2^{1/4}7\pi^2|\Gamma(-1/4)|}
\sum_{\substack{m,n\in \cB\\ \gcd(m,n)=1}}
\frac{\gamma_1\gamma_2\gamma_3(mn) \phi^*(mn)^4 }{
m^2n^2}\\
&\times
\sum_{\substack{m'',n''\in \cA\\ 
\gcd(m'',n'')=1\\
\gcd(m''n'',2mn)=1
}}
\frac{ \gamma_1\gamma_2(m''n'')\phi^*(m''n'')^4 }{
{m''}^3{n''}^3 \phi_2^*(m''n'')}.
\end{align*}
These sums over $m,n,m'',n''$ are absolutely convergent and positive, as one checks by considering the associated Euler products. This therefore shows that $\tau_\mathrm{Br}>0$, as required to complete the proof of Theorem \ref{thm2}.

\end{document}